\documentclass[oneside,final,11pt]{amsart}
\usepackage[T1]{fontenc}
\usepackage[latin9]{inputenc}
\usepackage{geometry}
\usepackage{graphicx}
\usepackage{soul}
\usepackage{xcolor}
\usepackage{amssymb}

\newtheorem{theorem}{Theorem}[section]
\newtheorem{lemma}[theorem]{Lemma}
\newtheorem{proposition}[theorem]{Proposition}

{ \theoremstyle{definition}
\newtheorem{definition}[theorem]{Definition}}
{ \theoremstyle{definition}
\newtheorem{conjecture}[theorem]{Conjecture}}
{ \theoremstyle{remark}
\newtheorem{remark}[theorem]{Remark}}
\usepackage{placeins}
\usepackage{cite}
\usepackage{caption}
\usepackage{enumerate}
\usepackage{enumitem}
\usepackage{bmpsize}

\title{KPZ and Airy limits of Hall-Littlewood random plane partitions}
\date{\today}
\author{Evgeni Dimitrov}

\begin{document}

\maketitle 

\begin{abstract}
In this paper we consider a probability distribution $\mathbb{P}^{q,t}_{HL}$ on plane partitions, which arises as a one-parameter generalization of the $q^{volume}$ measure in \cite{Ok}. This generalization is closely related to the classical multivariate Hall-Littlewood polynomials, and it was first introduced by Vuleti{\'c} in \cite{Vul}.

We prove that as the plane partitions become large ($q$ goes to $1$, while the Hall-Littlewood parameter $t$ is fixed), the scaled bottom slice of the random plane partition converges to a deterministic limit shape, and that one-point fluctuations around the limit shape are asymptotically given by the GUE Tracy-Widom distribution. On the other hand, if $t$ simultaneously converges to its own critical value of $1$, the fluctuations instead converge to the one-dimensional Kardar-Parisi-Zhang (KPZ) equation with the so-called narrow wedge initial data.

The algebraic part of our arguments is closely related to the formalism of Macdonald processes \cite{BorCor}. The analytic part consists of detailed asymptotic analysis of the arising Fredholm determinants.
\end{abstract}

\tableofcontents

\section{Introduction and main results}\label{introduction}
The main results of this paper are contained in Section \ref{mainResults}. The two sections below give background for and define the main object we study, which is a certain $2$-parameter family of probability distributions on plane partitions. 

%
\subsection{Preface}\label{intro0} \hspace{2mm}\\

Roughly $30$ years ago Kardar, Parisi and Zhang \cite{KPZ} studied the time evolution of random growing interfaces and proposed the following stochastic partial differential equation for the height function $\mathcal{H}(T,X) \in \mathbb{R}$ ($T \in \mathbb{R}_+$ is time and $X \in \mathbb{R}$ is space)
\begin{equation}\label{KPZE}
 \partial_{T} \mathcal{H}(T,X) = \frac{1}{2} \partial_X^2 \mathcal{H}(T,X) + \frac{1}{2}( \partial_X \mathcal{H}(T,X))^2 + \dot{\mathcal{W}}(T,X).
\end{equation}
The randomness $\dot{\mathcal{W}}$ models the deposition mechanism and is taken to be space-time Gaussian white noise, so that formally $\mathbb{E}\left[\dot{\mathcal{W}}(T,X) \dot{\mathcal{W}}(S,Y) \right] = \delta(T-S)\delta(X-Y)$. Drawing upon the work of Forster, Nelson and Stephen \cite{FNS}, KPZ predicted that for large time $T$, the height function $\mathcal{H}(T,X)$ exhibits fluctuations around its mean of order $T^{1/3}$ and spatial correlation length of order $T^{2/3}$.
The critical exponents $1/3$ and $2/3$ are believed to be universal for a large class of growth models, which has become known as the KPZ universality class. A growth model is believed to belong to the KPZ universality class if it satisfies the following (imprecise) conditions:
\begin{enumerate}[label = \arabic{enumi}., leftmargin=1cm]
\item there is a smoothing mechanism, disallowing deep holes and high peaks (in (\ref{KPZE}) this is reflected by the Laplacian $ \frac{1}{2} \partial_X^2 \mathcal{H}(T,X) $);
\item growth is slope-dependent, ensuring lateral growth of interfaces (captured by $\frac{1}{2}( \partial_X \mathcal{H}(T,X))^2$ in (\ref{KPZE}));
\item randomness is driven by short space-time correlated noise (the term $\dot{\mathcal{W}}(T,X)$ in (\ref{KPZE})).
\end{enumerate}
For additional background the reader is referred to \cite{CU2,Q,SS}.
 
It took a quarter of a century to prove that the KPZ equation was in the KPZ universality class itself (by demonstrating the $1/3$ and $2/3$ exponents) \cite{CorQ, BQS, BorCor, BCF, CorHamK, SS2} and it is important to note the contribution of integrable (or exactly solvable) models for this success. Historically, methods for analyzing exactly solvable discretizations of the KPZ equation such as the (partially) asymmetric simple exclusion process (ASEP), the $q$-deformed totally asymmetric simple exclusion process ($q$-TASEP), or the O'Connell-Yor semi-discrete directed random polymers were developed first (see the review \cite{CM} and references therein). Consequently, these stochastic processes were shown to converge (under special {\em weakly asymmetric} or {\em weak noise} scaling) to the KPZ equation. The exact formulas available for the processes allowed one to conclude that they belong to the KPZ universality class and after appropriate scaling the same could be concluded for the solution to the KPZ equation. We remark that the developed methods allow one to analyze the KPZ equation only within a certain class of initial conditions.\\

Since their discovery many of the discrete stochastic processes have become interesting in their own right as fundamental models for interacting particle systems, directed polymers in random media and parabolic Anderson models. These processes typically come with some enhanced algebraic structure, which makes them more amenable to detailed analysis and hence provides the most complete access to various phenomena such as phase transition, intermittency, scaling exponents, and fluctuation statistics. One particular algebraic framework, which has enjoyed substantial interest and success in analyzing various probabilistic systems in the last several years, is the theory of Macdonald processes \cite{BorCor}. Macdonald processes are defined in terms of a remarkable class of symmetric functions, called Macdonald symmetric functions, which are parametrized by two numbers $(q,t)$ - see \cite{Mac}. By leveraging some of their algebraic properties, Macdonald processes have proved useful in solving a number of problems in probability theory, including computing exact Fredholm determinant formulas and associated asymptotics for one-point marginal distributions of the O'Connel-Yor semi-discrete directed polymer \cite{BorCor, BCF}; log-gamma discrete directed polymer \cite{BorCor, BCR};  KPZ/stochastic heat equation \cite{BCF}; $q$-TASEP \cite{Bar15, BorCor, BorCor2, BCS} and $q$-PushASEP \cite{BP, CP}.

There exists a natural family of operators, called the Macdonald difference operators, which are diagonalized by the Macdonald symmetric functions. Using these operators one can express the expectation of a large class of observables for Macdonald processes in terms of contour-integrals. The approach of studying Macdonald processes through these observables was initiated in \cite{BorCor}, where it was used to analyze the $q$-Whittaker process (a special case of Macdonald processes, corresponding to setting $t = 0$). This approach has subsequently been generalized and put on much more abstract footing in \cite{BorGor}, where it was suggested that it can be used to study various other special cases of Macdonald processes, coming from degenerations of Macdonald to other symmetric functions. 

The purpose of this paper is to use the approach of Macdonald difference operators to study a different degeneration of the Macdonald process, called the Hall-Littlewood process, which corresponds to setting $q = 0$. Our motivation for studying the Hall-Littlewood process is that it arises naturally in a problem of random plane partitions. The distribution on plane partitions we consider, called $\mathbb{P}^{r,t}_{HL}$ in the text and defined in the next section, was first considered by Vuleti{\'c} in \cite{Vul}, where she discovered a generalization of the famous MacMahon formula and identified an important geometric structure of the measure. The measure $\mathbb{P}^{r,t}_{HL}$ is a one-parameter generalization of the usual $r^{vol}$ measure on plane partitions, which is recovered if one sets $t = 0$ (the volume parameter is usually denoted by $q$ in the literature, and also in the abstract above, but we reserve this letter for the $q$ in the Macdonald polynomials and use $r$ instead for the remainder of the text). 

The algebraic part of our arguments consists of developing a framework for the Macdonald difference operators in the Hall-Littlewood case. Although our discussion is parallel to the one for the $q$-Whittaker case in \cite{BorCor}, we remark that there are several technical modifications that need to be made, which require us to redo most of the work there. In the Hall-Littlewood setting the operators approach gives access to a single observable and we find a Fredholm determinant formula for its $t$-Laplace transform. This result is given in Proposition \ref{finlength} and we believe it to be of separate interest as it can be applied to generic Hall-Littlewood measures and its Fredholm determinant form makes it suitable for asymptotic analysis. For the particular model we consider, the observable is insufficient to study the $3$-dimensional diagram; however, we are able to use it to analyze the one-point marginal distribution of the bottom part of the diagram.

The main results of the paper (Theorems \ref{TW} and \ref{TCDRP} below) describe the asymptotic distribution of the bottom slice of a plane partition, distributed according to $\mathbb{P}^{r,t}_{HL}$, in two limiting regimes: when $r \rightarrow 1^-$, $t\in (0,1)$ - fixed and when $r,t \rightarrow 1^-$ in some critical fashion. In both cases one observes the same limit shape, while the fluctuations in the first limiting regime converge to the Tracy-Widom GUE distribution \cite{TWPaper}, and to the distribution of the Hopf-Cole solution to the KPZ equation with narrow wedge initial data \cite{CorQ, Ber} in the second one. The latter results suggest that our model belongs to the KPZ universality class, although some care needs to be taken. Typically, models belonging to the KPZ universality class are characterized by some dynamics (interacting particle systems, growing interfaces, random polymers etc.), so that the system evolves with time. In sharp contrast, the model we consider is {\em stationary}, i.e. there is no notion of time.

In order to prove our main results we specialize the general formula for the $t$-Laplace transform from Proposition \ref{finlength} to the particular measure we consider. Subsequently, we find two different representations of this formula that are suitable for the two limiting regimes. When $t \in (0,1)$ is fixed and $r \rightarrow 1^-$ the $t$-Laplace transform converges to an indicator function and our Fredholm determinant formula converges to the CDF of the Tracy-Widom GUE distribution. When both $r,t \rightarrow 1^-$ the $t$-Laplace transform converges to the usual Laplace transform and our Fredholm determinant formula converges to the Laplace transform of the partition function of the continuous directed random polymer \cite{AKQ, Cal}. The main difficulties in establishing the above convergence results are finding suitable contours for our Fredholm determinants and representations for the integrands. We reduce the convergence results to verifying certain exponential bounds for the integrands, which are obtained through a careful analysis on the (specially) constructed contours. This detailed asymptotic analysis of the arising Fredholm determinants forms the analytic part of our arguments.\\

Even though our methods do not allow us to verify it directly, we believe that if $t \in[0,1)$ is fixed one still obtains a $3$-dimensional limit shape in the limit $r \rightarrow 1^-$. That limit shape (if it exists) necessarily depends on $t$ as the volume of the (rescaled) diagram satisfies a law of large numbers and converges to an explicit function of $t$ (see Section \ref{Section1.4} for details). This function decreases to $0$ as $t$ increases from $0$ to $1$, which suggests that the measure $\mathbb{P}^{r,t}_{HL}$ concentrates on diagrams of smaller size as $t$ increases. In sharp contrast, the result of Theorem \ref{TW} suggests that while the volume of the plane partition decreases in $t$ the bottom slice asymptotically looks the same. The latter is quite surprising and we are not aware of this phenomenon occurring in other random tiling/plane partition models. As can be observed in simulations what happens is that the $3$-dimensional limit shape becomes flatter and concentrates on diagrams, which have a fixed base but are quite thin. We refer to Section \ref{Section1.4} for further details.

Another interesting feature of our model is that it is rich enough to produce the Tracy-Widom GUE and KPZ statistics under different scaling limits. The Tracy-Widom GUE distribution and, more generally, the Airy process \cite{Spohn} have been shown to arise as universal scaling limits of a wide variety of probabilistic systems including random matrix theory, stochastic growth processes, interacting particle systems, directed polymers in random media, random tilings and random plane partitions (see \cite{FerA} and \cite{QR} and references therein). It is believed that the Airy process also arises as the large time limit of the properly translated and scaled solution to the KPZ equation with narrow wedge initial data. The latter statement has been verified at the level of one point statistics for example in \cite{CorQ}; however, there is significant (non-rigorous) evidence supporting the multi-point convergence (see the discussion at the end of Section 1.2 in \cite{CU2}).

An important and well-studied link between the KPZ equation and Airy process is established through their mutual connection to directed random polymers in $1$ + $1$ dimension. Specifically, the free energy fluctuations of the continuous directed random polymer (a universal scaling limit of discrete directed polymer models \cite{AKQ}) are related to the narrow wedge initial data solution to the KPZ equation, while the fluctuations of certain zero-temperature degenerations of directed polymer models (like last-passage percolation) are related to the Airy process (see \cite{QR} and references therein for precise statements). The latter link can be understood as both models arising as different scaling limits of the same underlying stochastic dynamical systems. The situation is very different for stationary stochastic models. Specifically, while the Airy process has been related to interface fluctuations of random tiling and plane partition models no such connection has been established for the KPZ equation. In this sense, the appearance of the solution to the KPZ equation with narrow wedge initial data as a scaling limit of our stationary model $\mathbb{P}^{r,t}_{HL}$ is quite surprising. The distribution $\mathbb{P}^{r,t}_{HL}$ is thus the first example of a stationary model exhibiting KPZ statistics, and we view this as one of the main novel contributions of this work. 

We now turn to carefully describing the measure $\mathbb{P}^{r,t}_{HL}$ and explaining our results in detail.
%
\subsection{The measure $\mathbb{P}^{r,t}_{HL}$}\label{intro} \hspace{2mm}\\

We recommend Section \ref{partitions} for a brief overview of some concepts related to partitions and Young diagrams. A plane partition is a Young diagram filled with positive integers that form non-increasing rows and columns. A {\em connected component} of a plane partition is the set of all connected boxes of its Young diagram that are filled with the same number. The number of connected components in a plane partition $\pi$ is denoted by $k(\pi)$. Figure \ref{S1_1} shows an example of a plane partition and the 3-d Young diagram representing it. The connected components, which are separated in the Young diagram with bold lines, naturally correspond to the grey terraces in the 3-d diagram.

\begin{figure}[h]
\centering
\scalebox{0.6}{\includegraphics{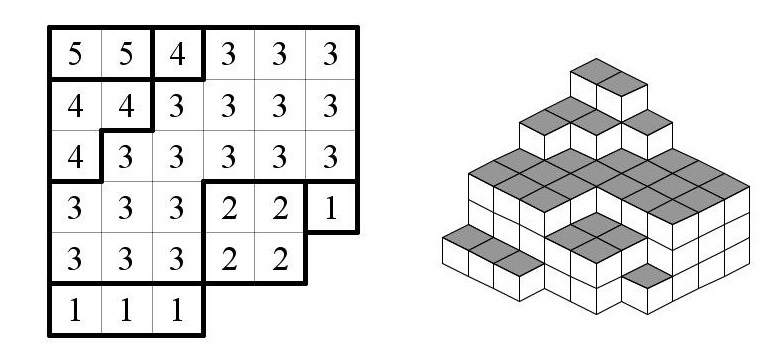}}
\caption{A plane partition and its 3-d Young diagram. In this example $k(\pi) = 7$.}
\label{S1_1}
\end{figure}

If a box $(i,j)$ belongs to a connected component $C$, we define its {\em level} $h(i,j)$ as the smallest $h \in \mathbb{N}$ such that $(i +h, j +h) \not \in C$. A {\em border component} is a connected subset of a connected component where all boxes have the same level. We also say that the border component is of this level. For the example above, the border components and their levels are illustrated in Figure \ref{S1_2}. 

\begin{figure}[h]
\centering
\scalebox{0.6}{\includegraphics{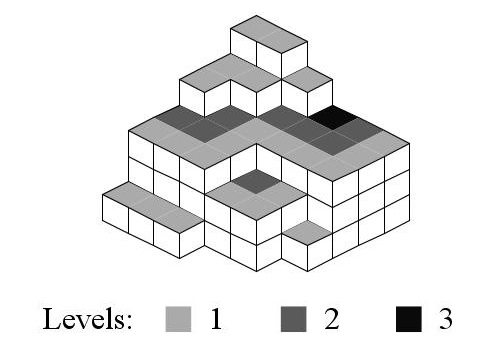}}
\caption{Border components and their levels.}
\label{S1_2}
\end{figure}

For each connected component $C$ we define a sequence $(n_1, n_2, ...)$ where $n_i$ is the number of $i$-level border components of $C$. We set 
$$P_C(t) := \prod_{i \geq 1} ( 1 - t^i)^{n_i}.$$ 
Let $C_1,C_2,...C_{k(\pi)}$ be the connected components of $\pi$. We define
\begin{equation}\label{Apoly}
A_\pi(t) := \prod_{i = 1}^{k(\pi)} P_{C_i}(t).
\end{equation}
For the example above $A_\pi(t) = (1-t)^{7}(1-t^2)^3(1-t^3).$

Given two parameters $r,t \in (0,1)$ we define $\mathbb{P}^{r,t}_{HL}$ to be the probability distribution on plane partitions such that
$$\mathbb{P}^{r,t}_{HL} (\pi) \propto r^{|\pi|}A_\pi(t), $$
where $|\pi|$ denotes the volume of $\pi$, i.e. the number of boxes in its 3-d Young diagram. In \cite{Vul} it was shown that
\begin{equation}\label{Zconst}
\sum_\pi r^{|\pi|}A_\pi(t) = \prod_{n = 1}^\infty \left( \frac{1 - tr^n}{1 - r^n}\right)^n = : Z(r,t).
\end{equation}
The above explicitly determines $\mathbb{P}^{r,t}_{HL}$ as
\begin{equation}\label{PDEF}
\mathbb{P}^{r,t}_{HL} (\pi) := Z(r,t)^{-1} r^{|\pi|}A_\pi(t),
\end{equation}
with $Z(r,t)$ as in (\ref{Zconst}).\\

\begin{remark}
 In Section \ref{HL} it will be shown that $\mathbb{P}^{r,t}_{HL}$ arises as a limit of certain Macdonald processes. These processes are defined in terms of Hall-Littlewood symmetric functions, which explains the ``HL'' in our notation.
\end{remark}

\begin{remark}
In the literature, the volume parameter is usually denoted by $q$, but we reserve this letter for a different parameter, which appears in the definition of Macdonald polynomials, and instead use the letter $r$.
\end{remark}

The distribution $\mathbb{P}^{r,t}_{HL}$ has been studied in the cases $t = 0$ and $t = -1$. When $t = 0$ we have $\mathbb{P}^{r,0}_{HL}(\pi) = Z(r,0)^{-1} r^{|\pi|}$, where $Z(r,0)$ is given by the famous MacMahon formula
\begin{equation}
Z(r,0) = \sum_{\pi} r^{|\pi|} = \prod_{n = 1}^\infty \left( \frac{1}{1 - r^n}\right)^n.
\end{equation}
We summarize a few of the known results when $t = 0$. In \cite{Kenyon} it was shown that under suitable scaling a partition $\pi$, distributed according to $\mathbb{P}^{r,0}_{HL}$, converges to a particular limit shape as $r \rightarrow 1^-$ (see also \cite{KeOk07}). In \cite{Ok} it was shown that  $\mathbb{P}^{r,0}_{HL}$ is described by a Schur process and has the structure of a determinantal point process with an explicit correlation kernel, suitable for asymptotic analysis. In \cite{FSpohn} it was shown that under suitable scaling the edge of the limit shape converges to the Airy process.

When $t = -1$ the measure $\mathbb{P}^{r,-1}_{HL}$ concentrates on strict plane partitions (these are plane partitions such that all border components have level $1$) and is described by a shifted Schur process as discussed in \cite{Vul2}. The shifted Schur process is shown to have the structure of a Pfaffian point process with an explicit correlation kernel, which can be analyzed as $r \rightarrow 1^-$. A limiting point density can be derived, which suggests a limit-shape phenomenon similar to the $t = 0$ case. To the author's knowledge there are no results regarding the edge asymptotics in this case.

The purpose of this paper is to study the distribution $\mathbb{P}^{r,t}_{HL}$ for $t \in (0,1)$. In particular, we will be interested in the behavior of a plane partition, distributed according to $\mathbb{P}^{r,t}_{HL}$, as the parameter $r$ goes to $1^-$. Part of the difficulty in dealing with the case $t \in (0,1)$ comes from the fact that a determinantal or Pfaffian point process structure is no longer availbable. Instead, we will use the formalism of Macdonald difference operators (see \cite{BorCor} and \cite{BorGor}) to obtain formulas for a certain class of observables for a plane partition $\pi$, distributed accodrding to $\mathbb{P}^{r,t}_{HL}$. These formulas can be asymptotically analyzed and imply one-point convergence results for the bottom slice of $\pi$.

%
\subsection{Main results} \label{mainResults}\hspace{2mm}\\

For a partition $\lambda$, we let $\lambda_1'$ denote its largest column (i.e. the number of non-zero parts). Given a plane partition $\pi$, we consider its diagonal slices $\lambda^t$ (alternatively $\lambda(t)$) for $t\in \mathbb{Z}$, i.e. the sequences
$$\lambda^k = \lambda(k) = (\pi_{i, i + k}) \hspace{3mm} \mbox{ for } i \geq \max(0, -k).$$
For $r \in (0,1)$, $\tau \in \mathbb{R}$ we define
\begin{equation}\label{S1scaling}
N(r) := \frac{1}{1-r} \hspace{2mm}  \mbox{ and }  \chi:= \left[ \frac{e^{-|\tau|/2}}{(1 + e^{-|\tau|/2})^2}\right]^{-1/3} = \left[ \frac{4}{\cosh^2(\tau/4)}\right]^{-1/3}.
\end{equation}
Below we analyze the large $N$ asymptotics of $\lambda_1'(\lfloor \tau N(r) \rfloor)$ of a random plane partition, distributed according to $\mathbb{P}^{r,t}_{HL}$. 

\begin{theorem}\label{TW}
Consider the measure $\mathbb{P}_{HL}^{r,t}$ on plane partitions, given in (\ref{PDEF}), with $t \in (0,1)$ fixed. Then for all $\tau \in \mathbb{R} \backslash \{0\}$ and $x \in \mathbb{R}$ we have
\begin{equation*}
\lim_{r \rightarrow 1^-} \mathbb{P}_{HL}^{r,t} \left( \frac{  \lambda'_1( \lfloor \tau N(r) \rfloor)  - 2N(r) \log (1 + e^{-|\tau|/2}) }{\chi^{-1} N(r)^{1/3}} \leq x\right) = F_{GUE}(x),
\end{equation*}
where $F_{GUE}$ is the GUE Tracy-Widom distribution \cite{TWPaper}. The coefficients $N(r)$ and $\chi$ are as in (\ref{S1scaling}).
\end{theorem} 

\begin{theorem}\label{TCDRP}
Consider the measure $\mathbb{P}_{HL}^{r,t}$ on plane partitions, given in (\ref{PDEF}). Suppose $T > 0$ is fixed and $\frac{- \log t}{(1-r)^{1/3}} = \chi (T/2)^{1/3}$. Then for all $\tau \in \mathbb{R}\backslash \{0\}$ and $x \in \mathbb{R}$ we have
\begin{equation*}
\lim_{r \rightarrow 1^-} \mathbb{P}_{HL}^{r,t} \left( \frac{ \lambda'_1( \lfloor \tau N(r) \rfloor)  - 2N(r)\log (1 + e^{-|\tau|/2}) }{ \chi^{-1} N(r)^{1/3} (T/2)^{-1/3}} + \log(N(r)^{1/3}\chi^{-1}(T/2)^{-1/3}) \leq x\right) = F_{CDRP}(x)
\end{equation*}
 where $F_{CDRP}(x) = \mathbb{P}\left(\mathcal{F}(T,0) + T/24 \leq x \right)$ and $\mathcal{F}(T,X)$ is the Hopf-Cole solution to the Kardar-Parisi-Zhang equation with narrow wedge initial data \cite{CorQ, Ber}. The coefficients $N(r)$ and $\chi$ are as in (\ref{S1scaling}).
\end{theorem}

The definitions of $F_{GUE}(x)$ and $F_{CDRP}(x)$ are provided below in Definition \ref{TWDef}. In Sections \ref{Section4} and \ref{Section7} we will reduce the proofs of the above results to claims on certain asymptotics of Fredholm determinant formulas. Throughout the paper, we will, rather informally, refer to the limiting regime in Theorem \ref{TW} as ``the GUE case'' and to the one in Theorem \ref{TCDRP} as ``the CDRP case''.

\begin{remark}
The exclusion of the case $\tau = 0$ appears to be a technical assumption, necessary for our proofs to work. It is possible that the arguments of this paper can be modified to include this case, but we will not pursue this goal.
\end{remark}

Before we record the limiting distributions that appear in our results, we briefly discuss the definition of $\mathcal{F}(X,T)$. The {\em continuous directed random polymer} (CDRP) is a universal scaling limit for 1 + 1 dimensional directed random polymers \cite{AKQ, Cal}. Its partition function with respect to general boundary perturbations is given as follows (this is Definition 1.7 in \cite{BCF}). 

\begin{definition}\label{SHE}
The partition function for the continuum directed random polymer with boundary perturbation $\ln \mathcal{Z}_0(X)$ is given by the solution to the stochastic heat equation (SHE) with multiplicative Gaussian space-time white noise and $\mathcal{Z}_0(X)$ initial data:
\begin{equation}
\partial_T \mathcal{Z} = \frac{1}{2} \partial_X^2 \mathcal{Z} + \mathcal{Z} \dot{\mathcal{W}}, \hspace{5mm} \mathcal{Z}(0,X) = \mathcal{Z}_0(X).
\end{equation}
The initial data $\mathcal{Z}_0(X)$ may be random but is assumed to be independent of the Gaussian space-time white noise $ \dot{\mathcal{W}}$ and is assumed to be almost surely a sigma-finite positive measure. Observe that even if $\mathcal{Z}_0(X)$ is zero in some regions, the stochastic PDE makes sense and hence the partition function is well-defined.
\end{definition}
A detailed description of the SHE and the class of initial data for which it is well-posed can be found in \cite{CorQ, Ber}. Provided, $\mathcal{Z}_0$ is an almost surely sigma-finite positive measure, it follows from the work of Mueller \cite{Muller} that, almost surely, $\mathcal{Z}(T,X)$ is positive for all $T > 0$ and $X \in \mathbb{R}$ and hence its logarithm is a well-defined random space-time function. The following is Definition 1.8 in \cite{BCF}.

\begin{definition}\label{freeE}
For $\mathcal{Z}_0$ an almost surely sigma-finite positive measure define the free energy for the continuous directed random polymer with boundary perturbation $\ln \mathcal{Z}_0(X)$ as 
$$\mathcal{F}(T,X) = \ln \mathcal{Z} (T,X).$$
\end{definition}
The random space-time function $\mathcal{F}$ is also the Hopf-Cole solution to the Kardar-Parisi-Zhang equation with initial data $\mathcal{F}_0(X) = \ln \mathcal{Z}_0(X)$ \cite{CorQ, Ber}. In this paper, we will focus on the case when $\mathcal{Z}_0(X) = {\bf 1}_{\{X = 0\}}$, which is known as the {\em narrow wedge} or $0$-{\em spiked} initial data \cite{CorQ, BCF}. In Theorem 1.10 of \cite{BCF} it was shown that when $\mathcal{Z}_0(X) = {\bf 1}_{\{X = 0\}}$, one has the following formula for the Laplace tansform of $\exp (\mathcal{F}(T,0) + T/24)$.
\begin{equation}\label{KCDRP}
\mathbb{E} \left[ e^{-e^x \exp(\mathcal{F}(T,0) + T/24)}\right] = \det (I - K_{CDRP})_{L^2}(\mathbb{R}_+),
\end{equation}
where the right-hand-side (RHS) denotes the Fredholm determinant (see Section \ref{FredholmS}) of the operator $K_{CDRP}$, given in terms of its integral kernel
\begin{equation}\label{KCDRPDef}
K_{CDRP}(\eta, \eta') := \int_{\mathbb{R}} dt \frac{e^x}{e^x + e^{-t/\sigma}}Ai(t + \eta)Ai(t + \eta').
\end{equation}
In the above formula $\sigma = (2/T)^{1/3}$, $x \in \mathbb{R}$ and $Ai(\cdot)$ is the Airy function.

We now record the definitions of the limiting distributions that appear in Theorems \ref{TW} and \ref{TCDRP}. The first part of the following definition appears in Definition 1.6 in \cite{BCF}.
\begin{definition}\label{TWDef} The GUE Tracy-Widom distribution \cite{TWPaper} is defined as
$$F_{GUE}(x) := \det ( I - K_{Ai})_{L^2(x,\infty)},$$
where $K_{Ai}$ is the Airy kernel, that has the integral representation
$$K_{Ai}(\eta, \eta') = \frac{1}{(2\pi \iota)^2} \int_{e^{-2\pi \iota/3}\infty}^{e^{2\pi \iota/3}\infty} dw  \int_{e^{-\pi \iota/3}\infty}^{e^{\pi \iota/3}\infty}dz \frac{1}{z-w}\frac{e^{z^3/3 - z\eta'}}{e^{w^3/3 - w\eta}},  $$
where the contours $z$ and $w$ do not intersect.\\

Suppose $\mathcal{F}(T,X)$ is the free energy for the CDRP with boundary perturbation $\ln \mathcal{Z}_0(X)$ and $\mathcal{Z}_0(X) = {\bf 1}_{\{X = 0\}}$ as in Definition \ref{freeE}. Then we define
$$F_{CDRP}(x):= \mathbb{P}(\mathcal{F}(T,0) + T/24 \leq x).$$
\end{definition}

%
\subsection{Discussion and extensions}\label{Section1.4} \hspace{2mm}\\
In this section we discuss some of the implications of the results of the paper and some of their possible extensions. \\

We start by considering possible limit shape phenomena. In \cite{Kenyon} it was shown that if each dimension of a plane partition $\pi$, distributed according to  $\mathbb{P}^{r,t}_{HL}$ with $t = 0$, is scaled by $1-r$ then as $r \rightarrow 1^-$ the distribution concentrates on a limit shape with probability $1$. We expect that a similar phenomenon occurs for any value $t \in (0,1)$. The limit shape, if it exists, should depend on $t$, which one observes by considering the volume of the plane partition. Specifically, we have that 
$$\mathbb{E}\left[ |\pi| \right] = \frac{r\frac{d}{dr}Z(r,t)}{Z} \mbox{ and } Var(|\pi|) = \mathbb{E}\left[ |\pi|^2 \right] - \mathbb{E}\left[ |\pi|\right]^2 = r\frac{d}{dr} \mathbb{E}\left[|\pi|\right].$$
Using that $Z(r,t) = \prod_{n = 1}^\infty \left( \frac{1 - tr^n}{1 - r^n}\right)^n$ one readily verifies that 
$$\mathbb{E}\left[ |\pi| \right] = \sum_{k = 1}^\infty \frac{r^k(1 + r^k)}{(1 - r^k)^3} - \sum_{k = 1}^\infty t^k\frac{r^k(1 + r^k)}{(1 - r^k)^3}.$$ 
The latter implies that $\lim_{r \rightarrow 1^-} \mathbb{E}\left[ (1-r)^3|\pi| \right] = 2\zeta(3) - 2Li_3(t)$, where $\zeta(s) = \sum_{n = 1}^\infty \frac{1}{n^s}$ is the Riemann zeta function and $Li_3(z) = \sum_{k = 1}^\infty \frac{z^k}{k^3}$ is the polylogarithm of order $3$. In addition, one verifies that $\lim_{r \rightarrow 1^-} Var \left( (1-r)^3|\pi| \right) = 0$ and so the rescaled volume $(1-r)^3|\pi|$ converges in probability to $2\zeta(3) - 2Li_3(t)$. In particular, the volume decreases from $2\zeta(3)$ to $0$ as $t$ varies from $0$ to $1$. When $t = 1$ the measure $\mathbb{P}^{r,t}_{HL}$ is concentrated on the empty plane partition for any value of $r$ and so convergence of the volume to $0$ is expected.

In sharp contrast, the result of Theorem \ref{TW} suggests that while the volume of the plane partition decreases in $t$ the bottom slice asymptotically looks the same. The latter has been empirically verified through simulations and is presented in Figures \ref{S1_3} - \ref{S1_6}, where the red line indicates the limit shape $2\log (1 + e^{-|\tau|/2})$ in Theorem \ref{TW}.

\begin{figure}[h]
\centering
\begin{minipage}{.5\textwidth}
  \centering
  \includegraphics[width=0.9\linewidth]{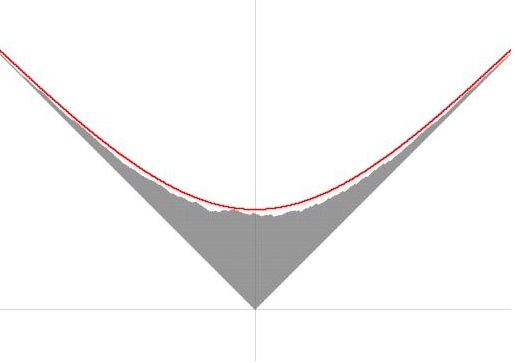}
\captionsetup{width=.9\linewidth}
  \caption{$ t= 0$.}
  \label{S1_3}
\end{minipage}%
\begin{minipage}{.5\textwidth}
  \centering
  \includegraphics[width=0.9\linewidth]{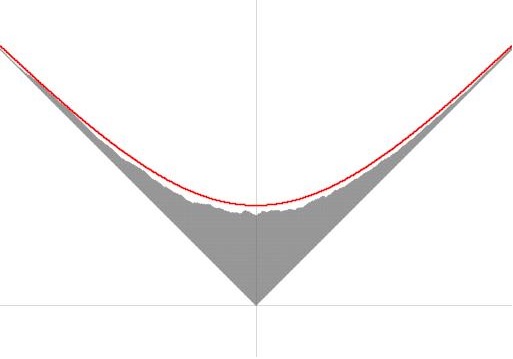}
\captionsetup{width=.9\linewidth}
  \caption{$t = 0.2$. }
  \label{S1_4}
\end{minipage}
\end{figure}

\begin{figure}[h]
\centering
\begin{minipage}{.5\textwidth}
  \centering
  \includegraphics[width=0.9\linewidth]{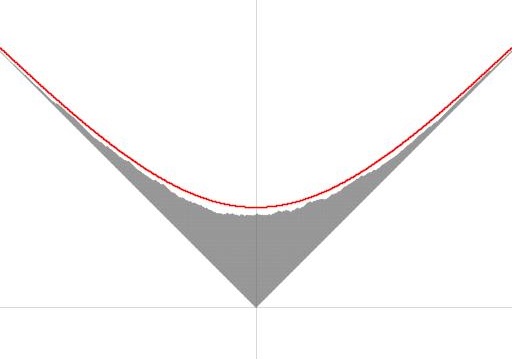}
\captionsetup{width=.9\linewidth}
  \caption{$t = 0.4$.}
  \label{S1_5}
\end{minipage}%
\begin{minipage}{.5\textwidth}
  \centering
  \includegraphics[width=0.9\linewidth]{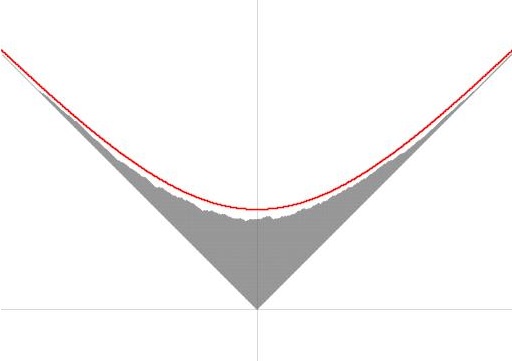}
\captionsetup{width=.9\linewidth}
  \caption{$t = 0.6$. }
  \label{S1_6}
\end{minipage}
\end{figure}
What happens as $t$ increases to $1$ is that the mass from the top part of the plane partition $\pi$ decreases (so $\pi_{i,j}$ decrease), but the base (given by the non-zero $\pi_{i,j}$) remains asymptotically the same. The latter can be observed in the left parts of Figures \ref{S1_7} and \ref{S1_8} (we will get to the right parts shortly). 
\begin{figure}[h]
\centering
\scalebox{0.53}{\includegraphics{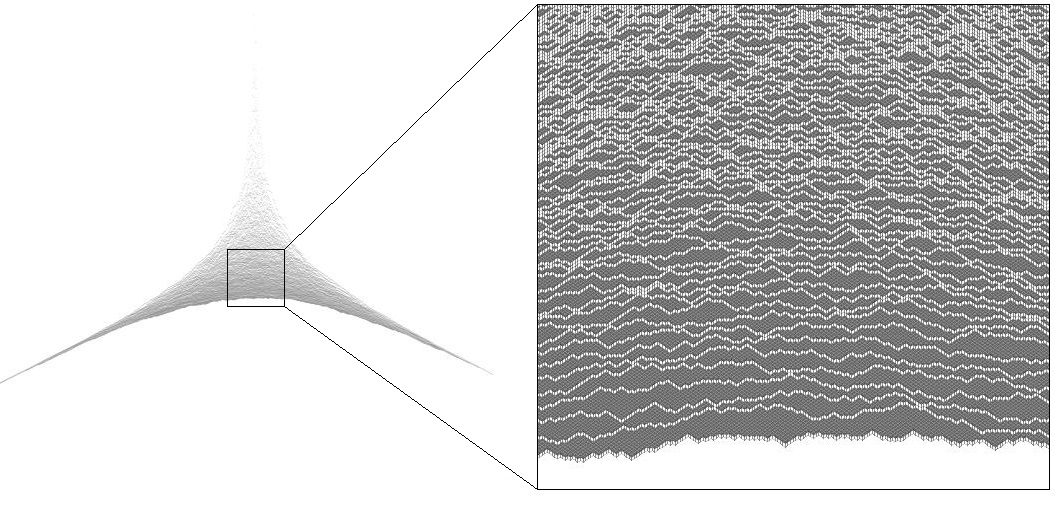}}
\caption{Simulation with $t = 0.4$.}
\label{S1_7}
\end{figure}
\begin{figure}[h]
\centering
\scalebox{0.53}{\includegraphics{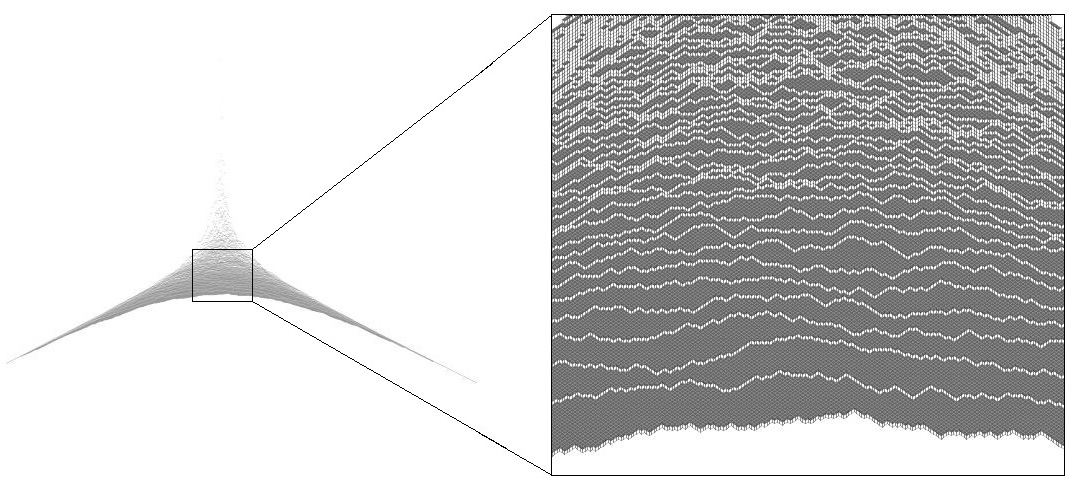}}
\caption{Simulation with $t = 0.8$.}
\label{S1_8}
\end{figure}
\FloatBarrier

We next turn to possible extensions of Theorems \ref{TW} and \ref{TCDRP}. The statement of Theorem \ref{TW} can be understood as a one-point convergence result about the fluctuations of the bottom slice of a plane partition $\pi$, distributed according to $\mathbb{P}^{r,t}_{HL}$, to $F_{GUE}$. $F_{GUE}$ is the one point marginal distribution of the Airy process and in \cite{FSpohn} it was shown that the fluctuations in the case of $t = 0$ converge as a process to the Airy process. Consequently, it is natural to suppose that the same occurs for any value of $t \in (0,1)$. We will take this idea further, using the fact that the Airy process appears as the distribution of the bottom line of the Airy line ensemble \cite{CorHamA}, and conjecture that the fluctuations of all horizontal slices of $\pi$ converge (in the sense of line ensembles - see the discussion at the beginning of Section 2.1 in \cite{CorHamA}) to the Airy line ensemble. The exact formulation is presented in Conjecture \ref{S1Conj1}. 

In a similar fashion, a natural extension of Theorem \ref{TCDRP} is to show that the fluctuations of the bottom slice converge as a process to $\mathcal{F}(T,X)$. The (shifted) Hopf-Cole solution to the KPZ equation with narrow wedge initial data $\mathcal{F}(T,X) + T/24$ is also the distribution of the top line of the KPZ line ensemble \cite{CorHamK}, and so we will conjecture that the fluctuations of all horizontal slices of $\pi$ (upon appropriate shifts and scaling) converge (in the sense of line ensembles) to the KPZ line ensemble in the sense of line ensembles. The formulation is presented in Conjecture \ref{S1Conj2}.

 For $\tau > 0$ let $f(\tau) = 2\log (1 + e^{-\tau/2})$, $f'(\tau) = -\frac{e^{-\tau/2}}{1 + e^{-\tau/2}}$ and $f''(\tau) = \frac{1}{2} \frac{e^{-\tau/2}}{(1 + e^{-\tau/2})^2}$. Also set $N(r) = \frac{1}{1-r}$. With this notation we have the following conjectures.
\begin{conjecture}\label{S1Conj1}
Consider the measure $\mathbb{P}_{HL}^{r,t}$ on plane partitions, given in (\ref{PDEF}), with $t \in (0,1)$ fixed. For $\tau \in \mathbb{R}$ define the random $\mathbb{N} \times \mathbb{R}$-indexed  line ensemble $\Lambda^\tau$ as
\begin{equation}\label{S1conj1E}
\Lambda^\tau_k(s) = \frac{\lambda_k'( \lfloor\tau N + sN^{2/3} \rfloor) - Nf(|\tau|) - sN^{2/3}f'(|\tau|) - (1/2)s^2N^{1/3}f''(|\tau|)}{\sqrt[3]{2f''(|\tau|) N}}.
\end{equation}
 Then as $r \rightarrow 1^-$ we have $\Lambda^\tau \implies \mathcal{A}^\tau$ (weak convergence in the sense of line ensembles), where $\mathcal{A}^\tau$ is defined as $\mathcal{A}^\tau_k(s) = \mathcal{A}_k(s\sqrt[3]{2f''(|\tau|)}/2)$ and $(\mathcal{A}_k)_{k \in \mathbb{N}}$ is the Airy line ensemble.
\end{conjecture}

\begin{conjecture}\label{S1Conj2}
Consider the measure $\mathbb{P}_{HL}^{r,t}$ on plane partitions, given in (\ref{PDEF}). Suppose $T > 0$ is fixed and $\frac{- \log t}{(1-r)^{1/3}} =\frac{(T/2)^{1/3}}{ \sqrt[3]{2f''(|\tau|)}}$. For $\tau \in \mathbb{R}$ define the random $\mathbb{N} \times \mathbb{R}$-indexed  line ensemble $\Xi^\tau$ as
\begin{equation}\label{S1conj2E}
\begin{split}
\Xi^\tau_k(s) =  \frac{ \lambda_k'( \lfloor\tau N + sN^{2/3} \rfloor)  - Nf(|\tau|) - sN^{2/3}f'(|\tau|) - (1/2)s^2N^{1/3}f''(|\tau|)}{(T/2)^{-1/3}\sqrt[3]{2f''(|\tau|) N} } - T/24 + \\
 + \log ((T/2)^{-1/3}\sqrt[3]{2f''(|\tau|) N} ) + (k-1)\log\left( \frac{ NT^{-1}(2f''(|\tau|))^{-3/2}}{2\sqrt{2}}\right) - \frac{s^2T^{1/3} (2f''(|\tau|))^{2/3}}{8} .
\end{split}
\end{equation}
Then as $r \rightarrow 1^-$ we have $\Xi^{\tau} \implies \mathcal{H}^{\tau, T}$ (weak convergence in the sense of line ensembles), where $\mathcal{H}^{\tau, T}$ is defined as $\mathcal{H}^{\tau, T}_k(s) = \mathcal{H}^T_k(sT^{2/3}\sqrt[3]{2f''(|\tau|)}/2)$ and $(\mathcal{H}_k^T)_{k \in \mathbb{N}}$ is the KPZ line ensemble.
\end{conjecture}

Motivation about the choice of scaling as well as partial evidence supporting the validity of these conjectures is given in Section \ref{Section9}. Here we will only make the observation that in the statement of Conjecture \ref{S1Conj1}, the separation between consecutive horizontal slices of $\pi$, distributed according to $\mathbb{P}_{HL}^{r,t}$ is suggested to be of order $N^{1/3}$, which is the order of the fluctuations. On the other hand, in Conjecture \ref{S1Conj2} there is a deterministic shift of order $N^{1/3}\log N$, while fluctuations remain of order $N^{1/3}$. The latter phenomenon can be observed in simulations, as is shown in Figures \ref{S1_7} and \ref{S1_8}. Namely, the conjectures suggest that as $t$ goes to $1$, one should observe a larger spacing between the bottom slices of $\pi$, which is clearly visible.

%
\subsection{Outline and acknowledgments} \hspace{2mm}\\

The introductory section above formulated the problem statement and gave the main results of the paper. In Section \ref{Section2} we present some background on partitions, symmetric functions, Macdonald processes and Fredholm determinants. In Section \ref{finiteS} we derive a formula for the $t$-Laplace transform of a certain random variable in terms of a Fredholm determinant using the approach of Macdonald difference operators. In Sections \ref{Section4} and \ref{Section7} we extend the results of Section \ref{finiteS} to a setting suitable for asymptotic analysis in the GUE and CDRP cases respectively and prove Theorems \ref{TW} and \ref{TCDRP}. Section \ref{SSart} summarizes various technical results used in the proofs of Theorems \ref{mainThm} and \ref{mainThm2}. Section \ref{Section9} presents a sampling algorithm for random plane partitions, provides empirical evidence supporting the results of this paper and further motivates out proposed conjectures.\\

I wish to thank my advisor, Alexei Borodin, for suggesting this problem to me and for his continuous help and guidance. Also, I thank Mirjana Vuleti{\' c} for helpful discussions.

\section{General definitions}\label{Section2}
In this section we summarize some facts about symmetric functions and Macdonald processes. Macdonald processes were defined and studied in \cite{BorCor}, which is the main reference for what follows together with the book of Macdonald \cite{Mac}. We explain how the measure $\mathbb{P}^{r,t}_{HL}$ arises as a limit of a certain sequence of Macdonald processes and end with some background on Fredholm determinants, used in the text.

%
\subsection{Partitions and Young diagrams}\label{partitions}\hspace{2mm} \\

We start by fixing terminology and notation. A {\em partition} is a sequence $\lambda = (\lambda_1, \lambda_2,\cdots)$ of non-negative integers such that $\lambda_1 \geq \lambda_2 \geq \cdots$ and all but finitely many elements are zero. We denote the set of all partitions by $\mathbb{Y}$. The {\em length} $\ell (\lambda)$ is the number of non-zero $\lambda_i$ and the {\em weight} is given by $|\lambda| = \lambda_1 + \lambda_2 + \cdots$ . If $|\lambda| = n$ we say that $\lambda$ {\em partitions} $n$, also denoted by $\lambda \vdash n$. There is a single partition of $0$, which we denote by $\varnothing$. An alternative representation is given by $\lambda = 1^{m_1}2^{m_2}\cdots$, where $m_j(\lambda) = |\{i \in \mathbb{N}: \lambda_i = j\}|$ is called the {\em multiplicity} of $j$ in the partition $\lambda$. There is a natural ordering on the space of partitions, called the {\em reverse lexicographic order}, which is given by
$$\lambda > \mu \iff \exists k \in \mathbb{N} \mbox{ such that } \lambda_i = \mu_i \mbox{, whenever }i < k \mbox{ and } \lambda_k > \mu_k.$$

A {Young diagram} is a graphical representation of a partition $\lambda$, with $\lambda_1$ left justified boxes in the top row, $\lambda_2$ in the second row and so on. In general, we do not distinguish between a partition $\lambda$ and the Young diagram representing it. The {\em conjugate} of a partition $\lambda$ is the partition $\lambda'$ whose Young diagram is the transpose of the diagram $\lambda$. In particular, we have the formula $\lambda_i' = |\{j \in \mathbb{N}: \lambda_j \geq i\}|$.

Given two diagrams $\lambda$ and $\mu$ such that $\mu \subset \lambda$ (as a collection of boxes), we call the difference $\theta = \lambda - \mu$ a {\em skew Young diagram}. A skew Young diagram $\theta$ is a {\em horizontal $m$- strip} if $\theta$ contains $m$ boxes and no two lie in the same column. If $\lambda - \mu$ is a horizontal strip we write $\lambda \succ \mu$. Some of these concepts are illustrated in Figure \ref{S2_1}.
\begin{figure}[h]
\centering
\scalebox{0.45}{\includegraphics{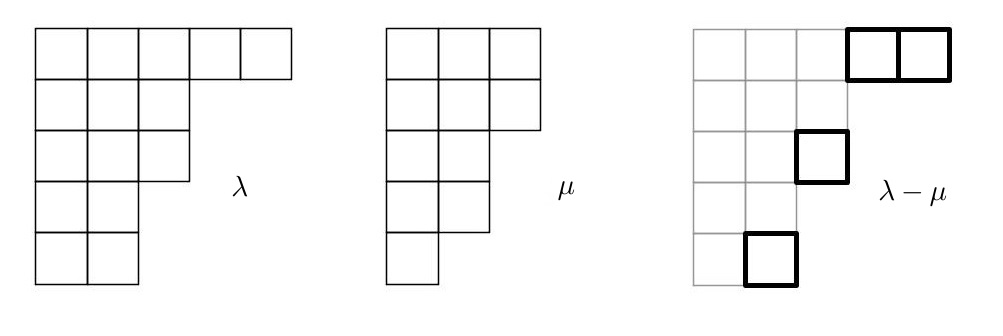}}
\caption{The Young diagram $\lambda = (5,3,3,2,2)$ and its transpose (not shown) $\lambda' = ( 5,5,3,1,1)$. The length $\ell(\lambda) = 5$ and weight $|\lambda| = 15$. The Young diagram $\mu = (3,3,2,2,1)$ is such that $\mu \subset \lambda$. The skew Young diagram $\lambda - \mu$ is shown in {\em black bold lines} and is a horizontal $4$-strip.}
\label{S2_1}
\end{figure}

A {\em plane partition} is a two-dimensional array of nonnegative integers
$$\pi = (\pi_{i,j}), \hspace{3mm} i,j = 0,1,2,...,$$
such that $\pi_{i,j} \geq \max (\pi_{i,j+1}, \pi_{i+1,j})$ for all $i,j \geq 0$ and the {\em volume} $|\pi| = \sum_{i,j \geq 0} \pi_{i,j}$ is finite. Alternatively, a plane partition is a Young diagram filled with positive integers that form non-increasing rows and columns. A graphical representation of a plane partition $\pi$ is given by a {\em $3$-dimensional Young diagram}, which can be viewed as the plot of the function 
$$(x,y) \rightarrow \pi_{\lfloor x \rfloor, \lfloor y \rfloor} \hspace{3mm} x,y > 0.$$
Given a plane partition $\pi$ we consider its diagonal slices $\lambda^t$ for $t\in \mathbb{Z}$, i.e. the sequences
$$\lambda^t = (\pi_{i, i + t}) \hspace{3mm} \mbox{ for } i \geq \max(0, -t).$$
One readily observes that $\lambda^t$ are partitions and satisfy the following interlacing property
$$\cdots \prec \lambda^{-2} \prec \lambda^{-1} \prec \lambda^0 \succ \lambda^1 \succ \lambda^2 \succ \cdots.$$
Conversely, any (finite) sequence of partitions $\lambda^{t}$, satisfying the interlacing property, defines a partition $\pi$ in the obvious way. Concepts related to plane partitions are illustrated in Figure \ref{S2_2}.

\begin{figure}[h]
\centering
\scalebox{0.45}{\includegraphics{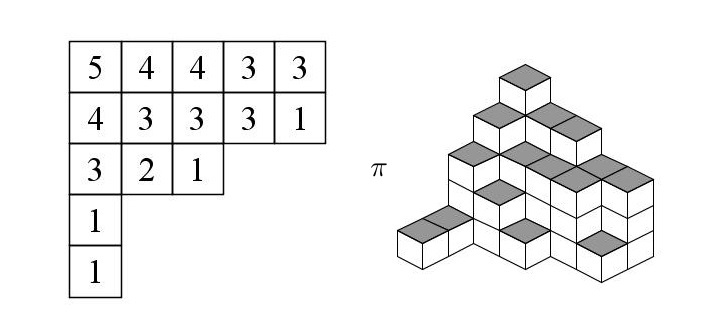}}
\caption{The plane partition $\pi = \varnothing \prec (1) \prec (1) \prec (3) \prec (4,2) \prec (5,3,1) \succ (4,3) \succ (4,3) \succ (3,1) \succ(3) \succ \varnothing$ . The volume $|\pi| = 41$.}
\label{S2_2}
\end{figure}

%
\subsection{Macdonald symmetric functions}\label{MacSym}\hspace{2mm} \\

We let $\Lambda_X$ denote the $\mathbb{Z}_{\geq 0}$ graded algebra over $\mathbb{C}$ of symmetric functions in variables $X = (x_1,x_2,...)$, which can be viewed as the algebra of symmetric polynomials in infinitely many variables with bounded degree, see e.g. Chapter I of \cite{Mac} for general information on $\Lambda_X$. One way to view $\Lambda_X$ is as an algebra of polynomials in Newton power sums
$$p_k(X) = \sum_{i = 1}^{\infty} x_i^k , \hspace{3mm} \mbox{     for          } k\geq 1.$$
For any partition $\lambda$ we define
$$p_\lambda(X) = \prod_{i = 1}^{\ell(\lambda)}p_{\lambda_i}(X),$$
and note that $p_\lambda(X)$, $\lambda \in \mathbb{Y}$ form a linear basis in $\Lambda_X$.

An alternative set of algebraically independent generators of $\Lambda_X$ is given by the elementary symmetric functions 
$$e_k(X) = \sum_{1\leq i_1 < i_2 < \cdots < i_k}x_{i_1} x_{i_2}\cdots x_{i_k}, \hspace{3mm} \mbox{  for  } k \geq 1.$$

In what follows we fix two parameters $q,t$ and assume that they are real numbers with $q,t \in (0,1)$. Unless the dependence on $q,t$ is important we will suppress them from our notation, similarly for the variable set $X$.

The Macdonald scalar product $\langle \cdot, \cdot \rangle$ on $\Lambda$ is defined via 
\begin{equation}\label{MSP}
\langle p_\lambda, p_\mu \rangle = \delta_{\lambda, \mu} \left( \prod_{i = 1}^{\ell(\lambda)} \frac{1 - q^{\lambda_i}}{1 - t^{\lambda_i}}\right)\left(\prod_{i  =1}^{\lambda_1} i^{m_i(\lambda)}m_i(\lambda)!\right).
\end{equation}
The following definition can be found in Chapter VI of \cite{Mac}.
\begin{definition}
 Macdonald symmetric functions $P_\lambda$, $\lambda \in \mathbb{Y}$, are the unique linear basis of $\Lambda$ such that
\begin{enumerate}[label = \arabic{enumi}., leftmargin=1.5cm]
\item $\langle P_\lambda, P_\mu \rangle = 0$ unless $\lambda = \mu$.
\item The leading (with respect to reverse lexicographic order) monomial in $P_\lambda$ is $\prod_{i=1}^{\ell(\lambda)}x_i^{\lambda_i}.$
\end{enumerate}
\end{definition}

\begin{remark}
 The Macdonald symmetric function $P_\lambda$ is a homogeneous symmetric function of degree $|\lambda|$.  
\end{remark}

\begin{remark} If we set $x_{N+1} = x_{N+2} = \cdots = 0$ in $P_\lambda(X)$, then we obtain the symmetric polynomials $P_\lambda(x_1,...,x_N)$ in $N$ variables, which are called the Macdonald polynomials.
\end{remark}

There is a second family of Macdonald symmetric functions $Q_\lambda$, $\lambda \in \mathbb{Y}$, which are dual to $P_\lambda$ with respect to the Macdonald scalar product:
$$Q_\lambda = \langle P_\lambda, P_\lambda \rangle^{-1}P_{\lambda}, \hspace{4mm} \langle P_\lambda,Q_\mu \rangle = \delta_{\lambda,\mu}, \hspace{2mm} \lambda, \mu \in \mathbb{Y}.$$

For two sets of variables $X = (x_1, x_2,...)$ and $Y = (y_1, y_2,...)$ define
$$\Pi(X;Y) = \sum_{\lambda \in \mathbb{Y}} P_{\lambda}(X) Q_{\lambda}(Y).$$
Then from Chapter VI (2.5) in \cite{Mac} we have
\begin{equation}\label{CauchyI}
\Pi(X;Y)  = \prod_{i,j = 1}^{\infty} \frac{(tx_iy_j;q)_\infty}{(x_iy_j;q)_\infty},
\end{equation}
where $(a;q)_\infty = (1-a)(1-aq)(1-aq^2)\cdots$ is the $q$-Pochhammer symbol. The above equality holds when both sides are viewed as formal power series in the variables $X$, $Y$ and it is known as the Cauchy identity.

We next proceed to define the skew Macdonald symmetric functions (see Chapter VI in \cite{Mac} for details). Take two sets of variables $X = (x_1,x_2, ...)$ and $Y = (y_1,y_2,...)$ and a symmetric function $f \in \Lambda$. Let $(X,Y)$ denote the union of sets of variables $X$ and $Y$. Then we can view $f(X,Y) \in \Lambda_{(X,Y)}$ as a symmetric function in $x_i$ and $y_i$ together. More precisely, let
$$f = \sum_{\lambda \in \mathbb{Y}} C_{\lambda} p_\lambda = \sum_{\lambda \in \mathbb{Y}} C_{\lambda}\prod_{i = 1}^{\ell(\lambda)}p_{\lambda_i},$$
be the expansion of $f$ into the basis $p_\lambda$ of power symmetric functions (in the above sum $C_\lambda = 0$ for all but finitely many $\lambda$). Then we have
$$f(X,Y) = \sum_{\lambda \in \mathbb{Y}} C_{\lambda} \prod_{i = 1}^{\ell(\lambda)}(p_{\lambda_i}(X) + p_{\lambda_i}(Y)).$$
In particular, we see that $f(X,Y)$ is the sum of products of symmetric functions of $x_i$ and symmetric functions of $y_i$.

Skew Macdonald symmetric functions $P_{\lambda / \mu}$, $Q_{\lambda/\mu}$ are defined as the coefficients in the expansion
\begin{equation}
P_{\lambda}(X, Y) = \sum_{\mu \in \mathbb{Y}}P_{\mu}(X) P_{\lambda/\mu}(Y) \hspace{2mm} \mbox{ and } \hspace{2mm} Q_{\lambda}(X, Y) = \sum_{\mu \in \mathbb{Y}}Q_{\mu}(X) Q_{\lambda/\mu}(Y) 
\end{equation}

\begin{remark}
The skew Macdonald symmetric function $P_{\lambda/ \mu}$ is $0$ unless $\mu \subset \lambda$, in which case it is homogeneous of degree $|\lambda| - |\mu|$.  
\end{remark}

\begin{remark}
 When $\lambda = \mu$, $P_{\lambda/ \mu} = 1$ and if $\mu = \varnothing$ (the unique partition of $0$), then $P_{\lambda / \mu} = P_\lambda.$
\end{remark}

We mention here two important special cases for the skew Macdonald symmetric functions. Suppose $x_2 = x_3 = \cdots = 0$. Then we have
$$P_{\lambda / \mu}(x_1) = \psi_{\lambda/\mu}x_1^{|\lambda| - |\mu|} \hspace{2mm} \mbox{ and }Q_{\lambda / \mu}(x_1) = \phi_{\lambda/\mu}x_1^{|\lambda| - |\mu|},$$
whenever $\lambda \succ \mu$ and zero otherwise. The coefficients $\phi_{\lambda / \mu}$ and $\psi_{\lambda /\mu}$ have exact formulas as is shown in Chapter VI (6.24) of \cite{Mac}, and we write them below. Let $f(u) = (tu;q)_{\infty}/(qu;q)_{\infty}$. If $\lambda \succ \mu$ then
\begin{equation}\label{phiForm}
\phi_{\lambda / \mu}(q,t) = \prod_{1 \leq i \leq j \leq \ell(\lambda)}\frac{f(q^{\lambda_i - \lambda_j}t^{j-i})f(q^{\mu_i - \mu_{j+1}}t^{j-i})}{f(q^{\lambda_i - \mu_j}t^{j-i})f(q^{\mu_i - \lambda_{j+1}}t^{j-i})},
\end{equation}
\begin{equation}\label{psiForm}
\psi_{\lambda / \mu}(q,t) = \prod_{1 \leq i \leq j \leq \ell(\mu)}\frac{f(q^{\mu_i - \mu_j}t^{j-i})f(q^{\lambda_i - \lambda_{j+1}}t^{j-i})}{f(q^{\lambda_i - \mu_j}t^{j-i})f(q^{\mu_i - \lambda_{j+1}}t^{j-i})},
\end{equation}
otherwise the coefficients are zero.

%
\subsection{The Macdonald process}\label{MacProc}\hspace{2mm}\\

A {\em specialization} $\rho$ of $\Lambda$ is a unital algebra homomorphism of $\Lambda$ to $\mathbb{C}$. We denote the application of $\rho$ to $f \in \Lambda$ as $f(\rho)$. One example of a specialization is the {\em trivial } specialization $\varnothing$, which takes the value $1$ at the constant function $1 \in \Lambda$ and the value $0$ at any homogeneous $f \in \Lambda$ of degree $\geq 1$. Since the power sums $p_n$ are algebraically independent generators of $\Lambda$, a specialization $\rho$ is uniquely defined by the numbers $p_n(\rho)$. Conversely, given any sequence $\alpha = a_1, a_2,...$ of complex numbers, we can define a specialization $\rho_\alpha$ by setting $p_n(\rho_\alpha) = a_n$ and linearly extending to the rest of $\Lambda$.

Given two specializations $\rho_1$ and $\rho_2$ we define their union $\rho = (\rho_1, \rho_2)$ as the specialization defined on power sum symmetric functions via
$$p_n(\rho_1, \rho_2) = p_n(\rho_1) + p_n(\rho_2), \hspace{2mm} n\geq 1.$$

One specialization that we will consider frequently is of the form $x_1 = a_1, ..., x_N = a_N$ and $x_k = 0$ for $k > N$, where $a_1,...,a_N$ are given complex numbers. That is, we set
$$p_n = \sum_{i = 1}^N a_i^n \mbox{ for all } n\in \mathbb{N}.$$
Notice that the above is well defined even if $N = \infty$, provided that $\sum_i |a_i|^n < \infty$ for each $n\geq 1$, which is ensured if $\sum_i |a_i| < \infty$. If $N < \infty$ we call the above a {\em finite length} specialization.
\begin{definition}
We say that a specialization $\rho$ of $\Lambda$ is Macdonald nonnegative (or just `nonnegative') if it takes nonnegative values on the skew Macdonald symmetric functions: $P_{\lambda/\mu}(\rho) \geq 0$ for any partitions $\lambda$ and $\mu$.
\end{definition}

One can show (see e.g. Section 2.2 in \cite{BorCor}) that if we have $a_i \geq 0$ and $\sum_i a_i < \infty$ in the specialization we considered before, then it is nonnegative. Such a specialization is called {\em Pure alpha}. We remark that finite unions of nonnegative specializations are nonnegative (see Section 2.2 in \cite{BorCor}).

Let $\rho_1$ and $\rho_2$ be two non-negative specializations, then one defines
$$\Pi(\rho_1, \rho_2) = \sum_{\lambda \in \mathbb{Y}}P_{\lambda}(\rho_1)Q_{\lambda}(\rho_2),$$
the latter being well-defined in $[1,\infty]$ (observe that $P_{\varnothing}(\rho_1) = 1 = Q_{\varnothing}(\rho_2)$, so that $\Pi(\rho_1, \rho_2) \geq 1$).

We now formulate the definition of the {\em Macdonald process}. Let $N$ be a natural number and fix nonnegative specializations $\rho_0^+,...,\rho_{N-1}^+, \rho_1^{-}, ...,\rho_{N}^{-}$, such that $\Pi(\rho_i^+, \rho_j^-) < \infty$ for all $i,j$. Consider two sequences of partitions $\lambda = (\lambda^{1}, ..., \lambda^{N})$ and $\mu = (\mu^{1},...,\mu^{N-1})$. We define their weight as
\begin{equation}\label{weight}
\mathcal{W}(\lambda, \mu) = P_{\lambda^{1}}(\rho_0^+)Q_{\lambda^{1}/\mu^{1}}(\rho_1^{-})P_{\lambda^{2}/\mu^{1}}(\rho_1^{+})  \cdots P_{\lambda^{N}/\mu^{N-1}}(\rho_{N-1}^{+})Q_{\lambda^{N}}(\rho_N^-).
\end{equation}

\begin{definition}
With the above notation, the Macdonald process ${\bf M}(\rho_0^+, ...,\rho_{N-1}^+; \rho_1^-,...,\rho_N^-)$ is the probability measure on sequences $(\lambda, \mu)$, given by
$${\bf M}(\rho_0^+, ...,\rho_{N-1}^+; \rho_1^-,...,\rho_N^-)(\lambda, \mu) = \frac{\mathcal{W}(\lambda, \mu) }{\prod_{0 \leq i < j \leq N}\Pi (\rho_i^+; \rho_j^-)}.$$
\end{definition}
Using properties of Macdonald symmetric functions one can show (see e.g. Proposition 2.4 in Section 2 of \cite{BorCor}) that the above definition indeed produces a probability measure, that is 
$$\sum_{\lambda, \mu }\mathcal{W}(\lambda, \mu) = \prod_{0 \leq i < j \leq N}\Pi (\rho_i^+; \rho_j^-).$$
The Macdonald process with $N = 1$ is called the {\em Macdonald measure} and is written as ${\bf MM}(\rho^+;\rho^-)$.

One important feature of Macdonald processes is that if we pick out subsequences of $(\lambda, \mu)$, then their distribution is also a Macdonald process (with possibly different specializations). One special case that is important for us is the distribution of $\lambda^{k}$ under projection of the above law. As shown in Section 2 of \cite{BorCor}, $\lambda^{k}$ is distributed according to the Macdonald measure ${\bf MM}(\rho^+_{[0,k-1]}; \rho^-_{[k,N]})$, where $\rho^{\pm}_{[a,b]}$ denotes the union of specializations $\rho^{\pm}_m$, $m = a,...,b$.

%
\subsection{The measure $\mathbb{P}^{r,t}_{HL}$ as a limit of Macdonald processes.}\label{HL} \hspace{2mm} \\

The main object of interest in this paper is a distribution $\mathbb{P}^{r,t}_{HL}$ on plane partitions, depending on two parameters $r,t \in (0,1)$, which satisfies $\mathbb{P}_{HL}(\pi) \propto r^{|\pi|}A_{\pi}(t)$ for a certain explicit polynomial $A_\pi$, depending on the geometry of $\pi$ (see Section \ref{intro} for the details). We explain how this measure arises as a limit of Macdonald processes, in which the parameter $q$ is set to $0$.

We begin by fixing a natural number $N$ and considering sequences of partitions $\lambda^{-N+1}, ..., \lambda^{N-1}$ such that
$$\varnothing \prec \lambda^{-N+1} \prec \cdots \prec \lambda^{-1} \prec \lambda^0 \succ \lambda^1 \succ \cdots \succ \lambda^{N-1} \succ \varnothing.$$
The latter sequences exactly represent the set of plane partitions, whose support lies in a square of size $N$, i.e. the set $\{\pi : \pi_{i,j} = 0$ if $i  >  N$ or $j > N\}$ (see Section \ref{partitions}). We next consider the collection of finite length specializations $\rho^+_n$, $\rho_n^-$ given by
\begin{equation*}
\begin{split}
&\rho_n^+ : x_1 = r^{-n -1/2}, x_2 = x_3 = \cdots = 0 \hspace{5mm} -N \leq n \leq -1, \\
&\rho_n^- : x_1 =  x_2 = x_3 = \cdots = 0 \hspace{13mm} -N + 1 \leq n \leq -1, \\
&\rho_n^- : x_1 = r^{n +1/2}, x_2 = x_3 = \cdots = 0 \hspace{14mm} 0 \leq n \leq N-1,\\
&\rho_n^+ : x_1 =  x_2 = x_3 = \cdots = 0 \hspace{26.5mm} 0 \leq n \leq N-2.
\end{split}
\end{equation*}
Consider the Macdonald process ${\bf M}(\rho^+_{-N},... \rho^+_{N-2}; \rho^-_{-N+1},... \rho^-_{N-1})$ and recall that the probability of a pair of sequences $(\lambda, \mu)$ with $\lambda = (\lambda^{-N+1}, \cdots, \lambda^{N - 1})$ and $\mu = (\mu^{-N+1}, \cdots, \mu^{N-2})$ is given by
\begin{equation}\label{measureMM}
{\bf M}(\rho^+_{-N},... \rho^+_{N-2}; \rho^-_{-N+1},... \rho^-_{N-1}) =\frac{ \prod_{n = -N + 1}^{N-1} P_{\lambda^n/\mu^{n-1}}(\rho_{n-1}^+)Q_{\lambda^n/\mu^{n}}(\rho_n^-)}{\prod_{-N+1 \leq i < j \leq N-2}\Pi (\rho_i^+; \rho_j^-)},
\end{equation}
where we set $\mu^{-N} = \mu^{N-1} = \varnothing$. Using properties of skew Macdonald polynomials we see that the above product is zero unless 
\begin{itemize}
\item $\mu^n = \lambda^n$ for $n < 0$ and $\mu^n = \lambda^{n+1}$ for $n\geq 0$,
\item $\varnothing \prec \lambda^{-N+1} \prec \cdots \prec \lambda^{-1} \prec \lambda^0 \succ \lambda^1 \succ \cdots \succ \lambda^{N-1} \succ \varnothing.$
\end{itemize}
If the two conditions above are satisfied we obtain that the numerator in (\ref{measureMM}) equals (see Section \ref{MacSym})
$$\prod_{n = -N + 1}^0 \psi_{\lambda^n/ \mu^{n-1}}(q,t)r^{(-2n + 1)(|\lambda^n| - |\mu^{n-1}|)/2}\times \prod_{n = 1}^N \phi_{\lambda^{n-1}/ \mu^{n-1}}(q,t)r^{(2n + 1)(|\lambda^{n-1}| - |\mu^{n-1}|)/2}.$$
Using that $\mu^n = \lambda^n$ for $n < 0$ and $\mu^n = \lambda^{n+1}$ for $n\geq 0$ we get
$$\sum_{n = -N+1}^0 \frac{-2n + 1}{2}(|\lambda^n| - |\mu^{n-1}|) + \sum_{n =1}^N \frac{2n + 1}{2}(|\lambda^{n-1}| - |\mu^{n-1}|) = \sum_{n = -N+1}^0 \frac{-2n + 1}{2}(|\lambda^n| - |\lambda^{n-1}|) $$ 
$$ + \sum_{n =1}^N \frac{2n + 1}{2}(|\lambda^{n-1}| - |\lambda^{n}|) = \sum_{n = - N + 1}^{-1} |\lambda^n| + \frac{1}{2}|\lambda^0| +  \frac{1}{2}|\lambda^0| + \sum_{n = 1}^{N-1}|\lambda^n| = \sum_{n = - N + 1}^{N-1} |\lambda^n| = |\pi|$$
where we set $\lambda^{-N} = \lambda^N = \varnothing$ and $\pi$ is the plane partition corresponding to the diagonal slices $\lambda^n$ (see Section \ref{partitions}). 

Letting $q \rightarrow 0$ in equations (\ref{phiForm}) and (\ref{psiForm}) we get (see $(5.8)$ and $(5.8')$ in Chapter 3 of \cite{Mac}):
$$\phi_{\lambda / \mu}(0,t) = \prod_{i \in I}(1 - t^{m_i(\lambda)}) \hspace{3mm} \mbox{ and }  \hspace{3mm} \psi_{\lambda/\mu}(0,t) = \prod_{j \in J}(1 - t^{m_j(\mu)}).$$
In the above formula we assume $\lambda \succ \mu$ otherwise both expressions equal $0$. The sets $I,J$ are:
$$I = \{ i \in \mathbb{N}: \lambda'_{i+1} = \mu'_{i+1} \mbox{ and }\lambda'_{i} > \mu'_{i}\} \mbox{ and }J = \{ j \in \mathbb{N}: \lambda'_{j+1} > \mu'_{j+1} \mbox{ and }\lambda'_{j} = \mu'_{j}\}.$$
Summarizing the above work, we see that ${\bf M}(\rho^+_{-N},... \rho^+_{N-2}; \rho^-_{-N+1},... \rho^-_{N-1})$ induces a probability measure on sequences $\varnothing \prec \lambda^{-N+1} \prec \cdots \prec \lambda^{-1} \prec \lambda^0 \succ \lambda^1 \succ \cdots \succ \lambda^{N-1} \succ \varnothing$ and hence on plane partitions $\pi$, whose support lies in the square of size $N$. Call the latter measure $\mathbb{P}^{r,t,N}_{HL}$ and observe that
$$\mathbb{P}^{r,t,N}_{HL}(\pi) =Z_N^{-1} r^{|\pi|} \prod_{n = -N + 1}^0\psi_{\lambda^n/\lambda^{n-1}}(0,t) \times \prod_{n = 1}^N \phi_{\lambda^{n-1}/ \lambda^{n}}(0,t) = Z_N^{-1}r^{|\pi|}B_{\pi}(t),$$
where $B_\pi(t)$ is an integer polynomial in $t$ and $Z_N$ is a normalizing constant. In \cite{Vul} it was shown that $B_\pi(t) = A_\pi(t)$ and the normalizing constant was evaluated to equal
$$Z_N(r,t) = \prod_{i = 1}^N\prod_{j = 1}^N \frac{1 - t r^{i + j - 1}}{1 - r^{i + j - 1}}.$$

\begin{remark}
The ``HL'' in our notation stands for Hall-Littlewood, since in the limit $q \rightarrow 0$ the Macdonald symmetric functions $P_\lambda(X;q,t)$ and $Q_{\lambda}(X;q,t)$ degenerate to the Hall-Littlewood symmetric functions $P_\lambda(X;t)$ and $Q_\lambda(X;t)$.
\end{remark}

As $N \rightarrow \infty$ the measures $\mathbb{P}^{r,t,N}_{HL}$ converge to the measure $\mathbb{P}^{r,t}_{HL}$ since $\lim_{N \rightarrow \infty} Z_N(r,t) = Z(r,t)$ - the normalizing constant in the definition of $\mathbb{P}^{r,t}_{HL}$ (see (\ref{Zconst})). Thus, we indeed see that $\mathbb{P}^{r,t}_{HL}$ arises as a limit of Macdonald processes, in which the parameter $q$ is set to $0$.\\

Our approach of studying $\mathbb{P}^{r,t}_{HL}$ goes through understanding the distribution of the diagonal slices $\lambda^k$. For $N > |k|$ we have that 
$$\mathbb{P}_{HL}^{r,t,N}(\lambda^k = \lambda) = Z_N^{-1} P_{\lambda}(r^{1/2},\cdots,r^{(2N-1)/2};t)Q_{\lambda}(r^{1/2+|k|},\cdots,r^{(2N-1)/2},\underbrace{0,0,...,0}_\text{|k|};t),
$$
where we used results in Section \ref{MacProc} and the proportionality of $P_\lambda$ and $Q_\lambda$ to combine the cases $k \geq 0$ and $k < 0$. Letting $N \rightarrow \infty$ we conclude that
$$\mathbb{P}^{r,t}_{HL}(\lambda^k = \lambda) = Z(r,t)^{-1} P_{\lambda}(r^{1/2},r^{3/2}, \cdots ;t)Q_{\lambda}(r^{1/2+|k|}, r^{3/2 + |k|}, \cdots ;t).$$
Finally, using the homogeneity of $P_\lambda$ and $Q_\lambda$, we see that 
$$\mathbb{P}^{r,t}_{HL}(\lambda^k = \lambda) = Z(r,t)^{-1}P_{\lambda}(a,ar,ar^2, \cdots ;t)Q_{\lambda}(a,ar, ar^2, \cdots ;t),$$
where $a(k) = r^{(1 + |k|)/2}$. It is this distribution, which we call the {\em Hall-Littlewood measure} with parameters $a,r,t \in (0,1)$,  that we will analyze in subsequent sections.

%
\subsection{Background on Fredholm determinants}\label{FredholmS}\hspace{2mm} \\

We present a brief background on Fredholm determinants. For a general overview of the theory of Fredholm determinants, the reader is referred to \cite{Simon} and \cite{Lax}. For our purposes the definition below is sufficient and we will not require additional properties. 

\begin{definition}
Fix a Hilbert space $L^2(X, \mu)$, where $X$ is a measure space and $\mu$ is a measure on $X$. When $X = \Gamma$, a simple (anticlockwise oriented) smooth contour in $\mathbb{C}$ we write $L^2(\Gamma)$ where for $z \in \Gamma$, $d\mu(z)$ is understood to be $\frac{dz}{2\pi \iota }$. 
\end{definition}
Let $K$ be an {\em integral operator} acting on $f(\cdot) \in L^2(X,\mu)$ by $Kf(x) = \int_XK(x,y)f(y)d\mu(y)$.  $K(x,y)$ is called the {\em kernel} of $K$ and we assume throughout $K(x,y)$ is continuous in both $x$ and $y$. If $K$ is a {\em trace-class} operator then one defines the Fredholm determinant of $I + K$, where $I$ is the identity operator, via
\begin{equation}\label{fredholmDef}
\det(I + K)_{L^2(X)} = 1 + \sum_{n = 1}^\infty \frac{1}{n!} \int_X \cdots \int_X \det \left[ K(x_i, x_j)\right]_{i,j = 1}^n \prod_{i = 1}^nd\mu(x_i),
\end{equation}
where the latter sum can be shown to be absolutely convergent (see \cite{Simon}). 

A sufficient condition for the operator $K(x,y)$ to be trace-class is the following (see \cite{Lax} page 345).
\begin{lemma}\label{traceclass}
An operator $K$ acting on $L^2(\Gamma)$ for a simple smooth contour $\Gamma$ in $\mathbb{C}$ with integral kernel $K(x,y)$ is trace-class if $K(x,y): \Gamma^2 \rightarrow \mathbb{R}$ is continuous as well as $K_2(x,y)$ is continuous in $y$. Here $K_2(x,y)$ is the derivative of $K(x,y)$ along the contour $\Gamma$ in the second entry.
\end{lemma}

The expression appearing on the RHS of (\ref{fredholmDef}) can be absolutely convergent even if $K$ is not trace-class. In particular, this is so if $X = \Gamma$ is a piecewise smooth, oriented compact contour and $K(x,y)$ is continuous on $X \times X$. Let us check the latter briefly. 

Since $K(x,y)$ is continuous on $X \times X$, which is compact, we have $|K(x,y)| \leq A$ for some constant $A > 0$, independent of $x,y \in X$. Then by Hadamard's inequality\footnote{Hadamard's inequality: the absolute value of the determinant of an $n \times n$ matrix is at most the product of the lengths of the column vectors.} we have
$$\left|  \det \left[ K(x_i, x_j)\right]_{i,j = 1}^n \right| \leq n^{n/2}A^n.$$
This implies that 
$$ \left| \frac{1}{n!} \int_X \cdots \int_X \det \left[ K(x_i, x_j)\right]_{i,j = 1}^n \prod_{i = 1}^nd\mu(x_i) \right| \leq \frac{n^{n/2}B^n}{n!},$$
where $B = A |\mu|(X)$. The latter is clearly absolutely summable because of the $n!$ in the denominator.\\

Whenever $X$ and $K$ are such that the RHS in (\ref{fredholmDef}) is absolutely convergent, we will still call it $\det (I + K)_{L^2(X)}$. The latter is no longer a Fredholm determinant, but some numeric quantity we attach to the kernel $K$. Of course, if $K$ is the kernel of a trace-class operator on $L^2(X)$ this numeric quantity agrees with the Fredholm determinant. Doing this allows us to work on the level of numbers throughout most of the paper, and avoid constantly checking if the kernels we use represent a trace-class operator.\\

The following lemmas provide a framework for proving convergence of Fredholm determinants, based on pointwise convergence of their defining kernels and estimates on those kernels.
\begin{lemma}\label{FDCT}
Suppose that $\Gamma$ is a piecewise smooth contour in $\mathbb{C}$ and $K^N(x,y)$, $N \in \mathbb{N}$ or $N = \infty$, are measurable kernels on $\Gamma \times \Gamma$ such that  $\lim_{N \rightarrow \infty} K^N(x,y) = K^\infty(x,y)$ for all $x,y\in \Gamma$. In addition, suppose that there exists a non-negative, measurable function $F(x)$ on $\Gamma$ such that
$$\sup_{N \in \mathbb{N}}\sup_{y\in \Gamma}|K^N(x,y)| \leq F(x) \mbox{ and } \int_{\Gamma}F(x) |d\mu(x)| = M < \infty.$$ 
Then for each $n \geq 1$ and $N$ one has that $\det \left[ K^N(x_i, x_j)\right]_{i,j = 1}^n$ is integrable on $\Gamma^n$, so that in particular $\int_{\Gamma} \cdots \int_{\Gamma} \det \left[ K^N(x_i, x_j)\right]_{i,j = 1}^n \prod_{i = 1}^nd\mu(x_i)$ is well defined. Moreover, for each $N$
$$\det(I + K^N)_{L^2(\Gamma)} = 1 +\sum_{n = 1}^\infty \frac{1}{n!} \int_{\Gamma} \cdots \int_{\Gamma} \det \left[ K^N(x_i, x_j)\right]_{i,j = 1}^n \prod_{i = 1}^nd\mu(x_i)$$
is absolutely convergent and $\lim_{N \rightarrow \infty} \det(I + K^N)_{L^2(\Gamma)} = \det(I + K^\infty)_{L^2(\Gamma)}.$
\end{lemma}
\begin{proof}
The following is similar to Lemma 8.5 in \cite{BCF}; however, it allows for infinite contours $\Gamma$ and assumes a weaker pointwise convergence of the kernels, while requiring a dominating function $F$. The idea is to use the Dominated Convergence Theorem multiple times.\\

Since $\lim_{N \rightarrow \infty} K^N(x,y) = K^\infty(x,y)$ we know that $\sup_{y\in \Gamma}|K^\infty(x,y)| \leq F(x) $ and also
$$\lim_{N \rightarrow \infty}  \det \left[ K^N(x_i, x_j)\right]_{i,j = 1}^n =  \det \left[ K^\infty(x_i, x_j)\right]_{i,j = 1}^n \mbox{ for all }x_1,...,x_n \in \Gamma.$$
By Hadamard's inequality we have
$$\left| \det \left[ K^N(x_i, x_j)\right]_{i,j = 1}^n\right| \leq n^{n/2} \prod_{i = 1}^nF(x_i),$$
which is integrable by assumption. It follows from the Dominated Convergence Theorem with dominating function $ n^{n/2} \prod_{i = 1}^nF(x_i)$ that for each $n \geq 1$ one has
$$\lim_{N \rightarrow \infty}  \int_{\Gamma} \cdots \int_{\Gamma} \det \left[ K^N(x_i, x_j)\right]_{i,j = 1}^n \prod_{i = 1}^nd\mu(x_i) = \int_{\Gamma} \cdots \int_{\Gamma} \det \left[ K^\infty(x_i, x_j)\right]_{i,j = 1}^n \prod_{i = 1}^nd\mu(x_i).$$

Next observe that 
$$\left| \int_{\Gamma} \cdots \int_{\Gamma} \det \left[ K^N(x_i, x_j)\right]_{i,j = 1}^n \prod_{i = 1}^nd\mu(x_i)\right| \leq  \int_{\Gamma} \cdots \int_{\Gamma} \left| \det \left[ K^N(x_i, x_j)\right]_{i,j = 1}^n \right| \prod_{i = 1}^n|d\mu(x_i)| \leq n^{n/2}M^n.$$
The latter shows the absolute convergence of the series, defining $\det(I + K^N)_{L^2(\Gamma)}$ for each $N$. A second application of the Dominated Convergence Theorem with dominating series $1 + \sum_{n \geq 1}\frac{n^{n/2}M^n}{n!}$ now shows the last statement of the lemma.
\end{proof}

\begin{lemma}\label{FDKCT}
Suppose that $\Gamma_1, \Gamma_2$ are piecewise smooth contours and $g^N_{x,y}(z)$ are measurable on $\Gamma_1^2 \times \Gamma_2$ for $N \in \mathbb{N}$ or $N = \infty$ and satisfy $\lim_{N \rightarrow \infty} g^N_{x,y}(z) = g_{x,y}^\infty(z)$ for all $x,y \in \Gamma_1,$ $z \in \Gamma_2$. In addition, suppose that there exist bounded non-negative measurable functions $F_1$ and $F_2$ on $\Gamma_1$ and $\Gamma_2$ respectively such that
$$\sup_{N \in \mathbb{N}}\sup_{y\in \Gamma_1}|g^N_{x,y}(z)| \leq F_1(x)F_2(z), \mbox{ and } \int_{\Gamma_i}F_i(u)|d\mu(u)| = M_i < \infty.$$ 
Then for each $N$ one has $\int_{\Gamma_2}|g^N_{x,y}(z)| |d\mu(z)|< \infty$ and in particular $K^N(x,y):= \int_{\Gamma_2}g^N_{x,y}(z) d\mu(z)$ are well-defined. Moreover, $K^N(x,y)$ satisfy the conditions of Lemma \ref{FDCT} with $\Gamma = \Gamma_1$ and $F = M_2 F_1$.
\end{lemma}
\begin{proof}
Since $\lim_{N \rightarrow \infty} g^N_{x,y}(z) = g_{x,y}^\infty(z)$ for all $x,y \in \Gamma_1,$ $z \in \Gamma_2$ we know that $\left|g^\infty_{x,y}(z)\right| \leq F_1(x)F_2(z)$ as well. Observe that for each $x,y \in \Gamma_1$ and $N$ one has that 
$$\int_{\Gamma_2}|g^N_{x,y}(z)| |d\mu(z)| \leq \int_{\Gamma_2}F_1(x)F_2(z)|d\mu(z)| \leq M_2F_1(x) < \infty.$$
Setting $K^N(x,y) = \int_{\Gamma_2}g^N_{x,y}(z)d\mu(z)$, we see that $|K^N(x,y)|  \leq  M_2F_1(x) $ for each $x,y \in \Gamma_1$ and $N$. As an easy consequence of Fubini's Theorem one has that $K^N(x,y)$ is measurable on $\Gamma_1^2$ (the case of real functions and measures $\mu$ can be found in Corollary 3.4.6 of \cite{Bog}, from which the complex extension is immediate). Using the Dominated Convergence Theorem with dominating function $F_1(x)F_2(z)$ we see that $\lim_{N \rightarrow \infty} K^N(x,y) = K^\infty(x,y)$.

\end{proof}

\section{Finite length formulas}\label{finiteS}
In this section, we derive formulas for the $t$-Laplace transform of the random variable $(1-t)t^{-\lambda_1'}$, where $\lambda$ is distributed according to the finite length Hall-Littlewood measure $\mathbb{P}_{X,Y}$ (see Section \ref{HL}). The main result in this section is Proposition \ref{finlength}, which expresses the $t$-Laplace transform as a Fredholm determinant. We believe that such a formula is of separate interest as it can be applied to generic Hall-Littlewood measures and its Fredholm determinant form makes it suitable for asymptotic analysis. The derivation of Proposition \ref{finlength} goes through a sequence of steps that is very similar to the work in Sections 2.2.3, 3.1 and 3.2 of \cite{BorCor}. There are, however, several technical modifications that need to be made, which require us to redo most of the work there. In particular, the statements below do not follow from some simple limit transition from those in \cite{BorCor} and additional work is required. 

In all statements in the remainder of this paper we will be working with the principal branch of the logarithm.

%
\subsection{Observables of Hall-Littlewood measures}\label{ObsHL}\hspace{2mm}\\

In this section we describe a framework for obtaining certain observables of Macdonald measures. Our discussion will be very much in the spirit of section 2.2.3 in \cite{BorCor}; however, the results we need do not directly follow from that work and so we derive them explicitly. In this paper we will be primarily working with finite length specializations, which greatly simplifies the discussion; however, we mention that the results below can be derived in a much more general setting as is done in \cite{BorGor}. Finally, our focus will be on the case when $q= 0$ in the Macdonald measure and we call this degeneration a {\em Hall-Littlewood measure}.

In what follows we fix a natural number $N$ and consider the space of functions in $N$ variables. Inside this space lies the space of symmetric polynomials $\Lambda_X$ in $N$ variables $X = (x_1,...,x_N)$. 
\begin{definition}
For any $u \in \mathbb{R}$ and $1 \leq i \leq n$ define the shift operator $T_{u,x_i}$ by
$$(T_{u,x_i}F)(x_1,...,x_N) := F(x_1,...,ux_i,...,x_N).$$
{\raggedleft For any subset $I \subset \{1,...,N\}$ of size $r$ define}
$$A_{I}(X;t) := t^{\frac{r(r-1)}{2}}\prod_{i \in I, j \not \in I} \frac{tx_i - x_j}{x_i - x_j}.$$
Finally, for any $r = 1, 2, ..., N$ define the Macdonald difference operator
$$D_N^r := \sum_{\substack{ I \subset \{1,...,N\}\\ |I| = r} }A_I(X;t) \prod_{i \in I}T_{q,x_i}.$$
\end{definition}
A key property of the Macdonald difference operators is that they are diagonalized by the Macdonald polynomials $P_\lambda$. Specifically, as shown in Chapter VI (4.15) of \cite{Mac}, we have
\begin{proposition}
For any partition $\lambda$ with $\ell(\lambda) \leq N$ 
$$D_N^rP_\lambda(x_1,...,x_N;q,t) = e_r(q^{\lambda_1} t^{N-1}, q^{\lambda_2}t^{N-2},..., q^{\lambda_N})P_\lambda(x_1,...,x_N;q,t),$$
where $e_r$ denote the elementary symmetric functions (see Section \ref{MacSym}).
\end{proposition}
{\raggedleft In particular, we see that }
$$D_N^1P_\lambda(x_1,...,x_N;q,t)  = \left( \sum_{i = 1}^N q^{\lambda_i}t^{N-i}\right) P_\lambda(x_1,...,x_N;q,t).$$ 

We now let $q \rightarrow 0$, while $t \in (0,1)$ is still fixed. In this limiting regime the Macdonald polynomials $P_\lambda(X;q,t)$ degenerate to the Hall-Littlewood polynomials $P_\lambda(X;t)$. In addition, the Macdonald difference operator $D_N^1$ degenerates to (we use the same notation)
$$D_N^1 =\sum_{i = 1}^N \prod_{j \neq i}\frac{tx_i - x_j}{x_i - x_j} T_{0,x_i}, \mbox{ and also }\sum_{i = 1}^N q^{\lambda_i}t^{N-i} \rightarrow t^{N-\lambda_1' -1} + \cdots + t^0 = \frac{1 - t^{N - \lambda_1'}}{1 - t}.$$
$D_N^1$ is still an operator on the space of functions in $N$ variables and we summarize the properties that we will need:
\begin{equation*}
\begin{split}1.& \mbox{ $D_N^1$ is linear.}  \\ 2. & \mbox{ If $F_n$ converge pointwise to a function $F$ in $N$ variables, then $D_N^1F_n$ converge pointwise to}\\ & \mbox{ $D_N^1F$ away from the set $\{ (x_1,...,x_N): x_i = x_j \mbox{ for some }i \neq j\}$ }. \\  3. & \mbox{ }D_N^1P_\lambda(x_1,...,x_N;t) = \frac{1 - t^{N - \lambda_1'}}{1 - t}  P_\lambda(x_1,...,x_N;t).\end{split}
\end{equation*}

\begin{proposition}
Assume that $F(u_1,...,u_N) = f(u_1)\cdots f(u_N)$ with $f(0) = 1$. Take $x_1, ... , x_N > 0$ and assume that $f(u)$ is holomorphic and non-zero in a complex neighborhood of an interval in $\mathbb{R}$ that contains $x_1,...,x_N$. Then we have
\begin{equation}\label{babyContour}
(D^1_NF)(x_1,...,x_N) = \frac{F(x_1,...,x_N)}{2\pi \iota } \int_{C} \prod_{ j = 1}^N \frac{tz - x_j}{z - x_j} \frac{1}{f(z)} \frac{dz}{(t-1)z},
\end{equation}
where $C$ is a positively oriented contour encircling $\{x_1,...,x_N\}$ and no other singularities of  the integrand.
\end{proposition}
\begin{proof}
The following proof is very similar to the proof of Proposition 2.11 in \cite{BorCor}. First observe that from $t \in (0,1)$ and our assumptions on $f$ a contour $C$ will always exist. Using continuity of both sides in the variables $x_1,...,x_N$ it suffices to prove the above when the $x_i$ are pairwise distinct. The contour encircles the simple poles at $x_1,...,x_N$ and the residue at $x_i$ equals
$$\prod_{ j \neq i}^N \frac{tx_i - x_j}{x_i - x_j}\frac{1}{f(x_i)}.$$
Using the Residue Theorem we conclude that the RHS of (\ref{babyContour}) equals
$$\sum_{i = 1}^N F(x_1,...,x_N)\prod_{ j \neq i}^N \frac{tx_i - x_j}{x_i - x_j}\frac{1}{f(x_i)} = \sum_{i = 1}^N \prod_{ j \neq i}^N \frac{tx_i - x_j}{x_i - x_j}f(x_j) = (D^1_NF)(x_1,...,x_N) .$$
\end{proof}

We next consider the operator $\mathcal{D}_N = \left[ \frac{(t-1)D_N^1 + 1}{t^{N}}\right]$. It satisfies Properties 1. and 2. above and Property 3. is replaced by
$$3.' \mbox{ }\mathcal{D}_NP_\lambda(x_1,...,x_N;t) = t^{-\lambda_1'}P_\lambda(x_1,...,x_N;t).$$

\begin{proposition}\label{babyContProp}
Assume that $F(u_1,...,u_N) = f(u_1)\cdots f(u_N)$ with $f(0) = 1$. Take $x_1, ... , x_N > 0$ and assume that $f(u)$ is holomorphic and non-zero in a complex neighborhood $D$ of an interval in $\mathbb{R}$ that contans $x_1,...,x_N$ and $0$. Then for any $k \geq 1$ we have
\begin{equation}\label{toddlerContour}
 (\mathcal{D}_N^k F)(x_1,...,x_N) = \frac{F(x_1,...,x_N)}{(2\pi \iota )^k} \hspace{-1mm}\int_{C_{0,1}} \hspace{-3mm}\cdots \int_{C_{0,k}} \hspace{-1mm}\prod_{1 \leq a < b \leq k}\hspace{-1mm}  \frac{z_a - z_b}{z_a - z_bt^{-1}}  \prod_{i =1}^k \hspace{-1mm}\left[ \prod_{ j = 1}^N\hspace{-1mm} \frac{z_i - x_jt^{-1}}{z_i - x_j} \right]\hspace{-2mm} \frac{1}{f(z_i)} \frac{dz_i}{z_i},
\end{equation}
where $C_{0,a}$ are positively oriented simple contours encircling $x_1,...,x_N$ and $0$ and no zeros of $f(z)$. In addition, $C_{0,a}$ contains $t^{-1}C_{0,b}$ for $a < b$ and $C_{0,1} \subset D$.
\end{proposition}
\begin{proof}
The proof is similar to the proof of Proposition 2.14 in \cite{BorCor}. In this proposition the existence of the contours $C_{0,a}$ depends on the properties of the function $f$. In what follows we will assume that they exist and whenever we use this result in the future with a particular function $f$ we will provide explicit contours satisfying the conditions in the proposition.\\

Using the continuity of both sides in $x_1,...,x_N$ it suffices to show the result when the $x_i$ are pairwise distinct. We now proceed by induction on $k \in \mathbb{N}$.\\

{\raggedleft {\em Base case: }$ k = 1$}. The RHS of (\ref{toddlerContour}) equals
$$\frac{F(x_1,...,x_N)}{2\pi \iota} \int_{C_{0,1}}  \left[ \prod_{ j = 1}^N \frac{z_1 - x_jt^{-1}}{z_1 - x_j} \right] \frac{1}{f(z_1)} \frac{dz_1}{z_1}.$$
The contour $C_{0,1}$ encircles the simple poles of the integrand at $x_1,...,x_N$ and $0$ and the residue at $0$ equals $t^{-N}$ (using $f(0) = 1$). If we now deform $C_{0,1}$ to a contour $C$, which no longer encircles $0$ but does encirlce $x_1,...,x_N$ we see, using the Residue Theorem, that the RHS of (\ref{toddlerContour}) equals
$$ t^{-N}F(x_1,...,x_N) +  \frac{F(x_1,...,x_N)}{2\pi \iota } \int_{C}  \left[ \prod_{ j = 1}^N \frac{z_1 - x_jt^{-1}}{z_1 - x_j} \right] \frac{1}{f(z_1)} \frac{dz_1}{z_1} = t^{-N}F(x_1,...,x_N) + $$
$$  (t-1)t^{-N}\frac{F(x_1,...,x_N)}{2\pi  \iota }\int_{C}  \left[ \prod_{ j = 1}^N \frac{tz_1 - x_j}{z_1 - x_j} \right] \frac{1}{f(z_1)} \frac{dz_1}{z_1(t-1)} = (\mathcal{D}_NF)(x_1,...,x_N).$$
In the last equality we used Proposition \ref{babyContProp} and the definition of $\mathcal{D}_N$. This proves the base case.\\

We next suppose that the result holds for $k \geq 1$ and wish to prove it for $k+1$. In particular, we have
\begin{equation*}
\begin{split} (\mathcal{D}_N^k F)(x_1,...,x_N) = &\frac{1}{(2\pi  \iota)^k} \int_{C_{0,1}} \cdots \int_{C_{0,k}} \prod_{j =1}^N g(x_j;z_1,...,z_k) \prod_{1 \leq a < b \leq k}  \frac{z_a - z_b}{z_a - z_bt^{-1}} \prod_{i =1}^k  \frac{1}{f(z_i)} \frac{dz_i}{z_i},\end{split}
\end{equation*}
where $g(u;z_1,...,z_k) = f(u) \prod_{ i = 1}^k \frac{z_i - ut^{-1}}{z_i - u}.$

We apply $\mathcal{D}_N$ to both sides in the above expression and observe we may switch the order of $\mathcal{D}_N$ and the integrals on the RHS. To see the latter, one may approximate the integrals by Riemann sums and use Property $1.$ of $\mathcal{D}_N$ to switch the order of the sums and the operator. Subsequently, one may use Property $2.$ to show that the change of the order also holds in the limit. We thus obtain
\begin{equation*}
\begin{split} (\mathcal{D}_N^{k + 1} F)(x_1,...,x_N) = &\frac{1}{(2\pi \iota )^k} \int_{C_{0,1}}\hspace{-4mm} \cdots \int_{C_{0,k}} \hspace{-2mm} (\mathcal{D}_NG)(x_1,...,x_N; z_1,...,z_k ) \prod_{1 \leq a < b \leq k}  \frac{z_a - z_b}{z_a - z_bt^{-1}} \prod_{i =1}^k  \frac{1}{f(z_i)} \frac{dz_i}{z_i},\end{split}
\end{equation*}
where $G(x_1,...,x_N; z_1,...,z_k ) = \prod_{j = 1}^N g(x_j;z_1,...,z_k).$ We now wish to apply the base case to the function $G$. Notice that $g(0) = 1$ and the zeros of $g(u)$ coincide with those of $f(u)$ except that it has additional zeros at $tz_i$ for $i = 1,...,k$. By assumption $tC_{0,i}$ contain $C_{0,k+1}$ for all $i = 1,...,k$ so the additional zeros of $g(u)$ are not contained in $C_{0,k+1}$, while $x_1,...,x_N$ and $0$ are. Thus the Base case is applicable and we conclude that 
\begin{equation*}
\begin{split} (\mathcal{D}_N^{k + 1} F)(x_1,...,x_N) = &\frac{1}{(2\pi \iota )^{k+1}} \int_{C_{0,1}} \cdots \int_{C_{0,k}} \int_{C_{0,k+1}} G(x_1,...,x_N; z_1,...,z_k ) \left[ \prod_{ j = 1}^N \frac{z_{k+1} - x_jt^{-1}}{z_{k+1} - x_j} \right] \times \\ & \frac{1}{g(z_{k+1};z_1,...,z_k)} \prod_{1 \leq a < b \leq k} \frac{z_a - z_b}{z_a - z_bt^{-1}} \frac{dz_{k+1}}{z_{k+1}}\prod_{i =1}^k  \frac{1}{f(z_i)} \frac{dz_i}{z_i},\end{split}
\end{equation*}
Expressing $g(z_{k+1};z_1,...,z_k)$ and $G(x_1,...,x_N; z_1,...,z_k )$ in terms of $f(z_i)$ and $F(x_1,...,x_N)$ we arrive at
\begin{equation*}
 (\mathcal{D}_N^{k+1} F)(x_1,...,x_N) = \frac{F(x_1,...,x_N)}{(2\pi \iota)^{k+1}}\hspace{-2mm} \int_{C_{0,1}}\hspace{-4mm} \cdots \int_{C_{0,k+1}}  \prod_{1 \leq a < b \leq k+1}  \frac{z_a - z_b}{z_a - z_bt^{-1}}\prod_{i =1}^{k+1} \hspace{-1mm}\left[ \prod_{ j = 1}^N \hspace{-1mm}\frac{z_i - x_jt^{-1}}{z_i - x_j} \right] \hspace{-2mm}\frac{1}{f(z_i)} \frac{dz_i}{z_i}.
\end{equation*}
This concludes the proof of the case $k+1$. The general result now proceeds by induction.
\end{proof}

Let $\rho_X$ and $\rho_Y$ be the nonnegative finite length specializations in $N$ variables $X = (x_1,...,x_N)$ and $Y = (y_1,...,y_N)$ respectively, with $x_i, y_i \in (0,1)$ for $i = 1, ... , N$. We consider the Macdonald measure ${\bf MM}(\rho_X; \rho_Y)$ with parameter $q = 0$ and denote the probability distribution and expectation with respect to this measure by $\mathbb{P}_{X,Y}$ and $\mathbb{E}_{X,Y}$. Using the Cauchy identity (see equation (\ref{CauchyI})) with $q = 0$ we get
\begin{equation}\label{preOp}
\sum_{\lambda \in \mathbb{Y}} P_\lambda(x_1,...,x_N; t) Q_\lambda(y_1,...,y_N;t) =\hspace{-1mm} \prod_{i,j = 1}^N\hspace{-2mm}\frac{1 - tx_iy_j}{1 -x_iy_j} = \prod_{i = 1}^N f_Y(x_i) \mbox{ with $f_Y(u) = \prod_{j = 1}^N\frac{1 - tuy_j}{1 -uy_j}$.}
\end{equation}

We want to apply $\mathcal{D}^k_N$ in the $X$ variable to both sides of (\ref{preOp}). We observe that the sum on the LHS is absolutely convergent so from Properties $1.$ and $2.$ we see that
\begin{equation}\label{LHS}
\mathcal{D}^k_N\sum_{\lambda \in \mathbb{Y}} P_\lambda(X; t) Q_\lambda(Y;t)  = \sum_{\lambda \in \mathbb{Y}} \mathcal{D}^k_NP_\lambda(X; t) Q_\lambda(Y;t)  = \sum_{\lambda \in \mathbb{Y}} t^{-k\lambda_1'}P_\lambda(X; t) Q_\lambda(Y;t),
\end{equation}
where in the last equality we used Property $3.'$ $k$ times. We remark that the latter sum is absolutely convergent as well, since $\lambda_1' \leq N$ on the support of $\mathbb{P}_{X,Y}$.

On the other hand, the RHS of (\ref{preOp}) satisfies the conditions of Proposition \ref{babyContProp} and in order to apply it we need to find suitable contours. The contours will exist provided $y_i$ are sufficiently small. So suppose $y_i < \epsilon \leq t^k$ for all $i$ and observe that the zeros of $f_Y(u)$, which are at $t^{-1}y_i^{-1}$, lie outside the circle of radius $\epsilon^{-1}t^{-1}$ around the origin. Let $C_{0,k}$ be the positively oriented circle around the origin of radius $1$ and let $C_{0,a}$ be positively oriented circles of radius slightly bigger than $t^{a-k}$, so that $C_{0,a}$ contains $t^{-1}C_{0,b}$ for all $a < b$ and $C_{0,1}$ has radius less than $\epsilon^{-1}$. Clearly such contours exist and satisfy the conditions of Proposition \ref{babyContProp}. Consequently, we obtain
\begin{equation}\label{RHS}
\mathcal{D}^k \prod_{i = 1}^N f_Y(x_i) = \frac{\prod_{i = 1}^N f_Y(x_i)}{(2\pi\iota )^k} \int_{C_{0,1}} \cdots \int_{C_{0,k}} \prod_{1 \leq a < b \leq k}  \frac{z_a - z_b}{z_a - z_bt^{-1}} \prod_{i =1}^k \left[ \prod_{ j = 1}^N \frac{z_i - x_jt^{-1}}{z_i - x_j} \right] \frac{dz_i}{f_Y(z_i)z_i}
\end{equation}

Equating the expressions in (\ref{LHS}) and (\ref{RHS}) and dividing by $\prod_{i = 1}^N f_Y(x_i)$ we arrive at
$$\sum_{\lambda \in \mathbb{Y}} t^{-k\lambda_1'}\frac{P_\lambda(X; t) Q_\lambda(Y;t)}{\Pi(X;Y)} =  \frac{1}{(2\pi \iota )^k} \int_{C_{0,1}} \hspace{-3mm}\cdots \int_{C_{0,k}} \prod_{1 \leq a < b \leq k}  \frac{z_a - z_b}{z_a - z_bt^{-1}} \prod_{i =1}^k \left[ \prod_{ j = 1}^N \frac{z_i - x_jt^{-1}}{z_i - x_j} \right]\hspace{-2mm} \frac{dz_i}{f_Y(z_i)z_i},$$
in which we recognize the LHS as $\mathbb{E}_{X,Y}\left[ t^{-k\lambda_1'}\right]$. We isolate the above result in a proposition.

\begin{proposition}\label{milestone1}
Fix positive integers $k$ and $N$ and a parameter $t\in(0,1)$. Let $\rho_X$ and $\rho_Y$ be the nonnegative finite length specializations in $N$ variables $X = (x_1,...,x_N)$ and $Y = (y_1,...,y_N)$ respectively, with $x_i, y_i \in (0,1)$ for $i = 1, ... , N$. In addition, suppose $y_i < \epsilon$ for all $i$. Then we have
$$\mathbb{E}_{X,Y}\left[ t^{-k\lambda_1'}\right] = \frac{1}{(2\pi \iota )^k} \int_{C_{0,1}} \cdots \int_{C_{0,k}} \prod_{1 \leq a < b \leq k}  \frac{z_a - z_b}{z_a - z_bt^{-1}} \prod_{i =1}^k \left[ \prod_{ j = 1}^N \frac{(z_i - x_jt^{-1})(1 - z_iy_j)}{(z_i - x_j)(1 - tz_iy_j)} \right] \frac{dz_i}{z_i},$$
where $C_{0,a}$ are positively oriented simple contours encircling $x_1,...,x_N$ and $0$ and contained in a disk of radius $\epsilon^{-1}$ around $0$. In addition, $C_{0,a}$ contains $t^{-1}C_{0,b}$ for $a < b$. Such contours will exist provided $\epsilon \leq t^k$.
\end{proposition}

Proposition \ref{milestone1} is an important milestone in our discussion as it provides an integral representation for a class of observables for $\mathbb{P}_{X,Y}$. In subsequent sections, we will combine the above formulas for different values of $k$, similarly to the moment problem for random variables, in order to better understand the distribution $\mathbb{P}_{X,Y}$.

%
\subsection{An alternative formula for $\mathbb{E}_{X,Y}\left[t^{-k\lambda_1'}\right]$}\hspace{2mm} \\

There are two difficulties in using Proposition \ref{milestone1}. The first is that the contours that we use are all different and depend implicitly on the value $k$. The second issue is that the formula for $\mathbb{E}_{X,Y}\left[t^{-k\lambda_1'}\right]$ that we obtain holds only when $y_i$ are sufficiently small (again depending on $k$). We would like to get rid of this restriction by finding an alternative formula for $\mathbb{E}_{X,Y}\left[t^{-k\lambda_1'}\right]$. This is achieved in Proposition \ref{finmoment}, whose proof relies on the following technical lemma. The following result is very similar to Proposition 7.2 in \cite{BBC}.

\begin{lemma}\label{nestedContours}
Fix $k \geq 1$ and $q \in (1,\infty)$. Assume that we are given a set of positively oriented closed contours $\gamma_1, ...,\gamma_k$, containing $0$, and a function $F(z_1,...,z_k)$, satisfying the following properties:
\begin{enumerate}[label = \arabic{enumi}., leftmargin=1.5cm]
\item $F(z_1,...,z_k) = \prod_{i = 1}^k f(z_i)$;
\item For all $1 \leq A < B \leq k$, the interior of $\gamma_A$ contains the image of $\gamma_B$ multiplied by $q$;
\item For all $1 \leq j \leq k$ there exists a deformation $D_j$ of $\gamma_j$ to $\gamma_k$ so that for all $z_1, ...,z_{j-1},z_j,...,z_k$ with $z_i \in \gamma_i$ for $1 \leq i < j$ and $z_i \in \gamma_k$ for $j < i \leq k$, the function $z_j \rightarrow F(z_1,...,z_j,...,z_k)$ is analytic in a neighborhood of the area swept out by the deformation $D_j$. 
\end{enumerate}
Then we have the following residue expansion identity:
\begin{equation}\label{nestedEq}
\begin{split}
&\int_{\gamma_1} \cdots \int_{\gamma_k} \prod_{1 \leq A < B \leq k} \frac{z_A - z_B}{z_A - qz_B}F(z_1, \cdots, z_k) \prod_{i = 1}^k\frac{dz_i}{ 2\pi \iota z_i} =  \ \sum_{\lambda \vdash k} \frac{(1-q)^k(-1)^kq^{\frac{-k(k-1)}{2}}k_{q }!}{m_1(\lambda)!m_2(\lambda)!\cdots} \\& \int_{\gamma_k}\cdots \int_{\gamma_k} 
\det \left[ \frac{1}{w_iq^{\lambda_i} - w_j}\right]_{i,j = 1}^{\ell(\lambda)}\prod_{j = 1}^{\ell(\lambda)}f(w_j)f(w_jq)\cdots f(w_jq^{ \lambda_j - 1}) \frac{dw_j}{2\pi \iota},\end{split}
\end{equation}
where $k_t! = \frac{(1 - t)(1-t^2) \cdots (1-t^k)}{(1-t)^k}$.
\end{lemma}
\begin{proof}
The proof of the lemma closely follows the proof of Proposition 7.2 in \cite{BBC}, and we will thus only sketch the main idea. We remark that in \cite{BBC} the considered contours do not contain $0$ and $q \in(0,1)$. Nevertheless, all the arguments remain the same and the result of that proposition hold in the setting of the lemma.

The strategy is to sequentially deform each of the contours $\gamma_{k-1}, \gamma_{k-2}, ..., \gamma_1$ to $\gamma_k$ through the deformations $D_i$ afforded from the hypothesis of the lemma. During the deformations one passes through simple poles, coming from $z_A - qz_B$ in the denominator of (\ref{nestedEq}), which by the Residue Theorem produce additional integrals of possibly fewer variables. Once all the contours are expanded to $\gamma_k$ one obtains a big sum of multivariate integrals over various residue subspaces, which can be recombined into the following form (see equation (38) in \cite{BBC}): 
$$ \sum_{\lambda \vdash k} \frac{(1-q)^k(-1)^kq^{\frac{-k(k-1)}{2}}}{m_1(\lambda)!m_2(\lambda)!\cdots} \int_{\gamma_k}\cdots \int_{\gamma_k} 
 \det \left[ \frac{1}{w_iq^{\lambda_i} - w_j}\right]_{i,j = 1}^{\ell(\lambda)} \times$$
$$E^q\left(w_1,qw_1,...,q^{\lambda_1 - 1}w_1,...,w_{\ell(\lambda)},qw_{\ell(\lambda)},...,q^{\lambda_{\ell(\lambda)} - 1}w_{\ell(\lambda)}\right)\prod_{j = 1}^{\ell(\lambda)} w_j^{\lambda_j}q^{\frac{\lambda_j(\lambda_j - 1)}{2}}\frac{d w_j}{2\pi \iota},$$
where
$$E^q(z_1,...,z_k) = \sum_{\sigma \in S_k} \prod_{1 \leq B < A \leq k} \frac{z_{\sigma(A)} - qz_{\sigma(B)}}{z_{\sigma(A)} - z_{\sigma(B)}}\frac{F(z_{\sigma(1)},...,z_{\sigma(n)})}{ \prod_{i = 1}^k z_{\sigma(i)}}.$$
By assumption $\frac{F(z_{1},...,z_{n})}{ \prod_{i = 1}^k z_{i}}$ is a symmetric function of $z_1,...,z_k$ and thus can be taken out of the sum, while the remaining expression evaluates to $k_q!$ as is shown in equation (1.4) in Chapter III of \cite{Mac}. Substituting this back and performing some cancellation we arrive at (\ref{nestedEq}).

\end{proof}

\begin{proposition}\label{finmoment}
Fix positive integers $k$ and $N$ and a parameter $t\in(0,1)$. Let $\rho_X$ and $\rho_Y$ be the nonnegative finite length specializations in $N$ variables $X = (x_1,...,x_N)$ and $Y = (y_1,...,y_N)$ respectively, with $x_i, y_i \in (0,1)$ for $i = 1, ... , N$. Let $C_0$ be a simple positively oriented contour, which is contained in the closed disk of radius $t^{-1}$ around the origin, such that $C_0$ encircles $x_1,...,x_N$ and $0$. Then we have
\begin{equation}\label{fm0}
\begin{split}
\mathbb{E}_{X,Y}\left[ t^{-k\lambda_1'}\right] =\sum_{\lambda \vdash k} \frac{ (t^{-1} - 1)^k k_{t}!}{m_1(\lambda)!m_2(\lambda)! \cdots} \int_{C_{0}}\cdots \int_{C_{0}}\det \left[ \frac{1}{w_it^{-\lambda_i} - w_j}\right]_{i,j = 1}^{\ell(\lambda)}\times\\
\prod_{j = 1}^{\ell(\lambda)}\prod_{i = 1}^N\frac{1 - x_i(w_jt)^{-1}}{1  - x_i(w_jt)^{-1}t^{\lambda_j}} \frac{1 - y_i(w_jt)t^{-\lambda_j}}{1 - y_i(w_jt) }\frac{d w_j}{2\pi \iota} , 
\mbox{ where $k_t! = \frac{(1 - t)(1-t^2) \cdots (1-t^k)}{(1-t)^k}$.}
\end{split}
\end{equation}
\end{proposition}
\begin{proof}
Let $C_{0,k} = C_0$ and let $C_{0,a}$ be such that $C_{0,a}$ contains $t^{-1}C_{0,b}$  \hspace{1mm}for all $a < b$, $a,b \in \{1,...,k\}$. Suppose $0 < \epsilon < t^k$ is sufficiently small so that $C_{0,1}$ is contained in the disk of radius $\epsilon^{-1}$ and suppose $y_i < \epsilon$ for $i = 1,...,N$. Then we may apply Proposition \ref{milestone1} to get
$$\mathbb{E}_{X,Y}\left[ t^{-k\lambda_1'}\right]=  \frac{1}{(2\pi\iota )^k} \int_{C_{0,1}} \cdots \int_{C_{0,k}} \prod_{1 \leq a < b \leq k}  \frac{z_a - z_b}{z_a - z_bt^{-1}} \prod_{i =1}^k \left[ \prod_{ j = 1}^N \frac{(z_i - x_jt^{-1})(1 - z_iy_j)}{(z_i - x_j)(1 - tz_iy_j)} \right] \frac{dz_i}{z_i}.$$
We may now apply Lemma \ref{nestedContours} (with $q = t^{-1}$)  to the RHS of the above and get
\begin{equation}\label{fm1}
\begin{split} \mathbb{E}_{X,Y}\left[ t^{-k\lambda_1'}\right] =\sum_{\lambda \vdash k} \frac{ (1-t^{-1})^k(-1)^kt^{\frac{k(k-1)}{2}}k_{t^{-1}}!}{m_1(\lambda)!m_2(\lambda)!\cdots} \int_{C_{0,k}}\cdots \int_{C_{0,k}}\det \left[ \frac{1}{w_it^{-\lambda_i} - w_j}\right]_{i,j = 1}^{\ell(\lambda)}\times\\
 \prod_{j = 1}^{\ell(\lambda)}G(w_j)G(w_jt^{-1})\cdots G(w_jt^{1 - \lambda_j})\frac{d w_j}{2\pi \iota}, \mbox{ where }G(w) = \prod_{j = 1}^N\frac{w - x_jt^{-1}}{w - x_j}\frac{1 - y_jw}{1 - ty_jw}.\end{split}
\end{equation}
Observe that $(-1)^k t^{\frac{k(k-1)}{2}} k_{t^{-1}}! (1 -t^{-1})^k = (t^{-1} - 1)^kk_t!$ and also
$$\prod_{j = 1}^{\ell(\lambda)}G(w_j)G(w_jt^{-1})\cdots G(w_jt^{1 - \lambda_j})  = \prod_{i = 1}^N  \prod_{j = 1}^{\ell(\lambda)}\frac{1 - x_i(tw_j)^{-1}}{1 - x_i(tw_j)^{-1}t^{\lambda_j}}\frac{1 - y_i(w_jt)t^{-\lambda_j}}{1 - y_i(w_jt)}.$$
Substituting these expressions into (\ref{fm1}) and recalling that $C_{0,k} = C_0$ we arrive at (\ref{fm0}). What remains is to extend the result to arbitrary $y_1,...,y_N \in (0,1)$ by analyticity. In particular, if we can show that both sides of (\ref{fm0}) define analytic functions on $\mathbb{D}^N$ ($\mathbb{D}$ is the unit complex disk), then because they are equal on $(0,\epsilon)^N$ it would follow they are equal on $\mathbb{D}^N$. This would imply the full statement of the proposition.\\

We start with the RHS of (\ref{fm0}). Observe that it is a finite sum of integrals over compact contours. Thus it suffices to show analyticity of the integrands in $y_i \in \mathbb{D}$. The integrand's dependence on $y_i$ is through $ \prod_{j = 1}^{\ell(\lambda)}\prod_{i = 1}^N\frac{1 - y_i(w_jt)t^{-\lambda_j}}{1 - y_i(w_jt) },$ which is clearly analytic on $\mathbb{D}^N$ as $|w_j| \leq t^{-1}$.

For the LHS of (\ref{fm0}) we have:
$$\mathbb{E}_{x,y}\left[ t^{-k\lambda_1'}\right] = \Pi(X;Y)^{-1}\sum_{\lambda \in \mathbb{Y}} P_\lambda(X)Q_\lambda(y_1,...,y_N),$$
where $\Pi(X;Y) = \prod_{i,j = 1}^{N} \frac{1 - tx_iy_j}{1 - x_iy_j}$. Clearly $\Pi(X;Y)$ is analytic and non-zero on $\mathbb{D}^N$ (as $x_i \in (0,1)$) and then so is $\Pi(X;Y)^{-1}$. In addition, the sum is absolutely convergent on $\mathbb{D}^N$, since by the Cauchy identity
$$\sum_{\lambda \mathbb{Y}} |P_\lambda(X)Q_\lambda(y_1,...,y_N)| \leq \sum_{\lambda \in \mathbb{Y}} P_\lambda(X)Q_\lambda(|y_1|,...,|y_N|) = \prod_{i,j = 1}^{N} \frac{1 - tx_i|y_j|}{1 - x_i|y_j|} < \infty.$$
As the absolutely converging sum of analytic functions is analytic and the product of two analytic functions is analytic we conclude that the LHS of (\ref{fm0}) is analytic on $\mathbb{D}^N$.
\end{proof}

%
\subsection{Fredhold determinant formula for $\mathbb{E}_{X,Y} \left[ \frac{1}{((1-t)ut^{-\lambda_1'}; t)_\infty}\right]$}\hspace{2mm}\\

In this section we will combine Proposition \ref{finmoment} with different values of $k$ to obtain a formula for the {\em $t$-Laplace transform} of $(1-t)t^{-\lambda_1'}$, which is defined by $\mathbb{E}_{X,Y} \left[ \frac{1}{((1-t)ut^{-\lambda_1'}; t)_\infty}\right]$. We recall that $(a;t)_{\infty} = (1 -a)(1 -at) (1-at^2) \cdots$ is the $t$-Pochhammer symbol.

The arguments we use to prove the following results are very similar to those in Section 3.2 in \cite{BorCor}.
\begin{proposition}\label{explimp}
Fix $N \in \mathbb{N}$ and $t \in (0,1)$. Let $\rho_X$ and $\rho_Y$ be the nonnegative finite length specializations in $N$ variables $X = (x_1,...,x_N)$ and $Y = (y_1,...,y_N)$ respectively, with $x_i, y_i \in (0,1)$ for $i = 1, ... , N$.  Suppose $|u| < t^{N + 1}$ is a complex number. Then we have
\begin{equation}\label{explim}
\lim_{M \rightarrow \infty} \sum_{k = 0}^M \frac{ u^k \mathbb{E}_{X,Y}[t^{-\lambda_1'k}]}{k_t!} = \mathbb{E}_{X,Y} \left[ \frac{1}{((1-t)ut^{-\lambda_1'};t)_\infty}\right].
\end{equation}
\end{proposition}
\begin{proof}
We have that 
$$\sum_{k = 0}^M \frac{ u^k \mathbb{E}_{X,Y}[t^{-\lambda_1'k}]}{k_t!} = \sum_{c = 0}^{N} \mathbb{P}_{X,Y}(\lambda_1' = c)\sum_{k = 0}^M \frac{ u^kt^{-ck}}{k_t!}$$
By our assumption on $u$ and Corollary 10.2.2a in \cite{Andrews} we have that the inner sum over $k$ converges to $\frac{1}{((1-t)ut^{-c};t)_\infty}$, as $M \rightarrow \infty$. Thus
$$\lim_{M \rightarrow \infty}\sum_{c = 0}^{N} \mathbb{P}_{X,Y}(\lambda_1' = c)\sum_{k = 0}^M \frac{ u^kt^{-ck}}{k_t!} = \sum_{c = 0}^{N}  \frac{\mathbb{P}_{X,Y}(\lambda_1' = c)}{((1-t)ut^{-c};t)_\infty} = \mathbb{E}_{X,Y} \left[ \frac{1}{((1-t)ut^{-\lambda_1'};t)_\infty}\right].$$
\end{proof}

\begin{proposition}\label{detbig}
Fix $N \in \mathbb{N}$, $t \in (0,1)$ and $x_i, y_i \in (0,1)$ for $i = 1, ... , N$. Then there exists $\epsilon > 0$ such that for $|u| < \epsilon$ and $u \not \in \mathbb{R}^+$ we have
\begin{equation}\label{detlim}
\begin{split} 1 + \lim_{M \rightarrow \infty} \sum_{k = 1}^M (t^{-1} - 1)^ku^k\sum_{\lambda \vdash k} \frac{1}{m_1(\lambda)!m_2(\lambda)!\cdots} \int_{C_{0}}\cdots \int_{C_{0}} \det \left[ \frac{1}{w_it^{-\lambda_i} - w_j}\right]_{i,j = 1}^{\ell(\lambda)} \times \\  \prod_{j = 1}^{\ell(\lambda)}\prod_{i = 1}^N\frac{1 - x_i(w_jt)^{-1}}{1  - x_i(w_jt)^{-1}t^{\lambda_j}}\frac{1 -y_i(w_jt)t^{-\lambda_j}}{1 - y_i(w_jt) }\frac{d w_j}{2\pi \iota}= \det(I + K^N_u)_{L^2(C_{0})}.
\end{split}
\end{equation}
In the above $C_0$ is the positively oriented circle of radius $t^{-1}$ around $0$. $K^N_u$ is defined in terms of its integral kernel
$$K^N_u(w; w') = \frac{1}{2\pi \iota} \int_{1/2 -\iota\infty}^{1/2 + \iota\infty} ds \Gamma(-s)\Gamma(1 + s)(-u(t^{-1} - 1))^s g^N_{w,w'}(t^s),$$
where 
$$g^N_{w,w'}(t^s) = \frac{1}{wt^{-s} - w'}\prod_{j = 1}^N \frac{(1 - x_j(wt)^{-1})(1 - y_j(wt)t^{-s})}{(1 - x_j(wt)^{-1}t^s)(1 - y_j(wt))}.$$
\end{proposition}
The proof of Proposition \ref{detbig} depends on two lemmas: Lemma \ref{gammal} and Lemma \ref{tracel}, whose proof is postponed to Section \ref{lemmasep}. Our choice for $C_0$ is made in order to simplify the proof.

\begin{proof}
From Lemma \ref{tracel} we know that $K^N_u$ is trace-class for $u \not \in \mathbb{R}^+$. Consequently we have that 
$$\det(I + K^N_u)_{L^2(C_{0})} = 1 + \sum_{n = 1}^\infty \frac{1}{n!}\int_{C_{0}}\cdots \int_{C_{0}} \det \left[ K^N_u(w_i, w_j)\right] _{i,j = 1}^n \prod_{i = 1}^{n}  \frac{dw_i}{2 \pi \iota } = $$
$$1 + \sum_{n = 1}^\infty \frac{1}{n!}\int_{C_{0}}\hspace{-3mm}\cdots \int_{C_{0}}\sum_{\sigma \in S_n} sign(\sigma) \prod_{i = 1}^{n} \left[ \frac{1}{2\pi i} \int_{1/2 - \iota\infty}^{1/2 + \iota\infty}  \hspace{-6mm}\Gamma(-s)\Gamma(1 + s)(-u(t^{-1} - 1))^s g_{w_i, w_{\sigma(i)}}(t^s)ds \right] \prod_{i = 1}^{n} \frac{dw_i}{2 \pi \iota }.$$
Using Lemma \ref{gammal} and the above formula we can find an $\epsilon > 0$ such that for $|u| < \epsilon$ and $u \not \in \mathbb{R}^+$ one has 
\begin{equation}\label{RHSdetbig}
\det(I + K^N_u) = 1 + \sum_{n = 1}^\infty \frac{1}{n!}\int_{C_{0}}\cdots \int_{C_{0}}\sum_{\sigma \in S_n} sign(\sigma) \prod_{i = 1}^{n} \left[ \sum_{j = 1}^{\infty} u^j(t^{-1} - 1)^{j} g_{w_i, w_{\sigma(i)}}(t^j) \right] \prod_{i = 1}^{n} \frac{dw_i}{2 \pi \iota }.
\end{equation}

Let us introduce the following short-hand notation
$$B(c_1,...,c_n) :=  \int_{C_{0}}\cdots \int_{C_{0}} \det \left[ \frac{1}{w_it^{-c_i} - w_j}\right]_{i,j = 1}^{n} \prod_{j = 1}^{n}\prod_{i = 1}^N\frac{1 - x_i(w_jt)^{-1}}{1  - x_i(w_jt)^{-1}t^{c_j}} \frac{1 - y_i(w_jt)t^{-c_j}}{1 - y_i(w_jt) }\frac{dw_j}{2\pi \iota}.$$
Notice that $B(c_1,...,c_n)$ is invariant under permutation of its arguments and that $\frac{(m_1(\lambda) + m_2(\lambda) + \cdots )!}{m_1(\lambda)!m_2(\lambda)!\cdots}$ is the number of distinct permutations of the parts of $\lambda$. The latter suggests that
$$ \sum_{\lambda \vdash k} \frac{(t^{-1} - 1)^ku^k}{m_1(\lambda)!m_2(\lambda)!\cdots}B(\lambda_1,...,\lambda_{\ell(\lambda)})=  \sum_{n \geq 1}  \sum_{\substack{c_1,c_2,...,c_n \geq 1 \\ \sum c_i = k }} \frac{(t^{-1} - 1)^ku^k}{n!} B(c_1,...,c_n).$$

Observe that for some positive constant $C$ we have
$$\left| \prod_{j = 1}^{n}\prod_{i = 1}^N\frac{1 - x_i(w_jt)^{-1}}{1  - x_i(w_jt)^{-1}t^{c_j}} \frac{1 -y_i(w_jt)t^{-c_j}}{1 - y_i(w_jt) }\right| \leq C^{Nn}t^{-Nk}\prod_{i = 1}^N\frac{1}{(1-x_i)^n(1-y_i)^n}.$$
The above together with Hadamard's inequality and the compactness of $C_0$ implies that for some positive constants $P,Q$ (independent of $k$ and $n$) we have $|B(c_1,...,c_n)| \leq n^{n/2}P^nQ^k$. The latter implies that for $|u| < \epsilon$ and $\epsilon$ sufficiently small the sum
$$\sum_{k = 1}^\infty\sum_{n \geq 1} \hspace{2mm} \sum_{\substack{c_1,c_2,...,c_n \geq 1 \\ \sum c_i = k }} \frac{(t^{-1} - 1)^ku^k}{n!} B(c_1,...,c_n)$$
is absolutely convergent. In particular, the limit on the LHS of equation (\ref{detlim}) exists and equals 
\begin{equation*}
1 + \sum_{n = 1}^{\infty}\frac{1}{n!} \hspace{2mm} \sum_{\substack{c_1,c_2,...,c_n \geq 1}}  [(t^{-1} - 1)u]^{c_1+\cdots+c_n} B(c_1,...,c_n).
\end{equation*}
Expanding the determinant inside the integral in the definition of $B(c_1,...,c_n)$ we see that the integrand equals $\sum_{\sigma \in S_n} sign(\sigma) \prod_{i = 1}^{n}g_{w_i,w_{\sigma(i)}}(t^{c_i})$. Consequently the LHS of equation (\ref{detlim}) equals
\begin{equation}\label{LHSdetbig}
1 + \sum_{n = 1}^{\infty}\frac{1}{n!} \hspace{2mm} \sum_{\substack{c_1,c_2,...,c_n \geq 1}}  [(t^{-1} - 1)u]^{c_1+\cdots+c_n}\int_{C_{0}}\cdots \int_{C_{0}} \sum_{\sigma \in S_n} sign(\sigma) \prod_{i = 1}^{n}g_{w_i,w_{\sigma(i)}}(t^{c_i}) \frac{dw_i}{2 \pi \iota}.
\end{equation}

What remains is to check that the two expressions in (\ref{LHSdetbig}) and (\ref{RHSdetbig}) agree. Since both are absolutely converging sums over $n$, it suffices to show equality of the corresponding summands. I.e. we wish to show that
\begin{equation}\label{laststretch}
\sum_{\substack{c_1,c_2,...,c_n \geq 1}} \hspace{-5mm} [(t^{-1} - 1)u]^{c_1+\cdots+c_n}\int_{C_{0}}\cdots \int_{C_{0}} \sum_{\sigma \in S_n} sign(\sigma) \prod_{i = 1}^{n}g_{w_i,w_{\sigma(i)}}(t^{c_i}) \frac{dw_i}{2 \pi \iota}= $$
$$= \int_{C_{0}}\cdots \int_{C_{0}}\sum_{\sigma \in S_n} sign(\sigma) \prod_{i = 1}^{n} \left[ \sum_{j = 1}^{\infty} u^j(t^{-1} - 1)^{j} g_{w_i, w_{\sigma(i)}}(t^j) \right]  \frac{dw_i}{2 \pi \iota }.
\end{equation}
By Fubini's Theorem (provided $|u|$ is sufficiently small) we may interchange the order of the sum and the integrals and the LHS of equation (\ref{laststretch}) becomes
$$\int_{C_{0}}\cdots \int_{C_{0}} \sum_{\substack{c_1,c_2,...,c_n \geq 1}}  [(t^{-1} - 1)u]^{c_1+...+c_n}\sum_{\sigma \in S_n} sign(\sigma) \prod_{i = 1}^{n}g_{w_i,w_{\sigma(i)}}(t^{c_i}) \frac{dw_i}{2 \pi \iota } = $$
$$ \int_{C_{0}}\cdots \int_{C_{0}}\sum_{\sigma \in S_n}sign(\sigma)\prod_{i = 1}^{n}\left[\sum_{c_i \geq 1}[(t^{-1} - 1)u]^{c_i} g_{w_i,w_{\sigma(i)}}(t^{c_i}) \right]  \frac{dw_i}{2 \pi \iota }.$$
From the above equation (\ref{laststretch}) is obvious. This concludes the proof.
\end{proof}

\begin{proposition}\label{finlength}
Fix $N \in \mathbb{N}$ and a parameter $t\in(0,1)$. Let $\rho_X$ and $\rho_Y$ be the nonnegative finite length specializations in $N$ variables $X = (x_1,...,x_N)$ and $Y = (y_1,...,y_N)$ respectively, with $x_i, y_i \in (0,1)$ for $i = 1, ... , N$. Then for $u \not \in \mathbb{R}^+$ one has that
\begin{equation}\label{finlengthform}
\mathbb{E}_{X,Y} \left[ \frac{1}{((1-t)ut^{-\lambda_1'}; t)_\infty} \right] = \det (I + K^N_u)_{L^2(C_0)}.
\end{equation}
The contour $C_0$ is the positively oriented circle of radius $t^{-1}$, centered at $0$, and the operator $K^N_u$ is defined in terms of its integral kernel
$$K^N_u(w,w') = \frac{1}{2\pi \iota} \int_{1/2 - \iota\infty}^{1/2 + \iota\infty} ds \Gamma(-s)\Gamma(1 + s)(-u(t^{-1}-1))^s g^N_{w,w'}(t^s),$$
where 
$$g^N_{w,w'}(t^s) = \frac{1}{wt^{-s} - w'}\prod_{j = 1}^N \frac{(1 - x_j(wt)^{-1})(1 - y_j(wt)t^{-s})}{(1 - x_j(wt)^{-1}t^s)(1 - y_j(wt))}.$$
\end{proposition}
\begin{proof}
Using Propositions \ref{finmoment}, \ref{explimp} and \ref{detbig} we have the statement of the proposition for $|u| < \epsilon$ and $u \not \in \mathbb{R}^{+}$ for some sufficiently small $\epsilon > 0$. To conclude the proof it suffices to show that both sides of (\ref{finlengthform}) are analytic functions of $u$ in $\mathbb{C} \backslash \mathbb{R}^{+}$. 

The RHS is analytic by Lemma \ref{tracel}, while the LHS of (\ref{finlengthform}) equals $\sum_{n = 0}^{N} \mathbb{P}_{x,y}(\lambda_1' = n) \frac{1}{(ut^{-n}; t)_\infty}$, and is thus a finite sum of analytic functions and so also analytic on $\mathbb{C} \backslash \mathbb{R}^{+}$.
\end{proof}

%
\subsection{Proof of  Lemmas \ref{gammal} and \ref{tracel}}\hspace{2mm}\label{lemmasep}\\

Versions of the following two lemmas appear in Section 3.2 of \cite{BorCor}.
\begin{lemma}\label{gammal}
Fix $N \in \mathbb{N}$, $t \in (0,1)$ and $x_i, y_i \in (0,1)$ for $i = 1, ... , N$. Let $w,w' \in \mathbb{C}$ be such that $|w| = |w'| = t^{-1}$ and let
$$g^N_{w,w'}(t^s) = \frac{1}{wt^{-s} - w'}\prod_{j = 1}^N \frac{(1 - x_j(wt)^{-1})(1 - y_j(wt)t^{-s})}{(1 - x_j(wt)^{-1}t^s)(1 - y_j(wt))}.$$
Then there exists $\epsilon > 0$ such that if $\zeta \in \{\zeta: |\zeta| < \epsilon, \zeta \not \in \mathbb{R}^+\}$, we have
\begin{equation}\label{gResidues}
\sum_{n = 1}^{\infty} g^N_{w,w'}(t^n) \zeta^n = \frac{1}{2\pi \iota} \int_{1/2 -\iota\infty}^{1/2 + \iota\infty}  \Gamma(-s)\Gamma(1 + s)(-\zeta)^s g^N_{w,w'}(t^s)ds.
\end{equation}
\end{lemma}
\begin{proof}
For simplicity we suppress $N$ from our notation. Let $R_M = M + 1/2$ ($M \in \mathbb{N})$ and set $A^1_M = 1/2 - \iota R_M$, $A^2_M =  1/2 + \iota R_M$, $A^3_M = R_M + \iota R_M$ and $A^4_M = R_M - \iota R_M$. Denote by $\gamma^1_M$ the contour, which goes from $A_M^1$ vertically up to $A_M^2$, by $\gamma^2_M$ the contour, which goes from $A_M^2$ horizontally to $A_M^3$, by $\gamma^3_M$ the contour, which goes from $A_M^3$ vertically down to $A_M^4$, and by $\gamma^4_M$ the contour, which goes from $A_M^4$ horizontally to $A_M^1$. Also let $\gamma_M = \cup_i \gamma_M^i$ traversed in order (see Figure \ref{S3_1}).

\begin{figure}[h]
\centering
\scalebox{0.6}{\includegraphics{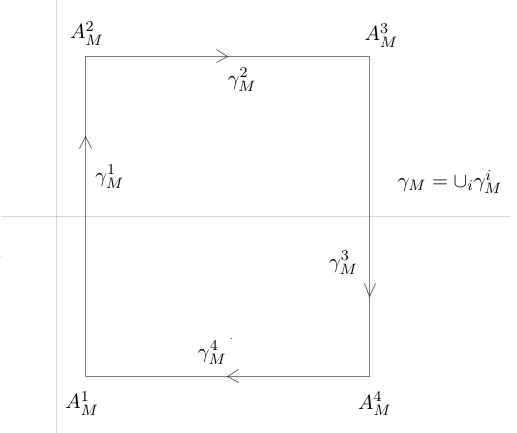}}
\caption{The contours $\gamma_M^i$ for $i = 1,...,4$.}
\label{S3_1}
\end{figure}

We make the following observations:
\begin{enumerate}[label = \arabic{enumi}., leftmargin=1.5cm]
\item $\gamma_M$ is negatively oriented.
\item The function $g_{w,w'}(t^s)$ is well-defined and analytic in a neighborhood of the closure of the region enclosed by $\gamma_M$. This follows from $|t^s| < 1$ for $Re(s) > 0$, which prevents any of the poles of $g_{w,w'}(t^s)$ from entering the region $Re(s) > 0$.
\item  If dist$(s, \mathbb{Z}) > c$ for some fixed constant $c > 0$, then $\left| \frac{\pi}{\sin(\pi s)} \right| \leq c'e^{-\pi|Im(s)|}$ for some fixed constant $c'$, depending on $c$. In particular, this estimate holds for all $s \in \gamma_M$ since dist$(\gamma_M, \mathbb{Z}) = 1/2$ for all $M$ by construction.\item If $-\zeta = re^{\iota \theta}$ with $|\theta| < \pi$ and $s = x+ \iota y$ then 
$$(-\zeta)^s = \exp\left( (\log(r) + \iota\theta)(x + \iota y)\right) = \exp\left( \log(r)x - y\theta + \iota(\log(r)y + x\theta)\right),$$
 since we took the principal branch. In particular, $|(-\zeta)^s| = r^xe^{-y\theta}$.
\end{enumerate}
We also recall Euler's Gamma reflection formula
\begin{equation}\label{Euler}
\Gamma(-s)\Gamma(1 + s) = \frac{\pi}{\sin(-\pi s)}.
\end{equation}

We observe for $s = x+ \iota y,$ with $x \geq 1/2$ that
$$|g_{w,w'}(t^s)| = \left| \frac{1}{wt^{-s} - w'}\prod_{j = 1}^N \frac{(1 - x_j(wt)^{-1})(1 - y_j(wt)t^{-s})}{(1 - x_j(wt)^{-1}t^s)(1 - y_j(wt))} \right| \leq \frac{\prod_{j = 1}^N \left| 1 - y_j(wt)t^{-s} \right|}{t^{-3/2} - t^{-1}}\prod_{i = 1}^N\frac{2}{(1-y_i)(1-x_i)}.$$
In addition, we have$\prod_{j = 1}^N \left| 1 - y_j(wt)t^{-s} \right| \leq Ce^{cx}$ for some positive constants $C,c > 0$, depending on $N$, $t$ and $y_i$.
Consequently, we see that if $\epsilon$ is chosen sufficiently small and $\zeta =re^{\iota\theta}$ with $r < \epsilon$ then
$$|g_{w,w'}(t^s)(- \zeta)^s| \leq Ce^{cx} \epsilon^x e^{|y\theta|} \leq Ce^{-cx} e^{|y\theta|},$$
with some new constant $C > 0$. In particular, the LHS in (\ref{gResidues}) is absolutely convergent, and we have
$$\sum_{n = 1}^{\infty} g_{w,w'}(t^n) \zeta^n = \lim_{M \rightarrow \infty}\sum_{n = 1}^{M} g_{w,w'}(t^n) \zeta^n.$$

From the Residue Theorem we have 
$$\sum_{n = 1}^{M} g_{w,w'}(t^n) \zeta^n = \frac{1}{2\pi\iota} \int_{\gamma_M}  \Gamma(-s)\Gamma(1 + s)(-\zeta)^s g_{w,w'}(t^s)ds.$$ 
The last formula used $Res_{s = k}\Gamma(-s)\Gamma(1 + s) = (-1)^{k+1}$ and observations $1.$ and $2.$ above. What remains to be shown is that
\begin{equation}\label{gammaConv}
\lim_{M \rightarrow \infty} \frac{1}{2\pi \iota} \int_{\gamma_M}  \Gamma(-s)\Gamma(1 + s)(-\zeta)^s g_{w,w'}(t^s)ds = \frac{1}{2\pi \iota} \int_{1/2 - \iota\infty}^{1/2 + \iota\infty}  \Gamma(-s)\Gamma(1 + s)(-\zeta)^s g_{w,w'}(t^s)ds.
\end{equation}

Observe that on $Re(s) = 1/2$ we have that $|g_{w,w'}(t^s)|$ is bounded, while from (\ref{Euler}) and observations 3. and 4. we have
\begin{equation}\label{gzest}
\left|  \Gamma(-s)\Gamma(1 + s)(-\zeta)^s \right| = \left| \frac{\pi}{\sin(-\pi s)}(-\zeta)^s \right| \leq c'\exp((|\theta|-\pi)|Im(s)|) r^{1/2} ,
\end{equation}
which decays exponentially in $|Im(s)|$ since $|\theta| < \pi$. Thus the integrand on the RHS of (\ref{gammaConv}) is exponentially decaying near $\pm \iota \infty$ and so the integral is well-defined. Moreover, from the Dominated Convergence Theorem we have that
$$\lim_{M \rightarrow \infty} \frac{1}{2\pi \iota} \int_{\gamma^1_M}  \Gamma(-s)\Gamma(1 + s)(-\zeta)^s g_{w,w'}(t^s)ds = \frac{1}{2\pi \iota} \int_{1/2 - \iota\infty}^{1/2 + \iota\infty}  \Gamma(-s)\Gamma(1 + s)(-\zeta)^s g_{w,w'}(t^s)ds.$$
We now consider the integrals 
$$\frac{1}{2\pi \iota} \int_{\gamma^i_M}  \Gamma(-s)\Gamma(1 + s)(-\zeta)^s g_{w,w'}(t^s),$$
when $i \neq 1$ and show they go to $0$ in the limit. If true, (\ref{gammaConv}) will follow. \\

Suppose that $i = 2$ or $i = 4$. Let $s = x+ \iota y \in \gamma_M^i$, so $|y| = R_M$ and we get
$$\left|  \Gamma(-s)\Gamma(1 + s)(-\zeta)^s g_{w,w'}(t^s) \right| \leq Ce^{-cx}e^{|\theta y|} c' e^{-\pi |y|} \leq Ce^{(|\theta| - \pi )R_M},$$
for some new constant $C > 0$. Since $|\theta| - \pi < 0$ we see that
$$\left| \frac{1}{2\pi \iota} \int_{\gamma^i_M}  \Gamma(-s)\Gamma(1 + s)(-\zeta)^s g_{w,w'}(t^s) \right| \leq C R_M e^{(|\theta| - \pi)R_M} \rightarrow 0 \mbox{ as } M \rightarrow \infty.$$

Finally, let $i = 3$. Let  $s = x+ \iota y \in \gamma_M^3$, so $x= R_M$ and we get
$$\left|  \Gamma(-s)\Gamma(1 + s)(-\zeta)^s g_{w,w'}(t^s) \right| \leq Ce^{-cx}e^{|\theta y|} c' e^{-\pi |y|} \leq Cc'e^{-cR_M}.$$
Consequently, we obtain
$$\left| \frac{1}{2\pi \iota} \int_{\gamma^3_M}  \Gamma(-s)\Gamma(1 + s)(-\zeta)^s g_{w,w'}(t^s) \right| \leq 2 R_M Cc'e^{-cR_M} \rightarrow 0  \mbox{ as } M \rightarrow \infty.$$
This concludes the proof of (\ref{gammaConv}) and hence the lemma.
\end{proof}

\begin{lemma}\label{tracel}
Fix $N \in \mathbb{N}$ or $N = \infty$, $t \in (0,1)$ and $x_i, y_i \in (0,1)$ for $i = 1, ... , N$ such that $\sum_i x_i < \infty$, $\sum_i y_i < \infty$. Suppose $u \in \mathbb{C}\backslash\mathbb{R}^+$. Consider the operator $K^N_u$ on $L^2(C_0)$ (here $C_0$ is the positive circle of radius $t^{-1}$), which is defined in terms of its integral kernel
$$K^N_u(w,w') = \frac{1}{2\pi \iota} \int_{1/2 - \iota\infty}^{1/2 + \iota\infty} ds \Gamma(-s)\Gamma(1 + s)(-u(t^{-1}-1))^s g^N_{w,w'}(t^s),$$
where 
$$g^N_{w,w'}(t^s) = \frac{1}{wt^{-s} - w'}\prod_{j = 1}^N \frac{(1 - x_j(wt)^{-1})(1 - y_j(wt)t^{-s})}{(1 - x_j(wt)^{-1}t^s)(1 - y_j(wt))}.$$
Then $K^N_u$ is trace-class. Moreover, as a function of $u$ we have that $\det(I + K_u^N)$ is an analytic function on $\mathbb{C}\backslash\mathbb{R}^+$.
\end{lemma}
\begin{proof}
We begin with the first statement of the lemma and suppress the dependence on $N$ and $u$ from the notation. From Lemma \ref{traceclass} it suffices to show that $K(w,w')$ is continuous on $C_0\times C_0$ and that $K_2(w,w')$ is continuous as well, where we recall that $K_2(w,w')$ is the derivative of $K(x,y)$ along the contour $C_0$ in the second entry.\\

 In equation (\ref{gzest}) we showed that if $-u(t^{-1} - 1) = re^{\iota\theta}$ with $|\theta| < \pi$ and $s = 1/2 + \iota y$, then
$$\left|  \Gamma(-s)\Gamma(1 + s)(-\zeta)^s \right| \leq C\exp((|\theta|-\pi)|y|) r^{1/2}$$
We observe that $g_{w,w'}(t^s)$ is continuous in $w,w'$ and moreover on $Re(s) = 1/2$ we have 
$$|g_{w,w'}(t^s)| \leq M = \frac{1}{t^{-3/2} - t^{-1}} \prod_{j = 1}^N \frac{(1 + x_j)(1 + y_j t^{-1/2})}{(1 -x_jt^{1/2})(1-y_j)} < \infty$$
 independently of $w,w'$.
So if $(w_n, w_n') \rightarrow (w,w')$ we have that $g_{w_n,w'_n}(t^s) \rightarrow g_{w.w'}(t^s)$ and by the Dominated Convergence Theorem, we conclude that $K(w_n,w_n') \rightarrow K(w,w')$ so that $K(w,w')$ is continuous on $C_0 \times C_0$. 

We next observe that 
$$K_2(w,w') = \iota w' \frac{d}{dw'}K(w,w') = \iota w'\frac{1}{2\pi \iota} \int_{1/2 - \iota\infty}^{1/2 + \iota\infty} ds \Gamma(-s)\Gamma(1 + s)(-u(t^{-1}-1))^s \frac{d}{dw'}g_{w,w'}(t^s),$$
where the change of the order of integration and differentiation is allowed by the exponential decay of the integrand. We have that $\frac{d}{dw'}g_{w,w'}(t^s) = -\frac{1}{wt^{-s} -w'} g_{w,w'}(t^s)$ so a similar argument as above now shows that $K_2(w,w')$ is continuous on $C_0 \times C_0$. We conclude that $K^N_u$ is indeed trace-class.\\

Since $K^N_u$ is trace-class we know that
$$\det(I + K_u^N) = 1 + \sum_{n \geq 1} \frac{1}{n!} \int_{C_0}\cdots \int_{C_0} \det\left[ K_u^N(w_i, w_j)\right]_{i,j = 1}^n \prod_{i = 1}^n\frac{dw_i}{2\pi \iota}.$$
We wish to show that the above sum is analytic in $u \in \mathbb{C}\backslash\mathbb{R}^+$. 

We begin by showing that $K_u^N(w,w')$ is analytic in $u$ for each $(w,w') \in C_0\times C_0$. Observe that on $(\mathbb{C}\backslash\mathbb{R}^+) \times (1/2 + \iota \mathbb{R})$,  $\Gamma(-s)\Gamma(1 + s)(-u(t^{-1}-1))^s g^N_{w,w'}(t^s)$ is jointly continuous in $(u,s)$ and analytic in $u$ for each $s$. From Theorem 5.4 in Chapter 2 of \cite{Stein} we know that for any $A\geq 0$ 
 $$h_A(u) := \int_{1/2 - \iota A}^{1/2 + \iota A} \Gamma(-s)\Gamma(1 + s)(-u(t^{-1}-1))^s g^N_{w,w'}(t^s) ds$$
is an analytic function of $u \in \mathbb{C}\backslash\mathbb{R}^+$. In addition, using our earlier estimates we see that 
$$|h_A(u) -K_u^N(w,w')| \leq 2|u|^{1/2}MC\int_A^\infty \exp((|\theta|-\pi)y) dy =\frac{2|u|^{1/2}MC}{\pi - |\theta|} \exp((|\theta|-\pi)A).$$
The latter shows that $h_A(u)$ converges uniformly on compact subsets of $\mathbb{C}\backslash\mathbb{R}^+$ to $K_u^N(w,w')$ as $A \rightarrow \infty$, which implies that $K_u^N(w,w')$ is analytic in $u$. Notice that when $A = 0$ the above shows that if $K'$ is a compact subset of $\mathbb{C}\backslash\mathbb{R}^+$ and $u \in K'$, we have $|K_u^N(w,w')| \leq C(K')$ for some contant $C > 0$ independent of $w,w'$. 

We next observe that $K^N_u(w,w')$ is jointly continuous in $u$ and $(w,w')$ and analytic in $u$ for each $w,w'$ from our proof above. The latter implies that $\det\left[K_u^N(w_i,w_j)\right]_{i,j = 1}^n$ is continuous on $C_0^n \times \mathbb{C}\backslash\mathbb{R}^+$ and analytic in $u$ for each $(w_1,...,w_n) \in C_0^n$. It follows from Theorem 5.4 in Chapter 2 of \cite{Stein} that 
$$H_n(u) = \frac{1}{n!} \int_{C_0}\cdots \int_{C_0} \det\left[ K_u^N(w_i, w_j)\right]_{i,j = 1}^n \prod_{i = 1}^n\frac{dw_i}{2\pi \iota},$$
is analytic in $u$. 

Finally, suppose $K' \subset \mathbb{C}\backslash\mathbb{R}^+$ is compact and $u \in K'$. Then from Hadamard's inequality and our earlier estimate on $|K_u^N(w,w')|$ we know that 
$$|H_n(u)| = \frac{1}{n!} \left|  \int_{C_0}\cdots \int_{C_0} \det\left[ K_u^N(w_i, w_j)\right]_{i,j = 1}^n \frac{dw_i}{2\pi \iota}\right| \leq  \frac{1}{n!} ( t^{-1})^n n^{n/2}C(K')^n = B^n \frac{n^{n/2}}{n!}.$$
The latter is absolutely summable, and since the absolutely convergent sum of analytic functions is analytic and $K'$ was arbitrary,  we conclude that $1 + \sum_{n = 1}^\infty H_n(u) = \det(I + K_u^N)_{L^2(C^0)}$ is analytic in $u$ on $\mathbb{C}\backslash\mathbb{R}^+$. This suffices for the proof.
\end{proof}

\section{GUE asymptotics}\label{Section4}
In this section, we use the results from Section \ref{finiteS} to get formulas for the $t$-Laplace transform of $t^{1-\lambda_1'}$, with $\lambda$ distributed according to the Hall-Littlewood measure with parameters $a,r,t \in (0,1)$ (see Section \ref{HL}). Subsequently, we analyze the formulas that we get in the limiting regime $r \rightarrow 1^-$, $t \in(0,1)$ - fixed and obtain convergence to the Tracy-Widom GUE distribution. In what follows, we will denote by $\mathbb{P}_{a,r,t}$ and $\mathbb{E}_{a,r,t}$ the probability distribution and expectation with respect to the Hall-Littlewood measure with parameters $a,r,t \in (0,1)$.

%
\subsection{Fredholm determinant formula for $\mathbb{E}_{a,r,t}\left[  \frac{1}{((1-t)ut^{-\lambda_1'}; t)_\infty}\right]$}\hspace{2mm} \\

In the following results, unless otherwise specified, $\det (I + K)_{L^2(C)}$ dentotes the absolutely convergent sum on the RHS of (\ref{fredholmDef}) - see the discussion in Section \ref{FredholmS}.

\begin{proposition}\label{prelimit}
Suppose $a,r,t\in(0,1)$ and let $\delta > 0$ be such that $a < ( 1 - \delta)$. Then for $u \in \mathbb{C}\backslash  \mathbb{R}^+$ one has that
\begin{equation}\label{prelimform}
\mathbb{E}_{a,r,t} \left[ \frac{1}{((1-t)ut^{-\lambda_1'}; t)_\infty} \right] = \det (I + K_u)_{L^2(C_0)}.
\end{equation}
The contour $C_0$ is a positively oriented piecewise smooth simple curve, contained in the closed annulus $A_{\delta, t}$ between the $0$-centered circles of radius $t^{-1}$ and $ \max \left( t^{-1}(1 - \delta/2), t^{-3/4}\right)$. The kernel $K_u(w, w')$ is defined as
\begin{equation}\label{Ku}
K_u(w,w') = \frac{1}{2\pi \iota} \int_{1/2 - \iota\infty}^{1/2 + \iota\infty} ds \Gamma(-s)\Gamma(1 + s)(-u(t^{-1}-1))^s g_{w,w'}(t^s),
\end{equation}
where 
$$g_{w,w'}(t^s) = \frac{1}{wt^{-s} - w'}\prod_{j = 0}^\infty \frac{(1 - ar^{j}(wt)^{-1})(1 - ar^{j}(wt)t^{-s})}{(1 - ar^{j}(wt)^{-1}t^s)(1 - ar^{j}(wt))}.$$
\end{proposition}
\begin{remark}
Proposition \ref{prelimit} will be the starting point for our asymptotic analysis in both the GUE and CDRP cases. In the different limiting regimes, we will encounter different contours, which will be suitably picked contours contained in $A_{\delta, t}$. 
\end{remark}

\begin{proof}
We first prove the proposition when $C_0$ is the positively oriented circle of radius $t^{-1}$. The starting point is Proposition \ref{finlength}, from which we see that whenever $u \not \in \mathbb{R}^{+}$ one has for every $N \in \mathbb{N}$
$$\mathbb{E}^N_{a,r,t} \left[ \frac{1}{((1-t)ut^{-\lambda'_1}; t)_\infty} \right] = \det (I + K^N_u)_{L^2(C_0)}.$$
Here $\mathbb{E}^N_{a,r,t}$ stands for the expectation with respect to the Macdonald measure on partitions, corresponding to $q = 0$ and $x_i = y_i = ar^{i-1}$ for $i = 1,...,N$ and $x_i = y_i = 0$ for $i > N$. The result would thus follow once we show that 
\begin{enumerate}[label = \arabic{enumi}., leftmargin=1.0cm]
\item $\lim_{N \rightarrow \infty} \mathbb{E}^N_{a,r,t} \left[ \frac{1}{((1-t)ut^{-\lambda'_1}; t)} \right] = \mathbb{E}_{a,r,t} \left[ \frac{1}{((1-t)ut^{-\lambda'_1}; t)} \right]$
\item $\lim_{N \rightarrow \infty}\det (I + K^N_u)_{L^2(C_0)} = \det (I + K_u)_{L^2(C_0)}$.
\end{enumerate}
Before we prove the above two statements let us remark that the two limiting quantities are indeed well-defined. The fact that $K_u$ is a trace-class operator on $L^2(C_0)$ follows from Lemma \ref{tracel}. Next, we observe that if $u\not \in \mathbb{R}^{+}$ then for any $n$ we have that $\frac{1}{(ut^{-n};t)_\infty}$ is well defined and moreover there exists a constant $M(u)$ such that $\left|\frac{1}{(ut^{-n};t)_\infty}\right| \leq M, $ for all $n$. Consequently, we can define unambiguously the expectation $ \mathbb{E}_{a,r,t} \left[ \frac{1}{((1-t)ut^{-\lambda'_1}; t)} \right]$ and it is a finite quantity.\\

We start with $1.$ Denote by $P_\lambda^N$ and $Q_\lambda^N$ the $N$-length specialization of the the Hall-Littlewood symmetric functions with $x_i = y_i = ar^{i-1}$ for $i = 1,...,N$ and $x_i = y_i = 0$ for $i > N$ (here $N$ is a positive integer or $\infty$). Also let $Z^N$ be the normalization constant, which in the above case equals
$$Z^N = \prod_{i,j = 1}^N \frac{1 - tar^{i-1}ar^{j-1}}{1 - ar^{i-1}ar^{j-1}} \mbox{ - this is the Cauchy identity in (\ref{CauchyI})}.$$
We obtain 
$$\mathbb{E}^N_{a,r,t} \left[ \frac{1}{((1-t)ut^{-\lambda_1'}; t)_\infty} \right]  = \frac{1}{Z^N}\sum_{\lambda \in \mathbb{Y}} P^N_\lambda Q^N_\lambda\frac{1}{((1-t)ut^{-\lambda'_1}; t)}_\infty.$$
One readily verifies that $ Z^N \nearrow Z^\infty$, $ P_\lambda^N \nearrow P_\lambda^\infty$ and $Q_\lambda^N \nearrow Q_\lambda^\infty$ as $N \rightarrow \infty$. Thus from the Dominated Convergence Theorem (with dominating function $M P^\infty_\lambda Q^\infty_\lambda$) we get
$$\lim _{N\rightarrow \infty} \sum_{\lambda \in \mathbb{Y}} P^N_\lambda Q^N_\lambda\frac{1}{((1-t)ut^{-\lambda'_1}; t)}_\infty =  \sum_{\lambda \in \mathbb{Y}} P^\infty_\lambda Q^\infty_\lambda\frac{1}{((1-t)ut^{-\lambda'_1}; t)}_\infty.$$
The latter implies that
$$\lim _{N\rightarrow \infty} \frac{1}{Z^N}\sum_{\lambda \in \mathbb{Y}} P^N_\lambda Q^N_\lambda\frac{1}{((1-t)ut^{-\lambda'_1}; t)}_\infty =\frac{1}{Z^\infty} \sum_{\lambda \in \mathbb{Y}} P^\infty_\lambda Q^\infty_\lambda\frac{1}{((1-t)ut^{-\lambda'_1}; t)}_\infty,$$
which concludes the proof of $1.$\\

Next we turn to $2.$ Firstly, we one readily observes that 
$$g^N_{w,w'} (t^s) \rightarrow g_{w,w'}(t^s), \mbox{ as } N \rightarrow \infty$$
and moreover we have
$$|g^N_{w,w'}(t^s)| \leq \frac{1}{t^{-3/2} - t^{-1}} \prod_{j = 0}^\infty \frac{(1 + ar^{j})(1 + ar^{j} t^{-1/2})}{(1 -ar^{j}t^{1/2})(1-ar^{j})} = M < \infty,$$
independently of $N, w, w'$.  Recall from (\ref{gzest}) that
$$\left|  \Gamma(-s)\Gamma(1 + s)(-(t^{-1} - 1)u)^s \right| \leq C\exp((|\theta|-\pi)|y|) r^{1/2},$$
where $-(t^{-1} - 1)u = re^{\iota\theta}$ and $s = 1/2 + \iota y$. It follows by the Dominated Convergence Theorem (with dominating function $MC\exp((|\theta|-\pi)|y|) r^{1/2}$)  that 
$$\lim_{N \rightarrow \infty} K^N_u(w,w') = K_u(w,w'),$$
and moreover there exists a finite constant $M_2$ (depending on $u$) such that $|K^N_u(w,w')| \leq M_2$ for all $N, w, w'$. Next we have from the Bounded Convergence Theorem that for every $n$ 
$$\lim_{N \rightarrow \infty} \frac{1}{n!}\int_{C_{0}}\cdots \int_{C_{0}} \det \left[ K_u^N(w_i, w_j)\right] _{i,j = 1}^n \prod_{i = 1}^{n}  \frac{dw_i}{2 \pi \iota } = \frac{1}{n!}\int_{C_{0}}\cdots \int_{C_{0}} \det \left[ K_u(w_i, w_j)\right] _{i,j = 1}^n \prod_{i = 1}^{n}  \frac{dw_i}{2 \pi \iota }.$$
By Hadamard's inequality we have that for each $n$ the above is bounded (in absolute value) by $\frac{n^{n/2}t^{-n}M_2^n}{n!}$. Consequently, by the Dominated Convergence Theorem we have that
$$\lim_{N \rightarrow \infty}\sum_{n = 1}^{\infty} \frac{1}{n!}\int_{C_{0}}\hspace{-3mm}\cdots \int_{C_{0}} \det \left[ K_u^N(w_i, w_j)\right] _{i,j = 1}^n \prod_{i = 1}^{n}  \frac{dw_i}{2 \pi \iota } =\sum_{n = 1}^{\infty} \frac{1}{n!}\int_{C_{0}}\hspace{-3mm}\cdots \int_{C_{0}} \det \left[ K_u(w_i, w_j)\right] _{i,j = 1}^n \prod_{i = 1}^{n}  \frac{dw_i}{2 \pi \iota }.$$
This concludes the proof of $2.$\\

We next wish to extend the result to a more general class of contours. Let $C$ be a positively oriented piecewise smooth simple contour contained in the annulus, described in the statement of the proposition. What we have proved so far is that 
\begin{equation}\label{deformal}
\mathbb{E}_{a,r,t} \left[ \frac{1}{((1-t)ut^{-\lambda_1'}; t)_\infty} \right] = 1 + \sum_{n = 1}^{\infty} \frac{1}{n!}\int_{C_0}\hspace{-3mm}\cdots \int_{C_0} \det \left[ K_u(w_i, w_j)\right] _{i,j = 1}^n \prod_{i = 1}^{n}  \frac{dw_i}{2 \pi \iota },
\end{equation}
where the latter sum is absolutely convergent. One readily verifies that $g_{w,w'}(t^s)$ is analytic in $w,w'$ on a neighborhood of $A_{\delta ,t}\times A_{\delta,t}$ and by the exponential decay of $\Gamma(-s)\Gamma(1 + s)(-(t^{-1} - 1)u)^s$ near $1/2 \pm \iota \infty$ the same is true for $K_u(w,w')$. It follows that $\det \left[ K_u(w_i, w_j)\right] _{i,j = 1}^n $ is analytic on a neighborhood of $A_{\delta ,t}^n$ and by Cauchy's theorem we may deform the contours $C_0$ in (\ref{deformal}) to $C$, without changing the value of the integrals. This is the result we wanted.
\end{proof}

%
\subsection{A formula suitable for asymptotics: GUE case}\hspace{2mm} \\

In this section we use Proposition \ref{prelimit} to derive an alternative $t$-Laplace transform, which is more suitable for asymptotic analysis in the GUE case. The following result makes references to two contours $\gamma_W(A)$ and $\gamma_Z(A)$, which depend on a real parameter $A\geq 0 $, as well as a function $S_{a,r}(\cdot)$, which we define below.
\begin{definition}\label{gammacont} For a parameter $A \geq 0$ define
$$\gamma_W(A) = \{ -A|y| + \iota y: y \in I\} \mbox{ and } \gamma_Z(A) = \{A|y| + \iota y: y \in I\}, \mbox{ where }I = \left[ -\pi,\pi\right].$$
The orientation is determined from $y$ increasing in $I$.
\end{definition}

\begin{definition}\label{S5DefFun}
For $a,r \in (0,1)$ define
$$S_{a,r}(z) :=\sum_{j = 0}^{\infty}  \log ( 1 + ar^j e^z) - \sum_{j = 0}^{\infty}  \log ( 1 + ar^j e^{-z}).$$
\end{definition}
The function $S_{a,r}$ plays a central role in our arguments and the properties that we will need are summarized in Section \ref{SSart}. We isolate the most basic facts about $S_{a,r}$ in a lemma below. The lemma appears again in Section \ref{SSart} as Lemma \ref{sartbasic}, where it is proved.
\begin{lemma}\label{sartbasicPrev}
Suppose that $\delta \in (0,1)$. Consider $r \in (0,1)$ and $a \in (0, 1-\delta]$. Then there exists $\Delta'(\delta) > 0$ such that $S_{a,r}(z)$ is well-defined and analytic on $D_\delta = \{z \in\mathbb{C} : |Re(z)| < \Delta'\}$ and satisfies
\begin{equation}
\exp( S_{a,r}(z))  =\prod_{j = 0}^{\infty}\frac{1 + ar^je^z}{1 + ar^je^{-z}} .
\end{equation}
\end{lemma}

\begin{proposition}\label{goodprelimitTW}
Suppose $a,r,t\in(0,1)$ and let $\delta > 0$ be such that $a < ( 1 - \delta)$. If $A > 0$ is sufficiently small (depending on $\delta$ and $t$) and $\gamma_W(A)$ and $\gamma_Z(A)$ are as in Definition \ref{gammacont}, then for $\zeta \in \mathbb{C}\backslash \mathbb{R}^+$ one has
$$\mathbb{E}_{a,r,t} \left[ \frac{1}{(\zeta t^{1-\lambda'_1}; t)_\infty} \right] = \det (I - \tilde K_\zeta)_{L^2(\gamma_W)}.$$
The kernel $\tilde K(W,W')$ has the integral representation
\begin{equation}\label{kernelF}
\tilde K_\zeta(W,W') = \frac{e^{W}}{2\pi \iota} \int_{\gamma_Z(A)}\frac{dZ  (-\zeta)^{f_t(Z,W) }}{e^{W'}- e^{Z}}G_{\zeta,t}(W,Z) \exp\left(S_{a,r}(Z) - S_{a,r}(W)\right).
\end{equation}
In the above formula, $S_{a,r}$ is as in Definition \ref{S5DefFun} and we have
\begin{equation}\label{Gandf}
G_{\zeta,t}(W,Z) := \sum_{ k \in \mathbb{Z}}\frac{\pi (-\log t)^{-1} (-\zeta)^{2\pi k \iota /(-\log t)}}{\sin (- \pi f_t(Z + 2\pi k\iota ,W))} \mbox{ and } f_t(Z,W) := \frac{Z-W}{-\log t}.
\end{equation}
\end{proposition}

\begin{proof}

We consider the contour $C_A : = \{ - t^{-1}e^{\iota \theta -A|\theta|} : \theta \in [-\pi, \pi]\}$, which is a positively oriented piecewise smooth contour. For $A > 0$ sufficiently small we know that $C_A$ is contained in the annulus $A_{\delta, t}$ in the statement of Proposition \ref{prelimit}. 
Consequently, from (\ref{prelimform}) we know that
$$\mathbb{E}_{a,r,t} \left[ \frac{1}{( (1-t) ut^{-\lambda'_1}; t)_\infty} \right] = 1 + \sum_{n = 1}^{\infty} \frac{1}{n!}\int_{C_{A}}\hspace{-3mm}\cdots \int_{C_{A}} \det \left[ K_u(w_i, w_j)\right] _{i,j = 1}^n \prod_{i = 1}^{n}  \frac{dw_i}{2 \pi \iota },$$
where $K_u(w,w')$ is as in (\ref{Ku}) and the above sum is absolutely convergent.  The $n$-th summand equals
$$\frac{1}{n!}\int_{-\pi}^{\pi}\hspace{-3mm}\cdots \int_{-\pi}^\pi \det \left[ K_u \left(-t^{-1}e^{\iota \theta_i -A|\theta_i|}, -t^{-1}e^{\iota \theta_j -A|\theta_j|}\right)\right] _{i,j = 1}^n \prod_{i = 1}^{n}  \frac{-t^{-1}e^{\iota\theta_i - A|\theta_i|}( \iota - Asign(\theta_i)) d\theta_i}{2 \pi \iota }.$$
Setting $y_i = \iota \theta_i - A|\theta_i|$ the above becomes
$$\frac{(-1)^n}{n!}\int_{\gamma_W(A)}\hspace{-3mm}\cdots \int_{\gamma_W(A)} \det \left[t^{-1}e^{y_i} K_u \left(-t^{-1}e^{y_i}, -t^{-1}e^{y_i} \right)\right] _{i,j = 1}^n \prod_{i = 1}^{n}  \frac{ dy_i}{2 \pi \iota }.$$
To conclude the proof it suffices to show that for $W,W' \in \gamma_W(A)$ and $\zeta = (t^{-1} - 1)u$ one has
\begin{equation}\label{S4kerneleq}
t^{-1}e^{W} K_u \left(-t^{-1}e^{W}, -t^{-1}e^{W'} \right) = \tilde K_\zeta(W,W').
\end{equation}
Setting $Z = (-\log t) s + W$, using the Euler Gamma reflection formula from (\ref{Euler}) and recalling $f_t(Z,W) = \frac{Z - W}{-\log t}$, we see that the LHS of (\ref{S4kerneleq}) equals
$$\frac{e^{W}}{2\pi \iota} \int_{\frac{-\log t}{2} + W -\iota\infty}^{\frac{-\log t}{2} + W + \iota\infty} \frac{(- \log t)^{-1}\pi dZ}{\sin (- \pi f_t(Z,W))}(-\zeta)^{f_t(Z,W)}  \frac{1}{e^{W'}- e^{Z}}\prod_{j = 0}^\infty \frac{(1+ ar^{j}e^{-W})(1 + ar^{j}e^{Z})}{(1 + ar^{j}e^{-Z})(1 + ar^{j}e^W)}.$$

If $W \in \gamma_W(A)$ we know that $Re \left[ \frac{-\log t}{2} + W\right] \in \left[\frac{-\log t}{2} - \pi A, \frac{-\log t}{2} \right]$. In addition, the only poles of the integrand for $Re(Z) > 0$ come from $  \frac{1 }{\sin (- \pi f_t(Z,W))}$ and are located at $W+ (-\log t)\mathbb{Z}.$ This implies that if $A$ is sufficiently small we may shift the $Z$- contour so that it passes through the point $A\pi$, without crossing any poles of the integrand (see Figure \ref{S4_3}).
\begin{figure}[h]
\centering
\begin{minipage}{.5\textwidth}
  \centering
  \includegraphics[width=0.9\linewidth]{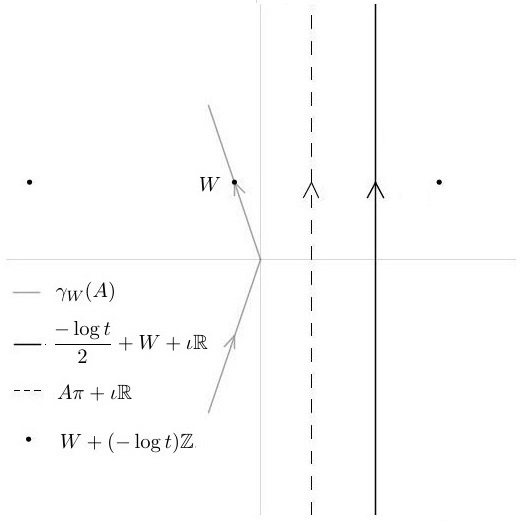}
\captionsetup{width=.9\linewidth}
  \caption{If $A$ is very small, no points of $W+ (-\log t)\mathbb{Z}$ fall between $A\pi + \iota \mathbb{R}$ and $\frac{-\log t}{2} + W + \iota \mathbb{R}$, when $W \in \gamma_W(A)$.}
  \label{S4_3}
\end{minipage}%
\begin{minipage}{.5\textwidth}
  \centering
  \includegraphics[width=0.9\linewidth]{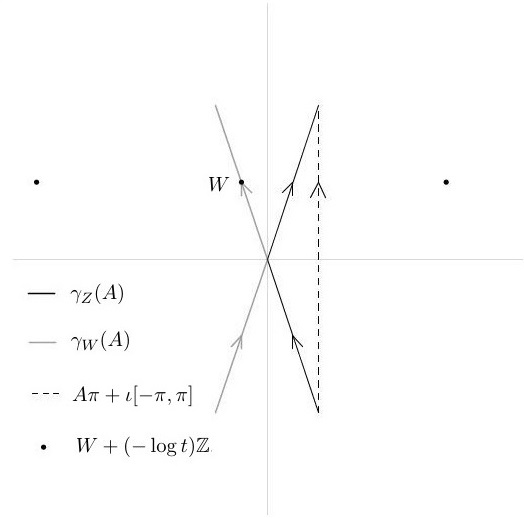}
\captionsetup{width=.9\linewidth}
  \caption{If $A$ is very small, no points of $W+ (-\log t)\mathbb{Z}$ fall between $A\pi + \iota [-\pi,\pi]$ and $\gamma_Z(A)$, when $W \in \gamma_W(A)$. }
  \label{S4_4}
\end{minipage}
\end{figure}
 The shift does not change the value of the integral by Cauchy's Theorem and the exponential decay of the integrand near $\pm \iota\infty$. Thus we get that the LHS of (\ref{S4kerneleq}) equals
$$\frac{e^{W}}{2\pi \iota} \int_{A\pi -\iota\infty}^{A\pi + \iota\infty} \frac{(- \log t)^{-1}\pi dZ}{\sin (- \pi f_t(Z,W))}(-\zeta)^{f_t(Z,W)}  \frac{1}{e^{W'}- e^{Z}}\prod_{j = 0}^\infty \frac{(1+ ar^{j}e^{-W})(1 + ar^{j}e^{Z})}{(1 + ar^{j}e^{-Z})(1 + ar^{j}e^W)}.$$

The next observation is that $e^{A\pi + \iota y}$ is periodic in $y$ with period $T = 2\pi$. Using this we see that the LHS of (\ref{S4kerneleq}) equals
$$ \frac{e^{W}}{2\pi \iota}  \sum_{ k \in \mathbb{Z}}\int_{A\pi -\iota T/2 + \iota k T }^{A\pi + \iota T/2 + \iota k T} \frac{(- \log t)^{-1}\pi dZ}{\sin (- \pi f_t(Z,W))}(-\zeta)^{f_t(Z,W)}  \frac{1}{e^{W'}- e^{Z}}\prod_{j = 0}^\infty \frac{(1+ ar^{j}e^{-W})(1 + ar^{j}e^{Z})}{(1 + ar^{j}e^{-Z})(1 + ar^{j}e^W)} = $$
$$ =\frac{e^{W}}{2\pi \iota}  \sum_{ k \in \mathbb{Z}}\int_{A\pi -\iota T/2 }^{A\pi + \iota T/2 } dZ\frac{(-\zeta)^{\iota k T/(-\log t)}(- \log t)^{-1}\pi }{\sin (- \pi f_t(Z + \iota k T,W))} \frac{(-\zeta)^{f_t(Z,W)} }{e^{W'}- e^{Z}}\prod_{j = 0}^\infty \frac{(1+ ar^{j}e^{-W})(1 + ar^{j}e^{Z})}{(1 + ar^{j}e^{-Z})(1 + ar^{j}e^W)}.$$
Let $(-\zeta) = re^{\iota \theta}$ with $|\theta| < \pi$. Then, using a similar argument as in (\ref{gzest}), we have for $|k| \geq 1$ 
\begin{equation}\label{Gest}
\left|\frac{(-\zeta)^{\iota k T/(-\log t)}}{\sin (- \pi f_t(Z + \iota k T,W))}\right| = \left|\frac{ e^{-\theta k T/(-\log t)}}{\sin (- \pi f_t(Z + \iota k T,W))}\right| \leq Ce^{|k| T(|\theta| - \pi)/(-\log t)},
\end{equation}
where $C$ is some positive constant, independent of $Z$ and $W$, provided $W \in \gamma_W(A)$, $|Im(Z)| \leq \pi$ and $Re(Z) = A\pi$. We observe the latter is summable over $k$. Additionally, we have 
$$\left|\frac{(-\zeta)^{f_t(Z,W) }}{e^{W'}- e^{Z}}\prod_{j = 0}^\infty \frac{(1+ ar^{j}e^{-W})(1 + ar^{j}e^{Z})}{(1 + ar^{j}e^{-Z})(1 + ar^{j}e^W)}\right| \leq \frac{1}{ e^{A\pi} - 1}\left| (-\zeta)^{f_t(Z,W) }\prod_{j = 0}^\infty \frac{(1+ ar^{j}e^{-W})(1 + ar^{j}e^{Z})}{(1 + ar^{j}e^{-Z})(1 + ar^{j}e^W)}\right|,$$
and the latter is bounded by some constant $M(\zeta, B)$, provided $Re(Z) = A\pi$ and $W \in \gamma_W(A)$. By Fubini's theorem, we may change the order of the sum and the integral and get that LHS of (\ref{S4kerneleq}) equals
$$\frac{e^{W}}{2\pi \iota} \int_{A\pi -\iota T/2}^{A\pi + \iota T/2}\frac{dZ  (-\zeta)^{f_t(Z,W) }}{e^{W'}- e^{Z}}\left[ \sum_{ k \in \mathbb{Z}}\frac{\pi (-\log t)^{-1} (-\zeta)^{\iota k T/(-\log t)}}{\sin (- \pi f_t(Z + \iota k T,W))}\right] \prod_{j = 0}^\infty \frac{(1+ ar^{j}e^{-W})(1 + ar^{j}e^{Z})}{(1 + ar^{j}e^{-Z})(1 + ar^{j}e^W)}.$$

From (\ref{Gest}) we see that $G_{\zeta,t}(W,Z)$, which is given by
$$\frac{\pi (-\log t)^{-1}}{\sin (- \pi f_t(Z,W))} + \sum_{|k| \geq 1}\frac{\pi (-\log t)^{-1}(-\zeta)^{\iota k T/(-\log t)}}{\sin (- \pi f_t(Z + \iota k T,W))},$$
is the sum of $\frac{\pi (-\log t)^{-1}}{\sin (- \pi f_t(Z,W))}$ and an analytic function in $Z$ in the region $D = \{Z \in \mathbb{C}: |Im(Z)| \leq \pi \mbox{ and }Re(Z) \geq 0\}$. In particular, the poles of $G_{\zeta,t}(W,Z)$ in $D$ are exactly at $W + (-\log t)\mathbb{N}$. If we now deform the contour $[A\pi - \iota\pi, A\pi + \iota\pi]$ to $\gamma_Z(A)$ (see Figure \ref{S4_4}) we will not cross any poles and from Cauchy's Theorem we will obtain that the LHS of (\ref{S4kerneleq}) equals
$$\frac{e^{W}}{2\pi \iota} \int_{\gamma_Z(A)}\frac{dZ  (-\zeta)^{f_t(Z,W) }}{e^{W'}- e^{Z}}G_{\zeta,t}(W,Z) \prod_{j = 0}^\infty \frac{(1+ ar^{j}e^{-W})(1 + ar^{j}e^{Z})}{(1 + ar^{j}e^{-Z})(1 + ar^{j}e^W)}.$$

From Lemma \ref{sartbasicPrev} (provided $A$ is sufficiently small so that $\gamma_Z(A), \gamma_W(A) \subset D_\delta$), we have that 
$$\prod_{j = 0}^\infty \frac{(1+ ar^{j}e^{-W})(1 + ar^{j}e^{Z})}{(1 + ar^{j}e^{-Z})(1 + ar^{j}e^W)} = \exp\left(S_{a,r}(Z) - S_{a,r}(W)\right).$$
Substituting this above we recognize the RHS of (\ref{S4kerneleq}).
\end{proof}

%
\subsection{Convergence of the $t$-Laplace transform (GUE case) and proof of Theorem \ref{TW} }\label{S43}\hspace{2mm}\\

Here we state the regime, in which we scale parameters and obtain an asymptotic formula for $\mathbb{E}_{a,r,t}\left[  \frac{1}{(\zeta t^{1-\lambda_1'}; t)_\infty}\right]$. The formula is analyzed below and used to prove Theorem \ref{TW}. One key reason we are considering the $t$-Laplace transform is that it asymptotically behaves like the expectation of an indicator function. The latter (as will be shown carefully below) allows one to obtain the limiting CDF of the properly scaled first column of a partition distributed according to the Hall-Littlewood measure with parameters  $a,r,t$ and match it with $F_{GUE}$ (see Definition \ref{TWDef}).\\

We summarize the limiting regime and some relevant expressions.
\begin{enumerate}[label = \arabic{enumi}., leftmargin=1.0cm]
\item We will let $r \rightarrow 1^{-}$ and keep $t \in (0,1)$ fixed.
\item We assume that $a$ depends on $r$ and for some $\delta > 0$ we have $\lim_{r \rightarrow 1^-} a(r) = a(1) \in (0,1-\delta]$.
\item We denote by $N(r) = \frac{1}{1-r}$, $M(r) = 2\sum_{k = 1}^\infty a(r)^k  \frac{(-1)^{k+1}}{1-r^k}$ and $\alpha = \left[ \frac{a(1)}{(1 + a(1))^2}\right]^{-1/3}$.
\end{enumerate}

\begin{equation}\label{zetadef}
\mbox{ For a given $x \in \mathbb{R}$ set $\zeta_x = - t^{M(r) + x\alpha^{-1} N(r)^{1/3}}$}.
\end{equation}

The following result is the key fact for the Tracy-Widom limit of the fluctuations of the first column of a partition distributed according to $\mathbb{P}_{a,r,t}$ in the GUE case. It shows that under the scaling regime described above the Fredholm determinant (and hence the $t$-Laplace transform) appearing in Proposition \ref{goodprelimitTW} converges to $F_{GUE}$. 
\begin{theorem} \label{mainThm}
Let $x \in \mathbb{R}$ be given and let $\zeta_x$ be given as in (\ref{zetadef}). If $A > 0$ is sufficiently small (depending on $\delta$ and $t$) then
\begin{equation}\label{mainLimit}
\lim_{r \rightarrow 1^{-}} \det( I - \tilde K_{\zeta_x})_{L^2(\gamma_W(A))} = F_{GUE}(x),
\end{equation}
where $F_{GUE}$ is the GUE Tracy-Widom distribution (see Definition \ref{TWDef}), $\gamma_W(A)$ is defined in Definition \ref{gammacont} and $\tilde K_{\zeta_x}$ is as in (\ref{kernelF}).
\end{theorem}
{\raggedleft In what follows we prove Theorem \ref{TW}, assuming the validity of Theorem \ref{mainThm}, whose proof is postponed until the next section.}\\

We begin by summarizing the key results from our previous work as well as recalling a couple of lemmas from the literature. From Proposition \ref{goodprelimitTW} and Theorem \ref{mainThm} we have that under the scaling described in the beginning of the section and any $x \in \mathbb{R}$ 
\begin{equation}\label{WCTW1}
\lim_{r \rightarrow 1^-} \mathbb{E}_{a,r,t} \left[ \frac{1}{(- t^{M(r) + \alpha^{-1}x N(r)^{1/3}} t^{1-\lambda'_1}; t)_\infty} \right] = F_{GUE}(x).
\end{equation}
Set $\xi_r := \alpha N(r)^{-1/3}\left ( \lambda'_1 - M(r)\right)$ and observe that (\ref{WCTW1}) is equivalent to
\begin{equation}\label{WCTW3}
\lim_{r \rightarrow 1^-} \mathbb{E}_{a,r,t} \left[ \frac{1}{((- t) \cdot t^{ - \left[ N(r)^{1/3} \alpha ^{-1} (\xi_r - x) \right]}; t)_\infty} \right] = F_{GUE}(x).
\end{equation}

The function that appears on the LHS under the expectation in (\ref{WCTW3}) has the following asymptotic property.
\begin{lemma}\label{seqf}
Fix a parameter $t\in (0,1)$. Then
\begin{equation}\label{seqfeq}
f_q(y) := \frac{1}{((- t)\cdot t^{qy}; t)_\infty} = \prod_{k = 1}^\infty \frac{1}{1 + t^{qy + k}}
\end{equation}
is increasing for all $q > 0$ and decreasing for all $q < 0$. For each $\delta > 0$ one has $f_q(y) \rightarrow {1}_{\{y > 0\}}$ uniformly on $\mathbb{R} \backslash [-\delta, \delta]$ as $q \rightarrow \infty$.
\end{lemma}
\begin{proof}
This is essentially Lemma 5.1 in \cite{FerVet}, but we present the proof for completeness. Each factor in the $t$-Pochhammer symbol $\frac{1}{1 + t^{qy + k}}$ is positive, increases in $y$ when $q > 0$ and decreases in $y$ when $q < 0$. This proves monotonicity.\\

Let $\delta > 0$ be given. If $y  < -\delta $ we have
\begin{equation}\label{seqf1}
0 \leq f_q(y) \leq \frac{1}{1 + t^{1 + qy}} \leq \frac{1}{1 + t^{1 - q\delta}} \rightarrow 0 \mbox{ as } q\rightarrow \infty.
\end{equation}
If $y > \delta$ we have
\begin{equation}\label{seqf2}
0 \geq \log f_q(y) \geq -\sum_{k = 1}^{\infty}\log \left[ 1 + t^{q \delta + k} \right] \rightarrow 0 \mbox{ as } q \rightarrow \infty,
\end{equation}
where the latter statement follows from the Dominated Convergence Theorem with dominating function $\log \left[ 1 + t^k \right]$. Exponentiating (\ref{seqf2}) and combining it with (\ref{seqf1}) proves the second part of the lemma.

\end{proof}

We will use the following elementary probability lemma (Lemma 4.1.39 of \cite{BorCor}).
\begin{lemma}\label{prob}
Suppose that $f_n$ is a sequence of functions $f_n: \mathbb{R} \rightarrow [0,1]$, such that for each $n$, $f_n(y)$ is strictly decreasing in $y$ with a limit of $1$ at $y = -\infty$ and $0$ at $y = \infty$. Assume that for each $\delta > 0$ one has on $\mathbb{R}\backslash[-\delta,\delta]$, $f_n \rightarrow {\bf 1}_{\{y < 0\}}$ uniformly. Let $X_n$ be a sequence of random variables such that for each $x \in \mathbb{R}$ 
$$\mathbb{E}[f_n(X_n - x)] \rightarrow p(x),$$
and assume that $p(x)$ is a continuous probability distribution function. Then $X_n$ converges in distribution to a random variable $X$, such that $\mathbb{P}(X < x) = p(x)$.
\end{lemma}

\begin{proof}(Theorem \ref{TW})
Let $r_n$ be a sequence converging to $1^-$ and set
$$f_n(y) = \frac{1}{((- t) \cdot t^{ - \left[ N(r_n)^{1/3} \alpha ^{-1} y\right]}; t)_\infty} \mbox{ and } X_n = \xi_{r_n}.$$

Lemma \ref{seqf} shows that $f_n$ satisfy the conditions of Lemma \ref{prob}. Consequently, Lemma \ref{prob} and (\ref{WCTW3}) show that $\xi_{r_n}$ converges weakly to the Tracy-Widom distribution. In particular, for each $x \in \mathbb{R}$ we have
\begin{equation}\label{WCTW4}
\lim_{r \rightarrow 1^-} \mathbb{P}_{a,r,t}  (\xi_r \leq x) = F_{GUE}(x).
\end{equation}

Consider $a(r) =  r^{(1 + |\lfloor \tau N(r) \rfloor|)/2}$. Since, $\lim_{r \rightarrow 1^-} r^{N} = e^{-1},$ we see that $\lim_{r \rightarrow 1^-} a(r) = a(1) = e^{-|\tau|/2} < 1$ ( whenever $\tau \neq 0$). This means that $\alpha^{-1} := \left[ \frac{a(1)}{(1 + a(1))^2}\right]^{1/3} =  \left[ \frac{e^{-|\tau|/2}}{(1 + e^{-|\tau|/2})^2}\right]^{1/3} =: \chi^{-1}$. From Section \ref{HL} we conclude that 
\begin{equation}\label{WCTW5}
\mathbb{P}_{HL}^{r,t} \left( \frac{ \lambda'_1( \lfloor \tau N(r) \rfloor)  - M(r) }{\chi^{-1} N^{1/3}} \leq x\right) = \mathbb{P}_{a,r,t} \left( \frac{ \lambda'_1 - M(r)}{\alpha^{-1} N^{1/3}} \leq x\right) = \mathbb{P}_{a,r,t}  (\xi_r \leq x),
\end{equation}
Combining (\ref{WCTW4}) and (\ref{WCTW5}) shows that if $\tau \neq 0$ one has
$$\lim_{r \rightarrow 1^-} \mathbb{P}_{HL}^{r,t} \left( \frac{ \lambda'_1( \lfloor \tau N(r) \rfloor)  - M(r) }{\chi^{-1} N^{1/3}} \leq x\right)= F_{GUE}(x).$$
In (\ref{delM}) we will show that $ M(r) = 2 N(r)\log (1 + a(1)) + O(1) = 2N(r) \log (1 +e^{-|\tau|/2}) + O(1).$ Substituting this above concludes the proof of the theorem.
\end{proof}

%
\subsection{Proof of Theorem \ref{mainThm}}\label{S44}\hspace{2mm}\\

We split the proof of Theorem \ref{mainThm} into four steps. In the first step we rewrite the LHS of (\ref{mainLimit}) in a suitable form for the application of Lemmas \ref{FDCT} and \ref{FDKCT}. In the second step we verify the pointwise convergence and in the third step we provide dominating functions, which are necessary to apply the lemmas. In the fourth step we obtain a limit for the LHS of (\ref{mainLimit}), subsequently we use a result from \cite{BCF}, to show that the limit we obtained is in fact $F_{GUE}$.

In Steps 2 and 3 we will require some estimates, which we summarize in Lemmas \ref{cubicEst} and \ref{techies} below. The proofs are postponed until Section \ref{SSart}.
\begin{lemma}\label{cubicEst}
Let $A > 0$ be sufficiently small. Then for all large $N$ we have
\begin{equation}\label{cubicZ1}
Re(S_{a,r} (z) - M(r)z) \leq -cN|z|^3 \mbox{ for all }  z \in \gamma_Z(A) \mbox{ and }
\end{equation}
\begin{equation}\label{cubicW1}
Re(S_{a,r} (z) - M(r)z) \geq cN|z|^3 \mbox{ for all } z \in \gamma_W(A). \\
\end{equation}
In the above $c > 0$ depends on $A$ and $\delta$. In addition, we have
\begin{equation}\label{cubicZ2}
 Re(S_{a,r} (z) - M(r)z ) = O(1) \mbox{ if } |z| = O(N^{-1/3}) \mbox{ and } \\
\end{equation}
\begin{equation}\label{cubicW2}
\lim_{N \rightarrow \infty}S_{a,r}(N^{-1/3}u) - M(r)N^{-1/3}u = u^3\alpha^{-3}/3  \mbox{ for all } u \in \mathbb{C}. \\
\end{equation}
\end{lemma}

\begin{lemma}\label{techies}
Let $t, u, U \in (0,1)$ be given such that $0 < u < U < \min ( 1 , -\log t/10)$. Suppose that $z,w \in \mathbb{C}$ are such that $Re(w) \in [- U, 0]$, $Re(z) \in [u, U]$. Then there exists a constant $C > 0$, depending on $t$ such that the following hold
\begin{equation}\label{yellow1}
\left| \frac{1}{e^z - e^w}\right| \leq Cu^{-1} \mbox{ and } \sum_{k \in \mathbb Z} \left| \frac{1}{\sin(-\pi f_t(z + 2\pi\iota k,w))}\right| \leq Cu^{-1}, \mbox{ where $f_t(z,w) = \frac{z -w}{-\log t}.$}
\end{equation}

\end{lemma}

{\raggedleft {\bf Step 1.}} For $A > 0$ define $\gamma'_W(A) = \{ -A|y| + \iota y: y \in \mathbb{R}\} \mbox{ and } \gamma'_Z(A) = \{A|y| + \iota y: y \in \mathbb{R}\}.$ Suppose $A > 0$ is sufficiently small, so that Proposition \ref{goodprelimitTW} holds. We consider the change of variables $z_i = N^{1/3}Z_i$ and $w_i = N^{1/3}W_i$ and observe that the LHS of (\ref{mainLimit}) can be rewritten as $\det (I -  \tilde K_x^N)_{L^2(\gamma_W'(A))}$, where
\begin{equation}\label{GUEeq1}
\begin{split} &\tilde K_x^N(w,w') =   \int_{\gamma_Z'(A)} g^{N,x}_{w,w'}(z)\frac{dz}{2\pi \iota}, \mbox{ and }
 g^{N,x}_{w,w'}(z) ={\bf 1}_{\{\max(|Im(w)|, |Im(w')|, |Im(z)|) \leq N^{1/3}\pi\}} \times 
\\ &\frac{e^{N^{-1/3}w}N^{-2/3} }{e^{N^{-1/3}w'}- e^{N^{-1/3}z}}G_{\zeta_x,t}(N^{-1/3}w,N^{-1/3}z)\frac{\exp(S_{a,r}(N^{-1/3}z) - MN^{-1/3}z - x\alpha^{-1} z)}{\exp(S_{a,r}(N^{-1/3}w) - MN^{-1/3}w - x\alpha^{-1} w)}  .
\end{split}
\end{equation}
We deform the contour $\gamma'_Z(A)$ inside the disc of radius $A^{-1}$ so that it is still piecewise smooth and contained in $\{z \in \mathbb{C}: Re(z) \geq 1/2\}$. Observe that the poles of $g^{N,x}_{w,w'}(z)$ in the right complex half-plane come from $G_{\zeta_x,t}$ and are thus located at least a distance of order $N^{1/3}$ from the imaginary axis. The later implies that if we perform, a deformation inside a disc of radius $O(1)$ we will not cross any poles provided $N$ is sufficiently large. In particular, our deformation does not change the value of $g^{N,x}_{w,w'}$ for all large $N$ by Cauchy's Theorem. We will continue to call the new contour by $\gamma'_Z(A)$. Deforming the contour has the advantage of shifting integration away from the singularity point $0$.\\

{\raggedleft {\bf Step 2.}} Let us now fix $w,w' \in \gamma'_W(A)$ and $z \in \gamma'_Z(A)$ and show that 
\begin{equation}\label{blue0}
\lim_{N \rightarrow \infty} g^{N,x}_{w,w'}(z)  =  g^{\infty,x}_{w,w'}(z) \mbox{, where } g^{\infty,x}_{w,w'}(z): = \frac{\exp(\alpha^{-3} z^3/3 - \alpha^{-3}w^3/3 - x\alpha^{-1}z + x\alpha^{-1}w)}{(w-z)(w'-z)}.
\end{equation}
One readily observes that
\begin{equation}\label{blue1}
\lim_{N \rightarrow \infty } e^{N^{-1/3}w}\frac{ {\bf 1}_{\{\max(|Im(w)|, |Im(w')|, |Im(z)|) \leq N^{1/3}\pi\}}} {N^{1/3}\left( e^{N^{-1/3}w'}- e^{N^{-1/3}z}\right)} = \frac{1}{w' - z}
\end{equation}
Using (\ref{cubicW2}) we get
\begin{equation}\label{blue2}
\lim_{N \rightarrow \infty}\frac{\exp(S_{a,r}(N^{-1/3} z) - MN^{-1/3}z- x\alpha^{-1} z)}{\exp(S_{a,r}(N^{-1/3}w) - MN^{-1/3}w - x\alpha^{-1} w)} = \exp(\alpha^{-3}(z^3/3 - w^3/3) - x\alpha^{-1}z+ x\alpha^{-1}w).
\end{equation}
From (\ref{Gandf})  we have
\begin{equation}\label{blue3}
 N^{-1/3}G_{\zeta_x,t}(N^{-1/3} w,N^{-1/3} z) =N^{-1/3} \sum_{ k \in \mathbb{Z}}\frac{\pi (-\log t)^{-1} (-\zeta_x)^{2\pi k \iota /(-\log t)}}{\sin (- \pi f_t(N^{-1/3}z + 2\pi k\iota ,N^{-1/3}w))}.
\end{equation}
Using a similar argument as in (\ref{gzest}) we see that for $|k| \geq 1$ and all large $N$ one has
$$\left| \frac{\pi (-\log t)^{-1} (-\zeta_x)^{2\pi k \iota /(-\log t)}}{\sin (- \pi f_t(N^{-1/3}z + 2\pi k\iota ,N^{-1/3}w))} \right| \leq C e^{-2|k|\pi/(-\log (t))}.$$
The latter is summable over $|k| \geq 1$ and killed by $N^{-1/3}$ in (\ref{blue3}). We see that the only non-trivial contribution in (\ref{blue3}) comes from $k = 0$ and so
\begin{equation}\label{blue4}
\lim_{N \rightarrow \infty}N^{-1/3}G_{\zeta_x,t}(N^{-1/3} w,N^{-1/3} z)  = \lim_{N \rightarrow \infty}N^{-1/3}\frac{\pi (-\log t)^{-1}}{\sin \left( \frac{\pi N^{-1/3}}{- \log t}(w - z) \right)} =\frac{1}{w - z}.
\end{equation}
Equations  (\ref{blue1}), (\ref{blue2}) and (\ref{blue4}) imply (\ref{blue0}).\\

{\raggedleft {\bf Step 3.}} We now proceed to find estimates of the type necessary in Lemma \ref{FDKCT} for the functions $g^{N,x}_{w,w'}(z)$. If $z$ is outside of the disc of radius $A^{-1}$ (so lies on the undeformed portion of $\gamma'_Z(A)$) and $|Im(z)| \leq \pi N^{1/3}$ the estimates of (\ref{cubicZ1}) are applicable (provided $A$ is small enough) and so we obtain
\begin{equation}\label{cutoff1}
|\exp(S_{a,r}(N^{-1/3} z) - MN^{-1/3}z- x\alpha^{-1} z)| \leq C\exp(-c|z|^3 + |x\alpha^{-1}z|),
\end{equation}
where $C,c$ are positive constants. Next suppose $z$ is contained the disc of radius $A^{-1}$ around the origin (i.e. lies on the portion of $\gamma'_Z(A)$ we deformed). From (\ref{cubicW2}) we know that $ S_{a,r}(N^{-1/3} z) - MN^{-1/3}z$ is $O(1)$. This implies that $ |\exp(S_{a,r}(N^{-1/3} z) - MN^{-1/3}z- x\alpha^{-1} z)|$ is bounded and the estimate (\ref{cutoff1}) continues to hold with possibly a bigger $C$. 

If $w \in \gamma_W'(A)$ and $|Im(w)| \leq \pi N^{1/3}$ the estimates of (\ref{cubicW1}) are applicable (provided $A$ is small enough) and we obtain
\begin{equation}\label{cutoff2}
|\exp(-S_{a,r}(N^{-1/3} w) +MN^{-1/3}w+ x\alpha^{-1} w)| \leq C\exp(-c|w|^3 + |x\alpha^{-1}w|),
\end{equation}
for some $C,c >0$. 

If $A$ is sufficiently small so that $A\pi < \min(1, -\log t/ 10)$, then the estimates in Lemma \ref{techies} hold (with $u = (1/2)N^{-1/3}$ and $U = A\pi$), provided  $\max(|Im(w)|, |Im(w')|, |Im(z)|) \leq N^{1/3}\pi $, $z \in \gamma'_Z(A)$ and $w', w\in \gamma'_W(A)$. Consequently, for some positive constant $C$ we have\\
\begin{equation}\label{cutoff3}
\left| \frac{N^{-1/3}}{e^{N^{-1/3}w'}- e^{N^{-1/3}z}}N^{-1/3} G_{\zeta_x,t}(N^{-1/3}w,N^{-1/3}z) \right| \leq C.
\end{equation}

Observe that $e^{N^{-1/3}w} = O(1)$ when $|Im(w)| \leq \pi N^{1/3}$ and $w \in \gamma_W'(A)$. Combining the latter with (\ref{cutoff1}), (\ref{cutoff2}) and (\ref{cutoff3}) we see that whenever $\max(|Im(w)|, |Im(w')|, |Im(z)|) \leq N^{1/3}\pi$, $z \in \gamma'_Z(A)$ and $w', w\in \gamma'_W(A)$  we have
\begin{equation}\label{cutoff4}
|g^{N,x}_{w,w'}(z)| \leq C\exp(-c|w|^3 + |x\alpha^{-1}w|)\exp(-c|z|^3 + |x\alpha^{-1}z|),
\end{equation}
where $C,c$ are positive constants. Since $g^{N,x}_{w,w'}(z) = 0$ when $\max(|Im(w)|, |Im(w')|, |Im(z)|) > N^{1/3}\pi$ we see that (\ref{cutoff4}) holds for all $z \in \gamma'_Z(A)$ and $w', w\in \gamma'_W(A)$. \\

{\raggedleft{\bf Step 4.}} We may now apply Lemma \ref{FDKCT} to the functions $g^{N,x}_{w,w'}(z)$ with $F_1(w) = C\exp(-c|w|^3 + |x\alpha^{-1}w|) = F_2(w)$ and $\Gamma_1 = \gamma_W'(A)$, $\Gamma_2 = \gamma_Z'(A)$. Notice that the functions $F_i$ are integrable on $\Gamma_i$ by the cube in the exponential. As a consequence we see that if we set $\tilde K^{\infty}_x(w,w') :=  \int_{\gamma_Z'(A)} g^{\infty,x}_{w,w'}(z)\frac{dz}{2\pi \iota}$, then $\tilde K^N_x$ and $\tilde K^\infty_x$ satisfy the conditions of Lemma \ref{FDCT}, from which we conclude that 
\begin{equation}\label{cutoff5}
\lim_{r \rightarrow 1^{-}} \det( I - \tilde K_{\zeta_x})_{L^2(\gamma_W(A))} = \det( I - \tilde K^\infty_x)_{L^2(\gamma'_W(A))}.
\end{equation}
What remains to be seen is that $\det( I - \tilde K^\infty_x)_{L^2(\gamma'_W(A))} = F_{GUE}(x)$.\\

We have that $\det(I - \tilde K^\infty_{x})_{L^2(\gamma'_W)} = 1 + \sum_{n = 1}^\infty \frac{(-1)^n}{n!}H(n) ,$ where
$$H(n) = \sum_{\sigma \in S_n} sign(\sigma) \int_{\gamma'_W} ... \int_{\gamma'_W} \int_{\gamma'_Z} ... \int_{\gamma'_Z}  \prod_{i = 1}^n\frac{\exp(\alpha^{-3} Z_i^3/3 - \alpha^{-3}W_i^3/3 - x\alpha^{-1}Z_i + x\alpha^{-1}W_i)}{(W_i-Z_i)(W_{\sigma(i)} - Z_i)} \frac{dW_i}{2\pi \iota} \frac{dZ_i}{2\pi \iota}.$$
Consider the change of variables $z_i = \alpha^{-1} Z_i$, $w_i =  \alpha^{-1}  W_i$. Then we have
$$H(n) = \sum_{\sigma \in S_n} sign(\sigma) \int_{\gamma'_W} ... \int_{\gamma'_W} \int_{\gamma'_Z} ... \int_{\gamma'_Z}  \prod_{i = 1}^n\frac{\exp( z_i^3/3 - w_i^3/3 - xz_i + xw_i)}{(w_i-z_i)(w_{\sigma(i)} - z_i)} \frac{dw_i}{2\pi \iota} \frac{dz_i}{2\pi \iota}.$$
Consequently, we see that
$$\det(I - \tilde K^{\infty}_{x})_{L^2(\gamma'_W)} = \det (I + \tilde K_{Ai})_{L^2(\gamma'_W)},$$
where 
\begin{equation}\label{almostAiry}
\tilde K_{Ai}(w,w') =  \int_{\gamma'_Z} \frac{\exp( z^3/3 - w^3/3 - xz + xw)}{(w-z)(z - w')} \frac{dz}{2\pi \iota}.
\end{equation}
The proof of Lemma 8.6 in \cite{BCF} can now be repeated verbatim to show that
$$\det (I + \tilde K_{Ai})_{L^2(\gamma'_W)} = \det (I - K_{Ai})_{L^2(x,\infty)} = F_{GUE}(x).$$ 
This suffices for the proof.

\section{CDRP asymptotics }\label{Section7}
In this section, we obtain alternative formulas for the $t$-Laplace transform of $t^{1-\lambda_1'}$, with $\lambda$ distributed according to the Hall-Littlewood measure with parameters $a,r,t \in (0,1)$ (see Section \ref{HL}), which are more suitable for asymptotics in the CDRP case. Subsequently, we analyze the formulas that we get in the limiting regime $r,t \rightarrow 1^-$, and prove Theorem \ref{TCDRP}. In what follows, we will denote by $\mathbb{P}_{a,r,t}$ and $\mathbb{E}_{a,r,t}$ the probability distribution and expectation with respect to the Hall-Littlewood measure with parameters $a,r,t \in (0,1)$.

%
\subsection{A formula suitable for asymptotics: CDRP case}\hspace{2mm}\\ 

In this section we use Proposition \ref{prelimit} to derive an alternative representation for \\
$\mathbb{E}_{a,r,t} \left[ \frac{1}{(\zeta  t^{1-\lambda'_1}; t)_\infty} \right]$. In what follows we will make reference to the following contours
\begin{definition}\label{newcontours}
For $t \in (0,1)$ define
$$\gamma^t_- = \{-1/4 + \iota y: y \in [ - \pi (-\log t)^{-1}, \pi (-\log t)^{-1}] \},  \hspace{2mm} \gamma^t_+ = \{1/4 + \iota y: y \in [ - \pi (-\log t)^{-1}, \pi (-\log t)^{-1}] \},$$
$$ \gamma_- = \{-1/4 + \iota y: y \in \mathbb{R}\} \mbox{ and } \gamma_+ = \{1/4 + \iota y: y \in \mathbb{R}\}.$$
All contours are oriented upward.
\end{definition}

The following proposition is very similar to Proposition \ref{goodprelimitTW} and will be the starting point of our proof of Theorem \ref{TCDRP} the same way Proposition \ref{goodprelimitTW} was the starting point of the proof of Theorem \ref{TW}.
\begin{proposition}\label{goodprelimitCDRP}
Suppose $a,r,t\in(0,1)$ and let $\delta > 0$ be such that $a < ( 1 - \delta)$. If $t$ is sufficiently close to $1^-$ then for $\zeta \in \mathbb{C}\backslash \mathbb{R}^+$ one has
$$\mathbb{E}_{a,r,t} \left[ \frac{1}{(\zeta t^{1-\lambda'_1}; t)_\infty} \right] = \det (I - \hat K_\zeta)_{L^2( \gamma_-^t)}.$$
The kernel $\hat K_\zeta(W,W')$ has the integral representation
\begin{equation}\label{kernelart}
\hat K_\zeta(W,W') = \frac{t^{-W}}{2\pi \iota}\int_{\gamma^t_+}G_\zeta(W,Z) \frac{(-\log t)(-\zeta)^{Z-W} dZ}{ t^{-W'}- t^{-Z}}\frac{\exp\left(S_{a,r}((-\log t)Z)\right)}{ \exp\left( S_{a,r}((-\log t)W)\right)},
\end{equation}
where $G_{\zeta}(W,Z) = \sum_{k \in \mathbb{Z}}\frac{\pi(-\zeta)^{-2\pi k\iota /\log t}}{\sin(\pi(W - Z) + 2\pi k\iota /\log t)}$,  and the contours $\gamma_-^t$ and $\gamma^t_+$ are as in Definition \ref{newcontours}.
\end{proposition}
\begin{proof}
We consider the contour $C : = \{ - t^{-3/4}e^{\iota \theta} : \theta \in [-\pi, \pi]\}$, which is a positively oriented smooth contour, contained in the annulus $A_{\delta, t}$ in the statement of Proposition \ref{prelimit} for $t$ sufficiently close to $1^-$.  Consequently, from (\ref{prelimform}) we know that
$$\mathbb{E}_{a,r,t} \left[ \frac{1}{( (1-t) ut^{-\lambda'_1}; t)_\infty} \right] = 1 + \sum_{n = 1}^{\infty} \frac{1}{n!}\int_{C}\cdots \int_{C} \det \left[ K_u(w_i, w_j)\right] _{i,j = 1}^n \prod_{i = 1}^{n}  \frac{dw_i}{2 \pi \iota },$$
where $K_u(w,w')$ is as in (\ref{Ku}) and the above sum is absolutely convergent. The $n$-th summand equals
$$\frac{1}{n!}\int_{-\pi}^{\pi}\hspace{-3mm}\cdots \int_{-\pi}^\pi \det \left[ K_u \left(-t^{-3/4}e^{\iota \theta_i }, -t^{-3/4}e^{\iota \theta_j }\right)\right] _{i,j = 1}^n \prod_{i = 1}^{n}  \frac{-t^{-3/4}\iota e^{\iota\theta_i }d\theta_i}{2 \pi \iota }.$$
Setting $y_i = (-1/4) + \iota \theta_i/ (-\log t) $, the above becomes
$$\frac{(-1)^n}{n!}\int_{\gamma_-^t}\cdots \int_{\gamma_-^t} \det \left[ K_u \left(-t^{-3/4}t^{-y_i-1/4}, -t^{-3/4}t^{-y_j-1/4} \right)\right] _{i,j = 1}^n \prod_{i = 1}^{n}  \frac{t^{-3/4}t^{-y_i - 1/4}(- \log t) dy_i}{2 \pi \iota },$$
which can be rewritten as
$$\frac{(-1)^n}{n!}\int_{\gamma_-^t}\cdots \int_{\gamma_-^t} \det \left[(- \log t)t^{-1}t^{-y_i} K_u \left(-t^{-1}t^{-y_i}, -t^{-1}t^{-y_j} \right)\right] _{i,j = 1}^n \prod_{i = 1}^{n}  \frac{ dy_i}{2 \pi \iota },$$
and the latter is still absolutely summable over $n$. \\

To conclude the proof it suffices to show that for $W,W' \in \gamma_-^t$ and $\zeta = (t^{-1} - 1)u$ one has
\begin{equation}\label{kerneleq}
(- \log t)t^{-1}t^{-W} K_u \left(-t^{-1}t^{-W}, -t^{-1}t^{-W'} \right) = \hat K_\zeta(W,W').
\end{equation}
We observe that the LHS of (\ref{kerneleq}) equals
$$\frac{(- \log t)t^{-1}t^{-W}}{2\pi \iota} \int_{1/2 -\iota\infty}^{1/2 + \iota\infty}ds   \frac{\Gamma(-s)\Gamma(1+s) (-\zeta)^s}{ t^{-1}t^{-W'}-t^{-1}t^{-W}t^{-s}}\prod_{j = 0}^\infty \frac{(1+ ar^{j}t^{W})(1 + ar^{j}t^{-W}t^{-s })}{(1 + ar^{j}t^{W}t^{s})(1 + ar^{j}t^{-W})}.$$
We set $Z = s + W$,  and use that $Re(W) = -\frac{1}{4}$ for $W \in \gamma_-^t$ together with Euler's Gamma reflection formula (\ref{Euler}) to see that the above equals
$$\frac{t^{-W}}{2\pi \iota}  \int_{\gamma_+}\frac{\pi dZ}{\sin ( \pi(W-Z))} \frac{(-\log t)(-\zeta)^{Z-W}}{ t^{-W'}- t^{-Z}}\prod_{j = 0}^\infty \frac{(1+ ar^{j}t^{W})(1 + ar^{j}t^{-Z})}{(1 + ar^{j}t^{Z})(1 + ar^{j}t^{-W})}.$$

We observe that $t^{\iota s}$ is periodic in $s$ with period $T = \frac{2\pi}{-\log t}$. This allows us to rewrite the above formula as
$$\sum_{k \in \mathbb{Z}} \frac{t^{-W}}{2\pi \iota} \int_{\gamma_+^t}  \frac{\pi (-\zeta)^{\iota kT} }{\sin ( \pi(W- \iota kT- Z))} \frac{(-\log t)(-\zeta)^{Z-W}dZ}{ t^{-W'}- t^{-Z}}\prod_{j = 0}^\infty \frac{(1+ ar^{j}t^{W})(1 + ar^{j}t^{-Z})}{(1 + ar^{j}t^{Z})(1 + ar^{j}t^{-W})}.$$
Let $(-\zeta) = re^{\iota \theta}$ with $|\theta| < \pi$. Then, using a similar argument as in (\ref{gzest}), we have for $|k| \geq 1$ 
\begin{equation}\label{Gest2}
\left|\frac{\pi(-\zeta)^{\iota k T}}{\sin ( \pi(W+ \iota kT- Z)}\right| = \left|\frac{ \pi e^{-\theta k T}}{\sin (\pi(W+ \iota kT- Z))}\right| \leq Ce^{|k| T(|\theta| - \pi)},
\end{equation}
where $C$ is some positive constant, independent of $Z$ and $W$, provided $Z \in \gamma_+^t$ and $W\in \gamma_-^t$. The latter is clearly summable over $k$, which allows us to change the order of the sum and the integrals above and conclude that the LHS of (\ref{kerneleq}) equals
$$\frac{t^{-W}}{2\pi \iota}\int_{\gamma_+^t} \left[\sum_{k \in \mathbb{Z}} \frac{\pi(-\zeta)^{\iota k T}}{\sin ( \pi(W+ \iota kT- Z)}\right] \frac{(-\log t)(-\zeta)^{Z-W}dZ}{ t^{-W'}- t^{-Z}}\prod_{j = 0}^\infty \frac{(1+ ar^{j}t^{W})(1 + ar^{j}t^{-Z})}{(1 + ar^{j}t^{Z})(1 + ar^{j}t^{-W})}.$$
From Lemma \ref{sartbasicPrev} we have that if $t$ is sufficiently close to $1$ (so that $(-\log t)z \in D_\delta$ when $|Re(z)| = 1/4$) we have 
$$\prod_{j = 0}^\infty \frac{(1+ ar^{j}t^{W})(1 + ar^{j}t^{-Z})}{(1 + ar^{j}t^{Z})(1 + ar^{j}t^{-W})}= \frac{\exp\left(S_{a,r}((-\log t)Z)\right)}{ \exp\left( S_{a,r}((-\log t)W)\right)}.$$
Substituting this above we see that the LHS of (\ref{kerneleq}) equals
$$\frac{t^{-W}}{2\pi \iota}\int_{\gamma_+^t} \left[\sum_{k \in \mathbb{Z}} \frac{\pi(-\zeta)^{\iota k T}}{\sin ( \pi(W+ \iota kT- Z)}\right] \frac{(-\log t)(-\zeta)^{Z-W}dZ}{ t^{-W'}- t^{-Z}} \frac{\exp\left(S_{a,r}((-\log t)Z)\right)}{ \exp\left( S_{a,r}((-\log t)W)\right)},$$
which equals the RHS of (\ref{kerneleq}) once we identify the sum in the square brackets with $G_\zeta(W,Z)$.
\end{proof}

%
\subsection{Convergence of the $t$-Laplace transform (CDRP case) and proof of Theorem \ref{TCDRP}} \label{S72} \hspace{2mm}\\

Here we state the regime, in which we scale parameters and obtain an asymptotic formula for $\mathbb{E}_{a,r,t}\left[  \frac{1}{(\zeta t^{1-\lambda_1'}; t)_\infty}\right]$ in the CDRP case. The formula is analyzed below and used to prove Theorem \ref{TCDRP}. In the CDRP case the $t$-Laplace transform asymptotically behaves like the usual Laplace transform. The latter (as will be shown carefully below) allows one to obtain the limiting CDF of the properly scaled first column of of a partition distributed according to the Hall-Littlewood measure with parameters  $a,r,t$ and match it with $F_{CDRP}$ (see Definition \ref{TWDef}).\\

We summarize the limiting regime and some relevant expressions.
\begin{enumerate}[label = \arabic{enumi}., leftmargin=1.0cm]
\item We fix a positive parameter $\kappa$ and let $r \rightarrow 1^-$ and $t$ $\rightarrow 1^{-}$ so that $\kappa = \frac{- \log t}{(1-r)^{1/3}}$.
\item We assume that $a$ depends on $r$ and for some $\delta > 0$ we have $\lim_{r \rightarrow 1^-} a(r) = a(1) \in (0,1-\delta]$.
\item We denote by $N(r) = \frac{1}{1-r}$, $M(r) = 2\sum_{k = 1}^\infty (-1)^{k+1}a(r)^k\frac{1}{1 - r^k} $ and $\alpha = \left[ \frac{a(1)}{(1 + a(1))^2}\right]^{-1/3}$.
\end{enumerate}

\begin{equation}\label{zetadefC}
\mbox{ For a given $x \in \mathbb{R}$ set $\zeta_x = - t^{M(r) - x\kappa^{-1} N(r)^{1/3}}$} .
\end{equation}

The following result is the key fact for the limiting fluctuations of the first column of a partition distributed according to the Hall-Littlewood measure with parameters $a,r,t$ in the CDRP case. It shows that under the scaling regime described above the Fredholm determinant (and hence the $t$-Laplace transform) appearing in Proposition \ref{goodprelimitCDRP} converges to the Laplace transform of $\mathcal{F}(T,0) + T/24$ (see Definition \ref{TWDef} and equation (\ref{KCDRP})). The latter, as demonstrated below, implies convergence of the usual Laplace transforms and leads to a weak convergence necessary for the proof of Theorem \ref{TCDRP}.
\begin{theorem} \label{mainThm2}
Let $x \in \mathbb{R}$ be given and let $\zeta_x$ be given as in (\ref{zetadefC}). Then we have
\begin{equation}\label{mainLimitC}
\lim_{r \rightarrow 1^{-}} \det( I - \hat K_{\zeta_x})_{L^2(\gamma^t_-)} = \det( I -  K_{CDRP})_{L^2(\mathbb{R}^+)},
\end{equation}
where $K_{CDRP}$ is given in (\ref{KCDRPDef}) with $T = 2\kappa^3 \alpha^{-3}$, $\gamma^t_-$ is as in Definition \ref{newcontours}, and $\hat K_{\zeta_x}$ is as in (\ref{kernelart}).
\end{theorem}
{\raggedleft In what follows we prove Theorem \ref{TCDRP}, assuming the validity of Theorem \ref{mainThm2}, whose proof is postponed until the next section.}\\

We begin by summarizing the key results from our previous work that we will use as well as stating a couple of lemmas.
From Proposition \ref{goodprelimitCDRP} and Theorem \ref{mainThm2} we have that under the scaling described in the beginning of this section and any $x \in \mathbb{R}$ 
\begin{equation}\label{WCC1}
\lim_{r \rightarrow 1^-} \mathbb{E}_{a,r,t} \left[ \frac{1}{((-t)\cdot  t^{M(r) - \kappa^{-1}x N(r)^{1/3}}t^{-\lambda'_1}; t)_\infty} \right] = \det( I - K_{CDRP})_{L^2(\mathbb{R}^+)}.
\end{equation}
Set $\hat \xi_r := (-\log t)\left( \lambda_1' - M(r)\right) - \log( 1-t)$ and observe that (\ref{WCC1}) is equivalent to
\begin{equation}\label{WCC3}
\lim_{r \rightarrow 1^-} \mathbb{E}_{a,r,t} \left[ \frac{1}{((- t)(1-t) \cdot e^{\hat\xi_r + x}; t)_\infty} \right]  = \det( I - K_{CDRP})_{L^2(\mathbb{R}^+)}.
\end{equation}

The function that appears on the LHS under the expectation in (\ref{WCC3}) has the following asymptotic property.
\begin{lemma}\label{seqfC}
For $ t\in (0,1)$ and $x \geq 0$ let 
\begin{equation}\label{seqfCeq}
g_t(x) := \frac{1}{((- t)(1-t) x; t)_\infty} = \prod_{k = 1}^\infty \frac{1}{1 + (1-t)x t^ k}.
\end{equation}
Then $g_t (x) \rightarrow e^{-x} $ uniformly on $\mathbb{R}_{\geq 0}$ as $t \rightarrow 1^-$.
\end{lemma}

\begin{proof}
From the monotonicity of $g_t(x)$ and $e^{-x}$ it suffices to show the result only for compact subsets of $\mathbb{R}_{\geq 0}$. Using (10.2.7) in \cite{Andrews} one has that $\frac{1}{(-(1-t) x; t)_\infty} \rightarrow e^{-x} $ uniformly on compact subsets of $\mathbb{R}_{\geq 0}$ as $t \rightarrow 1^-$. Consequently,
$$g_t(x) = \frac{1 + (1-t)x}{(-(1-t) x; t)_\infty} =  \frac{1 }{(-(1-t) x; t)_\infty} + \frac{ (1-t)x}{(-(1-t) x; t)_\infty}$$
also converges uniformly to $e^{-x}$ on compact subsets $\mathbb{R}_{\geq 0}$ as $t \rightarrow 1^-$.
\end{proof}

We will use the following elementary probability lemma.
\begin{lemma}\label{probC}
Suppose that $f_n$ is a sequence of functions, $f_n: \mathbb{R}_{\geq 0} \rightarrow [0,1]$, such that $f_n(x) \rightarrow e^{-x}$ uniformly on $\mathbb{R}_{\geq 0}$. Let $X_n$ be a sequence of non-negagive random variables such that for each  $c > 0$ one has
$$\lim_{n \rightarrow \infty}\mathbb{E}[f_n(cX_n)] = p(c),$$
and assume that $p(c) = \mathbb{E} \left[ e^{-cX}\right]$ for some non-negative random variable $X$. Then we have
$$\lim_{n \rightarrow \infty} \mathbb{E}\left[e^{-cX_n}\right] = \mathbb{E} \left[ e^{-cX}\right].$$
In particular, $X_n$ converges in distribution to $X$ as $n \rightarrow \infty$.
\end{lemma}
\begin{proof}
Let $\epsilon > 0$ be given. We observe that
$$ \left| \mathbb{E} \left[ e^{-cX_n}\right] - \mathbb{E} \left[ f_n(cX_n)\right] \right| \leq \mathbb{E} \left[ \left|e^{-cX_n} - f_n(cX_n) \right|\right] \leq \sup_{x \in \mathbb{R}_{\geq 0}}\left|e^{-x} - f_n(x) \right| \rightarrow 0 \mbox{ as } n \rightarrow \infty.$$
In the second inequality we used that $X_n$ are non-negative and the last statement holds by assumption.

It follows that for every $c > 0$ (and clearly also when $c = 0$) 
$$\lim_{n \rightarrow \infty} \mathbb{E}\left[ e^{-cX_n}\right] = \mathbb{E} \left[ e^{-cX}\right].$$
The above statement implies $X_n$ converges to $X$ in distribution by Theorem 4.3 in \cite{Kallen}.
\end{proof}

\begin{proof}(Theorem \ref{TCDRP}) 
Let $r_n$ be a sequence converging to $1^-$ and set $t_n$ so that $(-\log t_n) = \kappa (1-r_n)^{1/3}$. Define
$$f_n(x) =\frac{1}{((- t_n)(1-t_n) \cdot x; t_n)_\infty} \mbox{ and } X_n = e^{\hat\xi_{r_n }}.$$

Lemma \ref{seqfC} shows that $f_n$ satisfy the conditions of Lemma \ref{probC}. In addition, recall that by (\ref{KCDRP}) we have
$$\det (I - K_{CDRP})_{L^2(\mathbb{R}_+)} = \mathbb{E} \left[ e^{-e^x\exp (\mathcal{F}(T,0) + T/24)} \right].$$
where $\mathcal{F}$ is as in Definition \ref{TWDef} and $T = 2\kappa^3\alpha^{-3}$. Consequently, Lemma \ref{probC} and (\ref{WCC3}) show that for $x \in \mathbb{R}$ one has
\begin{equation}\label{WCC4}
\lim_{n \rightarrow \infty} \mathbb{E}_{a,r_n,t_n}\left[ e^{-e^x\exp( \hat\xi_{r_n} )}\right] = \mathbb{E} \left[ e^{-e^x\exp (\mathcal{F}(T,0) + T/24)}. \right]
\end{equation}

In particular, $\exp(\hat\xi_{r} )$ converges weakly to $\exp (\mathcal{F}(T,0) + T/24) =e^{T/24} \mathcal{Z}(T,0) $. In \cite{MN} it was shown that $\mathcal{Z}(T,0)$ is a.s. positive and has a smooth density, thus we conclude that for each $x \in \mathbb{R}_+$ we have
\begin{equation*}
\lim_{r \rightarrow 1^-} \mathbb{P}_{a,r,t} (\exp(\hat \xi_{r}  ) \leq x) = \mathbb{P} ( \exp (\mathcal{F}(T,0) + T/24 )\leq x) .
\end{equation*}
Taking logarithms we see that for each $x \in \mathbb{R}$ we have
\begin{equation}\label{WCC7}
\lim_{r \rightarrow 1^-} \mathbb{P}_{a,r,t} (\hat \xi_{r}  \leq x) = \mathbb{P} ( \mathcal{F}(T,0) + T/24 \leq x) .
\end{equation}

Consider $a(r) =  r^{(1 + |\lfloor \tau N(r) \rfloor|)/2}$. Since, $\lim_{r \rightarrow 1^-} r^{N(r)} = e^{-1},$ we see that $\lim_{r \rightarrow 1^-} a(r) = a(1) = e^{-|\tau|/2} < 1$ (whenever $\tau \neq 0$). This means that $\alpha^{-1} := \left[ \frac{a(1)}{(1 + a(1))^2}\right]^{1/3} =  \left[ \frac{e^{-|\tau|/2}}{(1 + e^{-|\tau|/2})^2}\right]^{1/3} =: \chi$. From Section \ref{HL} we conclude that 
\begin{equation*}
\begin{split}
\mathbb{P}_{HL}^{r,t} \left( \frac{ \lambda'_1( \lfloor \tau N(r) \rfloor)  - M(r) }{ \chi^{-1} N(r)^{1/3} (T/2)^{-1/3}} + \log(N(r)^{1/3}\chi^{-1}(T/2)^{-1/3}) \leq x\right) =\\
 \mathbb{P}_{a,r,t} \left( \frac{ \lambda'_1 - M(r)}{ \alpha^{-1} N(r)^{1/3} (T/2)^{-1/3}} + \log(N(r)^{1/3}\alpha^{-1}(T/2)^{-1/3}) \leq x\right) \\
\end{split}
\end{equation*}
The latter implies that if we set $\kappa = (T/2)^{1/3}\alpha$ we will get
\begin{equation*}\label{WCC5}
\mathbb{P}_{HL}^{r,t} \left( \frac{ \lambda'_1( \lfloor \tau N(r) \rfloor)  - M(r) }{ \chi^{-1} N(r)^{1/3} (T/2)^{-1/3}} + \log(N(r)^{1/3}\chi^{-1}(T/2)^{-1/3}) \leq x\right)  = \mathbb{P}_{a,r,t}( \hat \xi_r + \log((1-t)\kappa^{-1}N(r)^{1/3})  \leq x).
\end{equation*}
One observes that $(1-t)\kappa^{-1}N(r)^{1/3} = \frac{1-t}{-\log t} \rightarrow 1$ as $r \rightarrow 1^-$ and so from (\ref{WCC7}) we conclude that
\begin{equation*}\label{WCC5}
\lim_{r \rightarrow 1^-}\mathbb{P}_{HL}^{r,t} \left( \frac{ \lambda'_1( \lfloor \tau N(r) \rfloor)  - M(r) }{ \chi^{-1} N(r)^{1/3} (T/2)^{-1/3}} + \log(N(r)^{1/3}\chi^{-1}(T/2)^{1/3}) \leq x\right)  = \mathbb{P} ( \mathcal{F}(T,0) + T/24 \leq x) .
\end{equation*}

From (\ref{delM}) we have $c_1= M(r) = 2 N(r)\log (1 + a(1)) + O(1) = 2N(r) \log (1 +e^{-|\tau|/2}) + O(1).$ Substituting this above concludes the proof of the theorem.
\end{proof}

%
\subsection{Proof of Theorem \ref{mainThm2}}\label{S73}\hspace{2mm}\\

We split the proof of Theorem \ref{mainThm2} into three steps. In the first step we rewrite the LHS of (\ref{mainLimitC}) in a suitable form for the application of Lemmas \ref{FDCT} and \ref{FDKCT} and identify the pointwise limit of the integrands. In the second step we provide dominating functions, which are necessary to apply the lemmas. In the third step we obtain a limit for the LHS of (\ref{mainLimitC}), subsequently we use a result from \cite{BCF}, to show that the limit we obtained is in fact $\det( I -  K_{CDRP})_{L^2(\mathbb{R}^+)}$.

In Steps 1 and 2 we will require some estimates, which we summarize in Lemmas \ref{cubicEstC} and \ref{techiesC} below. The proofs are postponed until Section \ref{SSart}.
\begin{lemma}\label{cubicEstC}
Let $t$ be sufficiently close to $1^-$. Then for all large $N$ we have
\begin{equation}\label{cubicZ1C}
Re(S_{a,r} ((-\log t)z) - M(r)(-\log t)z) \leq C -c|z|^2 \mbox{ for all }  z \in \gamma_+^t\mbox{ and }
\end{equation}
\begin{equation}\label{cubicW1C}
Re(S_{a,r} ((-\log t)z) - M(r)(-\log t)z) \geq  c|z|^2 - C \mbox{ for all } z \in \gamma_-^t. \\
\end{equation}
In the above $C, c > 0$ depends on $\delta$. In addition, we have
\begin{equation}\label{cubicW2C}
\lim_{N \rightarrow \infty}S_{a,r}((-\log t)u) - M(r)(-\log t)u = u^3\kappa^3\alpha^{-3}/3  \mbox{ for all } u \in \mathbb{C}. \\
\end{equation}
\end{lemma}

\begin{lemma}\label{techiesC}
Let $t \in (1/2,1)$. Then we can find a universal constant $C$ such that
\begin{equation}\label{yellow1C}
\left| \frac{1}{e^z - e^w}\right| \leq C \mbox{ and } \sum_{k \in \mathbb Z} \left| \frac{1}{\sin(\pi (w - \frac{2\pi k\iota}{-\log t} - z))}\right| \leq C \mbox{ when $Re(z) = 1/4$ and $Re(w) = -1/4$}.
\end{equation}
\end{lemma}

{\raggedleft {\bf Step 1.}} Observe that the LHS of (\ref{mainLimitC}) can be rewritten as $\det (I -  \hat K_x^N)_{L^2(\gamma_-)}$, where
\begin{equation}\label{CDRPeq1}
\begin{split} &\hat K_x^N(w,w') =   \int_{\gamma_+} g^{N,x}_{w,w'}(z)\frac{dz}{2\pi \iota}, \mbox{ and }
 g^{N,x}_{w,w'}(z) ={\bf 1}_{\{\max(|Im(w)|, |Im(w')|, |Im(z)|) \leq (-\log t)^{-1}\pi\}} \times 
\\ &t^{-w}G_{\zeta_x}(w,z) \frac{(-\log t)}{ t^{-w'}- t^{-z}}\frac{\exp\left(S_{a,r}((-\log t)z) + M(\log t)z+ xz\right)}{ \exp\left( S_{a,r}((-\log t)w) + M(\log t)w + xw\right)} .
\end{split}
\end{equation}

Let us now fix $w,w' \in \gamma'_W(A)$ and $z \in \gamma'_Z(A)$ and show that 
\begin{equation}\label{blue0C}
\lim_{N \rightarrow \infty} g^{N,x}_{w,w'}(z)  =  g^{\infty,x}_{w,w'}(z) \mbox{, where } g^{\infty,x}_{w,w'}(z): = \frac{\pi }{\sin (\pi(z-w))} \frac{1}{z-w'}\frac{\exp(\alpha^{-3} \kappa^3 z^3/3 +xz ) }{\exp(\alpha^{-3} \kappa^3 w^3/3  +xw) }.
\end{equation}
One readily observes that
\begin{equation}\label{blue1C}
\lim_{N \rightarrow \infty } t^{-w} {\bf 1}_{\{\max(|Im(w)|, |Im(w')|, |Im(z)|) \leq(-\log t)^{-1}\pi\}}\frac{(-\log t)}{ t^{-w'}- t^{-z}}= \frac{1}{w' - z}
\end{equation}
Using (\ref{cubicW2C}) we get
\begin{equation}\label{blue2C}
\lim_{N \rightarrow \infty}\frac{\exp\left(S_{a,r}((-\log t)z) + M(\log t)z+ xz\right)}{ \exp\left( S_{a,r}((-\log t)w) + M(\log t)w + xw\right)} =\frac{\exp(\alpha^{-3} \kappa^3 z^3/3 +xz ) }{\exp(\alpha^{-3} \kappa^3 w^3/3  +xw) }.
\end{equation}
From the definition of $G_{\zeta_x}$ we have
\begin{equation}\label{blue3C}
G_{\zeta_x}( w, z) =\sum_{ k \in \mathbb{Z}}\frac{\pi  (-\zeta_x)^{2\pi k \iota /(-\log t)}}{\sin (\pi (w-z) + 2\pi k\iota/\log t ))}.
\end{equation}

Using a similar argument as in (\ref{gzest}) we see that for $|k| \geq 1$ and all large $N$ one has
$$\left| \frac{\pi  (-\zeta_x)^{2\pi k \iota /(-\log t)}}{\sin (\pi (w-z) + 2\pi k\iota/\log t )} \right| \leq C e^{-2|k|\pi/(-\log t)}.$$
The latter is summable over $|k| \geq 1$ and since $1/(-\log t)$ goes to infinity the sum goes to $0$. We see that the only non-trivial contribution in (\ref{blue3C}) comes from $k = 0$ and so
\begin{equation}\label{blue4C}
\lim_{N \rightarrow \infty}G_{\zeta_x}( w, z) = \lim_{N \rightarrow \infty}\frac{\pi }{\sin (\pi (w-z) )} =\frac{\pi }{\sin (\pi (w-z) )}.
\end{equation}
Equations  (\ref{blue1C}), (\ref{blue2C}) and (\ref{blue4C}) imply (\ref{blue0C}).\\

{\raggedleft {\bf Step 2.}} We now proceed to find estimates of the type necessary in Lemma \ref{FDKCT} for the functions $g^{N,x}_{w,w'}(z)$. If $z\in \gamma_+$ and $|Im(z)| \leq \pi (-\log t)^{-1}$ the estimates of (\ref{cubicZ1C}) are applicable and so we obtain
\begin{equation}\label{cutoff1C}
|\exp(S_{a,r}((-\log t)z) + M(\log t)z + xz)| \leq C\exp(-c|z|^2+ |xz|),
\end{equation}
where $C,c$ are positive constants.

If $w \in \gamma_-$ and $|Im(w)| \leq \pi (-\log t)^{-1}$ the estimates of (\ref{cubicW1C}) are applicable and we obtain
\begin{equation}\label{cutoff2C}
|\exp(-S_{a,r}((-\log t)w) - M(\log t)w - xw)| \leq C\exp(-c|w|^2+ |xw|),
\end{equation}
for some $C,c >0$. 

From Lemma \ref{techiesC} we have for some $C > 0$ that
\begin{equation}\label{cutoff3C}
\left| G_{\zeta_x}(w,z) \frac{(-\log t)}{ t^{-w'}- t^{-z}}\right| \leq C.
\end{equation}

Observe that $t^{-w} = O(1)$ when $|Im(w)| \leq \pi (-\log t)^{-1}$ and $w \in \gamma_-$. Combining the latter with (\ref{cutoff1C}), (\ref{cutoff2C}) and (\ref{cutoff3C}) we see that whenever $\max(|Im(w)|, |Im(w')|, |Im(z)|) \leq (-\log t)^{-1}\pi$, $z \in \gamma_+$ and $w', w\in \gamma_-$  we have
\begin{equation}\label{cutoff4C}
|g^{N,x}_{w,w'}(z)| \leq C\exp(-c|w|^2 + |xw|)\exp(-c|z|^2 + |xz|),
\end{equation}
where $C,c$ are positive constants. Since $g^{N,x}_{w,w'}(z) = 0$ when $\max(|Im(w)|, |Im(w')|, |Im(z)|) >(-\log t)^{-1}\pi$ we see that (\ref{cutoff4C}) holds for all $z \in \gamma_+$ and $w', w\in \gamma_+$. \\

{\raggedleft{\bf Step 3.}} We may now apply Lemma \ref{FDKCT} to the functions $g^{N,x}_{w,w'}(z)$ with $F_1(w) = C\exp(-c|w|^2+ |xw|) = F_2(w)$ and $\Gamma_1 = \gamma_-$, $\Gamma_2 = \gamma_+$. Notice that the functions $F_i$ are integrable on $\Gamma_i$ by the square in the exponential. As a consence we see that if we set $\hat K^{\infty}_x(w,w') :=  \int_{\gamma_-} g^{\infty,x}_{w,w'}(z)\frac{dz}{2\pi \iota}$, then $\hat K^N_x$ and $\tilde K^\infty_x$ satisfy the conditions of Lemma \ref{FDCT}, from which we conclude that 
\begin{equation}\label{cutoff5C}
\lim_{r \rightarrow 1^{-}} \det( I - \hat K_{\zeta_x})_{L^2(\gamma_-^t)} = \det( I - \hat K^\infty_x)_{L^2(\gamma_-)}.
\end{equation}
What remains to be seen is that $\det( I - \tilde K^\infty_x)_{L^2(\gamma_-)} = \det( I -  K_{CDRP})_{L^2(\mathbb{R}^+)}$.\\

We have that $\det(I - \tilde K^\infty_{x})_{L^2(\gamma_-)} = 1 + \sum_{n = 1}^\infty \frac{(-1)^n}{n!}H(n) ,$ where 
$$H(n) = \sum_{\rho \in S_n} sign(\rho) \int_{\gamma_-} \hspace{-2mm}\cdots \int_{\gamma_-}\int_{\gamma_+}\hspace{-2mm}\cdots \int_{\gamma_+} \prod_{i = 1}^n\frac{\pi e^{ \alpha^{-3} \kappa^3 Z_i^3/3 - \kappa^3\alpha^{-3}W_i^3/3 + xZ_i - xW_i} }{\sin (\pi(Z_i-W_i))(Z_i-W_{\rho(i)})}\frac{dW_i}{2\pi \iota} \frac{dZ_i}{2\pi \iota},$$
Put $\sigma = \alpha\kappa^{-1}$ and consider the change of variables $z_i = \sigma^{-1} Z_i$, $w_i =  \sigma^{-1}  W_i$. Then we have
$$H(n) = \sum_{\rho \in S_n} sign(\rho) \int_{\frac{-1}{4\sigma} + \iota \mathbb{R}} \cdots \int_{\frac{-1}{4\sigma} + \iota \mathbb{R}}\int_{\frac{1}{4\sigma} + \iota \mathbb{R}} \cdots  \int_{\frac{1}{4\sigma} + \iota \mathbb{R}} \prod_{i = 1}^n\frac{\sigma\pi e^{ z_i^3/3 - w_i^3/3 + \sigma x z_i - \sigma xw_i} }{\sin (\sigma \pi(z_i-w_i))(z_i-w_{\rho(i)})}\frac{dw_i}{2\pi \iota} \frac{dz_i}{2\pi \iota}.$$
Consequently, we see that
$$\det(I - \hat K^\infty_{x})_{L^2(\gamma_-)} = \det (I + \hat K_{CDRP})_{L^2(\frac{-1}{4} + \iota\mathbb{R})},$$
where 
\begin{equation}\label{almostCDRP}
\hat K_{CDRP}(w,w') = \frac{-1}{2\pi \iota} \int_{\frac{1}{4\sigma} + \iota \mathbb{R}} dz\frac{\sigma \pi e^{ \sigma x (z- w)} }{\sin (\sigma \pi(z-w))}\frac{e^{ z^3/3 - w^3/3}}{z-w'}
\end{equation}

The proof of Lemma 8.8 in \cite{BCF} can now be repeated verbatim to show that 
$$\det (I + \hat K_{CDRP})_{L^2(\frac{-1}{4} + \iota\mathbb{R})} = \det (I - K_{CDRP})_{L^2(\mathbb{R}_+)}.$$ 
This suffices for the proof.

\section{The function $S_{a,r}$} \label{SSart}
In this section we isolate some of the more technical results that were implicitly used in the proofs of Theorems \ref{TW} and \ref{TCDRP}. We start by summarizing some of the analytic properties of the function $S_{a,r}$ (see Definition \ref{S5DefFun}). Subsequently, we identify different ascent/descent contours and analyze the real part of the function along them. We finish with several estimates that played a central role in the proofs of Theorems \ref{mainThm} and \ref{mainThm2}.

%
\subsection{Analytic properties} \hspace{2mm}\\

We summarize some of the properties of $S_{a,r}$ in a sequence of lemmas. For the reader's convenience we recall the definition of $S_{a,r}$.
$$S_{a,r}(z) :=\sum_{j = 0}^{\infty}  \log ( 1 + ar^j e^z) - \sum_{j = 0}^{\infty}  \log ( 1 + ar^j e^{-z}),$$
where $a,r \in (0,1)$.

\begin{lemma}\label{sartbasic}
Suppose that $\delta \in (0,1)$. Consider $r \in (0,1)$ and $a \in (0, 1-\delta]$. Then there exists $\Delta'(\delta) > 0$ such that $S_{a,r}(z)$ is well-defined and analytic on $D_\delta = \{z \in\mathbb{C} : |Re(z)| < \Delta'\}$ and satisfies
\begin{equation}\label{S5expoS}
\exp( S_{a,r}(z))  =\prod_{j = 0}^{\infty}\frac{1 + ar^je^z}{1 + ar^je^{-z}} .
\end{equation}
\end{lemma}
\begin{proof}
We let $\Delta' > 0$ be such that $(1-\delta) e^{\Delta'} < 1$. Since $r \in (0,1)$, we have that $|ar^je^{\pm z}| < 1$ for $z \in D_\delta$ and $j \geq 0$. Consequently, $\log (1 + ar^j e^{\pm z})$ is a well-defined analytic function on $D_\delta$ for each $j \geq 0$.\\

Let $K \subset D_\delta$ be compact. Then there exists a constant $C(K) > 0$ such that $\left| e^{\pm z}\right| \leq C$ for all $z \in K$. It follows, that for all large $j$ one has $\left| e^{\pm z}ar^j\right| < 1/2$. Using that $| \log( 1 + w)| \leq 2|w|$ when $|w| < 1/2$ we see that $|\left (1 + ar^j e^{\pm z}) \right| \leq 2 C ar^j$ for all large $j$, which are summable. This implies that the sums $\sum_{j = 0}^{\infty}  \log ( 1 + ar^j e^{\pm z})$ are absolutely convergent on $K$. This in particular shows $S_{a,r}$ is well-defined, but also, since the absolutely convergent sum of analytic functions is analytic, we conclude that $S_{a,r}(z)$ is analytic on $D_\delta$.\\

Next let $z \in D_\delta$. From our work above 
$$S_{a,r}(z) = \lim_{M \rightarrow \infty}\left[ \sum_{j = 0}^{M}  \log ( 1 + ar^j e^{z}) - \sum_{j = 0}^{M}  \log ( 1 + ar^j e^{-z})\right].$$
By continuity of the exponential we see that
$$\exp(S_{a,r}(z)) = \lim_{M \rightarrow \infty} \exp\left[ \sum_{j = 0}^{M}  \log ( 1 + ar^j e^{z}) - \sum_{j = 0}^{M}  \log ( 1 + ar^j e^{-z})\right] = \lim_{M \rightarrow \infty}\prod_{j = 0}^{M}\frac{1 + ar^je^{z}}{1 + ar^je^{-z}} ,$$
which equals the RHS of (\ref{S5expoS}).
\end{proof}

\begin{lemma}\label{sartcoeff}
Assume the notation in Lemma \ref{sartbasic}. Then $S_{a,r}(z)$ is an odd function on $D_\delta$ and the power series expansion of $S_{a,r}(z)$ near zero has the form
\begin{equation}\label{Sanal}
S_{a,r} = c_1z + c_3z^3 + \cdots, \mbox{ where } c_{2l+1} = \frac{2}{(1-r)(2l+1)!}\sum_{k = 1}^{\infty} k^{2l}(-1)^{k+1}a^k\frac{1-r}{1 - r^k} \in \mathbb{R}.
\end{equation}
Moreover, for each $l \geq 1$ one has that
\begin{equation}\label{coeffest}
c_{2l + 1} \leq \frac{1}{(1-r)\delta^{2l + 1}}.
\end{equation}
\end{lemma}
\begin{proof}
The fact that $S_{a,r}$ is odd follows from its definition and Lemma \ref{sartbasic}. Next we consider $G(z) = \sum_{j = 0}^{\infty}  \log ( 1 + ar^j e^{z})$. On $D_\delta$ we have that $\left| ar^j e^z\right| < 1$ so we can use the power-series expansion for $\log (1 + x)$ to get
$$\sum_{j = 0}^{\infty}  \log ( 1 + ar^j e^z) = \sum_{j = 0}^{\infty}  \sum_{k = 1}^\infty \frac{(-1)^{k+1}}{k}(ar^j)^k e^{kz}.$$
Power-expanding the exponential, the above becomes
\begin{equation}\label{realsum}
\sum_{j = 0}^{\infty}  \sum_{k = 1}^\infty  \sum_{m = 0}^\infty \frac{1}{m!}\frac{(-1)^{k+1}}{k}(ar^j)^k k^m z^m.
\end{equation}
We will show that the above sum is absolutely convergent (provided $|z|$ is sufficiently small), which would allow us to freely rearrange the sum.\\

Consider $f(x) = \frac{1}{1-x} = \sum_{ j \geq 0 }x^j$ for $|x| < 1$. We know that for $|x| < 1$ and $m \geq 0$ we have
$$f^{(m)}(x) = \sum_{ j \geq 0 }(j+m)(j+m-1)\cdots( j+1)x^{j}, \mbox{ and } f^{(m)}(x) = \frac{m!}{(1-x)^{m+1}}.$$
Putting $x = a$ we see that 
\begin{equation}\label{trick1}
\sum_{k = 1}^\infty a^k k^{m-1} \leq \sum_{k = 1}^\infty a^k k^{m} \leq \sum_{k \geq 1 } (k+m)\cdots (k+1) a^{k} < \frac{m!}{(1-a)^{m+1}}.
\end{equation}
The latter shows that
$$  \sum_{m = 0}^\infty \sum_{k = 1}^\infty  \sum_{j = 0}^{\infty}  \frac{(ar^j)^k k^m |z|^m}{k m!} \leq \frac{1}{1 - r} \sum_{m = 0}^\infty \sum_{k = 1}^\infty \frac{k^{m-1}|z|^m}{m!}  a^k <  \frac{1}{1 - r} \sum_{m = 0}^\infty \frac{|z|^m}{(1-a)^{m+1}},$$
and the leftmost expression is finite for small enough $|z|$.\\

Rearranging (\ref{realsum}) we see that the coefficient in front of $z^m$ in $G(z)$ is $ \frac{1}{m!} \sum_{k = 1}^\infty \sum_{j = 0}^{\infty}  \frac{(-1)^{k+1}}{k}(ar^j)^k k^m $. 
Since $S_{a,r}(z) = G(z) - G(-z)$ we see that the even coefficients of $S_{a,r}(z)$ are zero, while the odd ones equal
$$c_{2l+1} =   \frac{2}{(2l + 1)!} \sum_{k = 1}^\infty \sum_{j = 0}^{\infty}    \frac{(-1)^{k+1}}{k}(ar^j)^k k^{2l + 1} = \frac{2}{(1-r)(2l + 1)!} \sum_{k = 1}^{\infty} k^{2l}(-1)^{k+1}a^k\frac{1-r}{1 - r^k},$$
as desired.

For the second part of the lemma observe that
$$ \left|\sum_{k = 1}^{\infty} k^{2l}(-1)^{k+1}a^k\frac{1-r}{1 - r^k} \right| \leq   \sum_{k = 1}^{\infty} k^{2l} a^k  < \frac{(2l)!}{(1-a)^{2l+1}},$$
where in the last inequaity we used (\ref{trick1}). If $l \geq 1$ and $a \in (0, 1- \delta]$ we conclude that
$$|c_{2l+1}| \leq \frac{2}{(1-r)(2l + 1)!}\frac{(2l)!}{(1-a)^{2l+1}} \leq \frac{1}{(1-r)\delta ^{2l+1}}.$$
\end{proof}

\begin{lemma}\label{Lascoeff}Let $c_1$ and $c_3$ be as in Lemma \ref{sartcoeff}. Also suppose that $a$, depends on $r$ and $\lim_{r \rightarrow 1^{-}}a(r) = a(1) \in (0, 1-\delta]$. Then
\begin{equation}\label{ascoeff}
\lim_{r \rightarrow 1^-}(1-r)c_1 = 2 \log (1 + a(1)) \mbox{ and } \lim_{r \rightarrow 1^-}(1-r)c_3 = \frac{1}{3} \frac{a(1)}{(1 + a(1))^2}.
\end{equation}
\end{lemma}
\begin{proof}
From Lemma \ref{sartcoeff} we know that $c_1 = \frac{2}{1-r}\sum_{k = 1}^{\infty} (-1)^{k+1}a(r)^k\frac{1-r}{1 - r^k}$. Consequently,
$$\lim_{r \rightarrow 1^-}(1-r)c_1 = 2\lim_{r \rightarrow 1^-} \sum_{k = 1}^{\infty} (-1)^{k+1}a(r)^k\frac{1-r}{1 - r^k} = 2 \sum_{k = 1}^{\infty} (-1)^{k+1}\frac{a(1)^k}{k} = 2\log (1 + a(1)),$$
where the middle equality follows from the Dominated Convergence Theorem with dominating function $(1 - \delta/2)^k$.\\

Similarly, we have $c_{3} = \frac{1}{3(1-r)}\sum_{k = 1}^{\infty} k^{2}(-1)^{k+1}a(r)^k\frac{1-r}{1 - r^k}$. Consequently,
$$\lim_{r \rightarrow 1^-}(1-r)c_3 = \frac{1}{3}\lim_{r \rightarrow 1^-}\sum_{k = 1}^{\infty} k^{2}(-1)^{k+1}a(r)^k\frac{1-r}{1 - r^k} =  \frac{1}{3}\sum_{k = 1}^{\infty}k (-1)^{k+1}a(1)^k =  \frac{1}{3}\frac{a(1)}{(1 + a(1))^2},$$
where the middle equality follows from the Dominated Convergence Theorem with dominating function $k^2(1-\delta/2)^k$.\\
\end{proof}

\begin{lemma}\label{LdelM}
Let $c_1$ and $c_3$ be as in Lemma \ref{sartcoeff}. Let $\tau \in \mathbb{R}\backslash \{0\}$ and suppose \\$a(r) = \exp\left(\log r \left(1/2 + \frac{1}{2}\left|\lfloor\frac{\tau }{1-r} \rfloor \right|\right)\right)$, then
\begin{equation}\label{delM1}
\lim_{r \rightarrow 1^-}(1-r)c_1  = 2 \log (1 +e^{-|\tau|/2}) \mbox{ and }   \lim_{r \rightarrow 1^-}(1-r)c_3 = \frac{1}{3} \frac{e^{-|\tau|/2}}{(1 + e^{-|\tau|/2})^2}.
\end{equation}
Moreover, one has
\begin{equation}\label{delM}
c_1- \frac{ 2\log (1 +e^{-|\tau|/2})}{1-r} = O(1), \mbox{ where  the constant depends on $\tau$}.
\end{equation}
\end{lemma}
\begin{proof}
Using that $r^{\frac{1}{1-r}} \rightarrow e^{-1}$ as $r \rightarrow 1^-$ we see that $a(1) = \lim_{r \rightarrow 1^-} a(r) = e^{-|\tau|/2}$. (\ref{delM1}) now follows from Lemma \ref{Lascoeff}.\\

 We can rewrite
$$c_1 - \frac{2\log (1 + a(1))}{1-r} =  I_1 + I_2, \mbox{ where } I_1 = \frac{2}{1-r}\sum_{k = 1}^{\infty}b_k \mbox{ and } I_2= \frac{2}{1-r}\sum_{k = 1}^{\infty} c_k,$$
$\mbox{with } b_k:=   (-1)^{k+1}\left[ a(r)^k\frac{1-r}{1 - r^k} - a(r)^k\frac{1}{k}\right]$ and $c_k : = (-1)^{k+1}\left[ a(r)^k\frac{1}{k} -  a(1)^k\frac{1}{k} \right].$ We will show that $I_1 = O(1) = I_2$. \\

We begin with $I_1$. One observes that 
$$\frac{1-r}{1 - r^k} - \frac{1}{k} = \frac{1}{1 + \cdots + r^{k-1}} - \frac{1}{k}  = \frac{k - 1 - r - \cdots - r^{k-1}}{k (1 + r + \cdots + r^{k-1})} = (1-r) \frac{r^{k-2} + 2r^{k-3} + \cdots + (k-1)r^0}{k (1 + r + \cdots + r^{k-1})}.$$
Consequently,
$$|b_k| \leq (1-r)a(r)^k \frac{1 + 2 + \cdots + (k-1)}{k} \leq \frac{k}{2}(1-r)a(r)^k.$$
It follows that 
$$\left| I_1 \right| \leq \frac{1}{1-r} \sum_{k = 1}^\infty (1-r)ka(r)^k \leq \frac{2}{(1 - a(r))^3} \leq \frac{2}{(1 - e^{-|\tau|/4})^3} = O(1),$$
where in the second inequality we used (\ref{trick1}) and the last inequality holds for all $r$ close to $1^-$.\\

Next we turn to $I_2 = \frac{2}{1-r}\left[ \log(1 + a(r)) -  \log (1 + a(1))\right]$. Since $\log (1 + x)$ is $C^1$ on $\mathbb{R}^+$, we see that $|I_2| \leq \frac{2C}{1-r}|a(r) - a(1)|$ for some constant $C$, independent of $r$ (provided it is sufficiently close to $1^-$, so that $|a(1) - a(r)| \leq 1/2$). Hence it suffices to show that $a(1) - a(r) = O(1-r)$. We know that 
$$  a(1) - a(r) = e^{-|\tau|/2} - \exp\left(\log r \left(1/2 + \frac{1}{2}\left|\lfloor\frac{\tau }{1-r} \rfloor \right|\right)\right)\in [A(r), B(r)],$$
where $A(r) = e^{-|\tau|/2}  - \exp\left(\log r/2 + \frac{\log r|\tau|}{2(1-r)}\right)$ and $B(r) = e^{-|\tau|/2}  - \exp\left(\log r + \frac{\log r|\tau|}{2(1-r)}\right).$
Thus it suffices to show that $A(r) = O(1-r) = B(r)$. We know that $r^{1/2}e^{-|\tau|/2} - e^{-|\tau|/2} = O(1-r) = r^{1}e^{-|\tau|/2} - e^{-|\tau|/2} $, thus it remains to show that $ e^{-|\tau|/2} - \exp\left( - \frac{- \log r|\tau|}{2(1-r)}\right)  = O(1-r)$. Using that $e^{-|\tau|u/2}$ is $C^1$ in $u$, we see that 
$$\left| e^{-|\tau|/2} - \exp\left( - \frac{- \log r|\tau|}{2(1-r)}\right) \right| \leq C\left| 1 -  \frac{- \log r}{1-r}\right|,$$
and the latter is clearly $O(1-r)$ by power expanding the logarithm near $1$.
\end{proof}

\begin{lemma}\label{dersart}
Assume the notation in Lemma \ref{sartbasic}. On $D_\delta$ one has
\begin{equation}
S'_{a,r}(z) = \sum_{j = 0}^{\infty} \frac{ar^j e^z}{1 + ar^{j}e^z} + \sum_{j = 0}^{\infty}  \frac{ar^j e^{-z}}{1 + ar^{j}e^{-z}} = \sum_{j = 0}^{\infty}ar^j \left[\frac{ e^z}{1 + ar^{j}e^z} + \frac{e^{-z}}{1 + ar^{j}e^{-z}} \right].
\end{equation}
\end{lemma}
\begin{proof}
In the proof of Lemma \ref{sartbasic} we showed that on $D_\delta$ 
$$S_{a,r}(z) =\sum_{j = 0}^{\infty}  \log ( 1 + ar^j e^z) - \sum_{j = 0}^{\infty}  \log ( 1 + ar^j e^{-z}),$$
the latter sum being absolutely convergent over compact subsets of $D_\delta$. From Theorem 5.2 in Chapter 2 of \cite{Stein} it follows that 
$$S'_{a,r}(z)  = \sum_{j = 0}^{\infty}  \frac{d}{dz} \log ( 1 + ar^j e^z)  - \sum_{j = 0}^{\infty}  \frac{d}{dz} \log ( 1 + ar^j e^{-z}) = \sum_{j = 0}^{\infty} \frac{ar^j e^z}{1 + ar^{j}e^z} + \sum_{j = 0}^{\infty}  \frac{ar^j e^{-z}}{1 + ar^{j}e^{-z}}.$$
\end{proof}

%
\subsection{Descent contours} \hspace{2mm}\\

In the following lemmas we demonstrate contours, along which the real part of $S_{a,r}(z)- zS'_{a,r}(0)$ varies monotonically. This monotonicity plays an important role in obtaining the estimates of Lemmas \ref{cubicEst} and \ref{cubicEstC}.

\begin{lemma}\label{descent}
Assume the notation in Lemma \ref{sartbasic}. Set $\epsilon = \pm 1$ and $c_1 = S'_{a,r}(0)$. Then there exists an $A_0 > 0$ such that if $0 < A \leq A_0$, one has
$$\frac{d}{dy}Re\left(S_{a,r}(Ay + \epsilon \iota y)-  c_1(Ay+\epsilon \iota y)\right) \leq 0 \mbox{ for all } y\in \left[ 0,\pi \right].$$
$$\frac{d}{dy}Re\left(S_{a,r}(-Ay + \epsilon \iota y)-  c_1(-Ay+\epsilon \iota y)\right) \geq 0 \mbox{ for all } y\in \left[0,\pi \right].$$
\end{lemma}
\begin{proof}
Choose $A_0 > 0$ sufficiently small so that $\left\{ \pm Ay + \iota y : y \in \left[-\pi,\pi \right]\right\} \subset D_\delta$, whenever $0 < A \leq A_0$.

Set $b_j = ar^j$. We will focus on the first statement. We have (using Lemma \ref{dersart}) that
$$\frac{d}{dy}Re\left(S_{a,r}(Ay + \epsilon \iota y)- c_1(Ay+ \iota y)\right) = \sum_{j = 0}^{\infty} Re\left[ b_j \left[\frac{e^{Ay + \epsilon \iota y}}{1 + b_je^{Ay + \epsilon \iota y}} + \frac{e^{-(Ay + \epsilon \iota y)}}{1 + b_je^{-(Ay + \epsilon \iota y)}} - \frac{2}{1 + b_j}\right](A + \epsilon \iota)\right].$$
We will show that each summand is $\leq 0$, provided $A$ is small enough. The latter would follow provided we know that for every $b \in (0, 1-\delta]$ one has
$$Re\left[ \left[\frac{e^{Ay + \epsilon \iota y}}{1 + be^{Ay + \epsilon \iota y}} + \frac{e^{-(Ay + \epsilon \iota y)}}{1 + be^{-(Ay + \epsilon \iota y)}} - \frac{2}{1 + b}\right](A + \epsilon \iota )\right] \leq 0.$$
Multiplying denominators by their complex conjugates and extracting the real part, we see that the above is equivalent to $I_1 + I_2 \leq 0$, where

$$I_1 := A\left[\frac{be^{2Ay} + e^{Ay}\cos(y)}{|1 + be^{Ay + \epsilon\iota y}|^2} +  \frac{be^{-2Ay} + e^{-Ay}\cos(y)}{|1 + be^{-Ay - \epsilon\iota y}|^2} - \frac{2}{1+b}\right] \mbox{ and }$$
$$I_2 := \frac{- e^{Ay}\epsilon\sin(\epsilon y)}{|1 + be^{Ay + \epsilon\iota y}|^2} + \frac{e^{-Ay} \epsilon\sin(\epsilon y)}{|1 + be^{-Ay - \epsilon\iota y}|^2}.$$
We show that $I_1 \leq 0$ and $I_2 \leq 0$, provided $A$ is small enough. \\

We start with $I_2$, which can be rewritten as
$$I_2 = \frac{- e^{Ay}\sin(y)}{1 + b^2e^{2Ay} +2 \cos(y)be^{Ay}} + \frac{e^{-Ay} \sin( y)}{1 + b^2e^{-2Ay} +2 \cos(y)be^{-Ay}}.$$
Since $y \in [0, \pi]$, we have that $\sin(y) \geq 0$. Hence it suffices to show that 
$$0 \geq \frac{- e^{Ay}}{1 + b^2e^{2Ay} +2 \cos(y)be^{Ay}} + \frac{e^{-Ay} }{1 + b^2e^{-2Ay} +2 \cos(y)be^{-Ay}} \iff$$
$$u^{-1} + b^2u + 2b\cos(y) \geq u + b^2u^{-1} + 2b\cos(y)$$
where $u = e^{-Ay} \in (0,1]$. The above now is equivalent to $(u^{-1} - u)(1 -b^2) \geq 0,$
which clearly holds if $u \in(0,1]$ and $b \in (0,1]$, as is the case. Hence $I_2 \leq 0$ without any restrictions on $A$ except that it is positive.\\

Next we analyze $I_1$, which can be rewritten as 
$$I_1 = A\left[\frac{be^{2Ay} + e^{Ay}\cos(y)}{1 + b^2e^{2Ay} + 2b\cos(y)e^{Ay}} +  \frac{be^{-2Ay} + e^{-Ay}\cos(y)}{1 + b^2e^{-2Ay} + 2b\cos(y)e^{-Ay}} - \frac{2}{1+b}\right].$$
We see that (since $A > 0$) $I_1 \leq 0 \iff$
$$(1 + b^2e^{-2Ay} + 2b\cos(y)e^{-Ay})(be^{2Ay} + e^{Ay}\cos(y))(1+b) + (1 + b^2e^{2Ay} + 2b\cos(y)e^{Ay})(be^{-2Ay} + e^{-Ay}\cos(y))(1+b) - $$
$$ - 2(1 + b^2e^{-2Ay} + 2b\cos(y)e^{-Ay})(1 + b^2e^{2Ay} + 2b\cos(y)e^{Ay}) \leq 0 \iff$$
$$\left(1 + b^2e^{-2Ay} + 2b\cos(y)e^{-Ay}\right)\left(be^{2Ay} + e^{Ay}\cos(y) - 1 - be^{Ay}\cos(y)\right)+ $$
$$ + \left(1 + b^2e^{2Ay} + 2b\cos(y)e^{Ay}\right)\left(be^{-2Ay} + e^{-Ay}\cos(y) - 1 - be^{-Ay}\cos(y)\right) \leq 0 \iff$$
$$f(y) = u(y)^2(b - b^2) + u(y)\cos(y)(1-b)^3 + [-2b - 2 + 2b^3 + 2b^2 + 4b\cos(y)^2 - 4b^2\cos(y)^2] \leq 0,$$
where $u(y) = e^{Ay} + e^{-Ay}$. We want to show that $f(y) \leq 0$ on $[0,\pi]$, provided $A$ is small enough.\\

First consider $y \in [0, \pi/2]$. We have
$$f'(y) = 2uu'(b-b^2) + u'\cos(y)(1-b)^3 - u\sin(y)(1-b)^3 + [-8b\cos(y)\sin(y) + 8b^2\cos(y)\sin(y)].$$
The last summand equals $8b\sin(y)\cos(y)(b - 1)$ and is clearly non-positive, when $y \in  [0, \pi/2]$. Thus
$$f'(y) \leq 2uu'(b-b^2) + u'\cos(y)(1-b)^3 - u\sin(y)(1-b)^3.$$
For $A$ sufficiently small we have $u' \leq 4Ay$, $u \leq 3$ and $\sin(y) > y/5$ on $[0,\pi/2]$. Thus we see
$$f'(y) \leq 24A(b-b^2)y + 4(1-b)^3A y - \frac{2}{5}(1-b)^3y.$$
For $A$ sufficiently small $f'(y) < 0$ on $(0,\pi/2)$ so $f$ is decreasing on $(0,\pi/2)$. But $f(0) = 0$ so we see $f(y) \leq 0$ when $y \in [0,\pi/2]$.

 Next we consider the case when $y \in [\pi/2, \pi]$. In that case $\cos(y) \leq 0$ and we see
$$f(y) \leq u(y)^2b(1-b) - 2(1-b)(1+b)^2 + 4b\cos(y)^2 (1-b).$$
The latter expression is non-positive exactly when
$$bu(y)^2 - 2(1+b)^2 + 4b\cos(y)^2 \leq 0.$$
For $A$ sufficiently small we have $u^2 \in [4,4+\epsilon_0)$ for all $y\in [\pi/2, \pi]$. Thus it suffices to show that we can find $\epsilon_0 > 0$ such that 
$$4b + b\epsilon_0 - 2(1+b)^2 +4b \leq 0 \iff b\epsilon_0 \leq 2(1-b)^2,$$
which is clearly possible as $b \in [0, 1- \delta]$. Thus we conclude that there exists $A > 0$ small enough so that the first statement of the lemma holds. Using that $S_{a,r}(z)$ is an odd function, the second statement of the lemma follows from the first and the same $A$ may be chosen.
\end{proof}

\begin{lemma} \label{monotoneC}
Assume the notation in Lemma \ref{sartbasic}. Suppose $t$ is sufficiently close to $1^-$. If $\beta \geq 0$ and $z =(-\log t) (\beta + \iota s)$ then 
$$\frac{d}{ds} Re(S_{a,r}(z)) \leq 0 \mbox{ when } s \in [0, \pi (-\log t)^{-1}] \mbox{ and } \frac{d}{ds} Re(S_{a,r}(z)) \geq 0 \mbox{ when }s \in [- \pi (-\log t)^{-1}, 0].$$

If $\beta \leq 0$ and $z = (-\log t) (\beta + \iota s)$ then
$$\frac{d}{ds} Re(S_{a,r}(z)) \geq 0 \mbox{ when } s \in [0, \pi (-\log t)^{-1}] \mbox{ and } \frac{d}{ds} Re(S_{a,r}(z)) \leq 0 \mbox{ when }s \in  [- \pi (-\log t)^{-1}, 0].$$

\end{lemma}

\begin{proof}
The dependence on $t$ comes from our desire to make $|\beta|(-\log t) < \Delta'$ in the statement of Lemma \ref{sartbasic}. We assume this for the remainder of the proof. \\

Setting $z =(-\log t) (\beta +  \iota s)$ we see from Lemma \ref{dersart}
$$\frac{d}{ds} S_{a,r}(z) =  \sum_{j = 0}^{\infty}\iota b_j (-\log t)\left[ \frac{e^{(- \log t)(\beta + \iota s)}}{1 + b_je^{(-\log t)(\beta + \iota s)}}  -  \frac{e^{(-\log t)(-\beta - \iota s)}}{1 + b_j e^{(- \log t)(-\beta - \iota s)}}\right],$$
where $b_j = ar^j$. Thus we see that 
$$\frac{d}{ds} Re (S_{a,r}(z) )= \sum_{j = 0}^{\infty} \left[ -\frac{b_j(- \log t)\sin(\theta)t^{-\beta}}{1 +2\cos(\theta)b_j t^{-\beta} + b_j^2 t^{-2\beta} }  +  \frac{b_j (-\log t)\sin(\theta)t^{\beta}}{1 +2\cos(\theta)b_j t^{\beta}+ b_j^2 t^{2\beta} }\right],$$
where $\theta = s(-\log t)$. \\

We now check that each summand has the right sign for the ranges of $s$ and $\beta$ in the statement in the lemma. We focus on $\beta \geq 0$ and $s \in [0, \pi (-\log t)^{-1}]$, all other cases can be handled similarly.

We want to show that 
$$ -\frac{b_j (- \log t)\sin(\theta)t^{-\beta}}{1 +2\cos(\theta)b_j t^{-\beta} + b_j^2 t^{-2\beta} }  +  \frac{b_j (-\log t)\sin(\theta)t^{\beta}}{1 +2\cos(\theta) b_j t^{\beta}+b_j^2 t^{2\beta} } \leq 0 \mbox{ for each } j.$$
Put $u = t^{-\beta}$ and $b_j = b$. Observe that for $s\in [0, \pi(-\log t)^{-1}]$, $\theta \in [0, \pi]$ so the above would follow from
$$-\frac{u}{1 +2\cos(\theta)b u + b^2 u^2}   +  \frac{u^{-1}}{1 +2\cos(\theta) bu^{-1}+b^2 u^{-2} }  \leq 0 \iff$$
$$\iff  u^{-1}(1 +2\cos(\theta)b u +b^2 u^2) \leq u(1 +2\cos(\theta) bu^{-1}+b^2 u^{-2})  \iff $$
$$ u^{-1} + 2\cos(\theta) b + b^2 u \leq  u + 2\cos(\theta)b + b^2u^{-1}  \iff  (u^{-1} - u)( 1 - b^2) \leq 0. $$
The latter is true since $u \geq 1$ and $b \in (0,1)$.
\end{proof}

%
\subsection{Proof of Lemmas \ref{cubicEst} and \ref{cubicEstC}}\label{AD} \hspace{2mm} \\

Suppose that $\delta > \epsilon > 0$ is sufficiently small so that $S_{a,r}$ has an analytic expansion in the disc of radius $\epsilon$ for $r \in (0,1)$ and $a \in (0,1-\delta]$. 
From (\ref{coeffest}) we know that when $|z| < \epsilon$ one has
\begin{equation}\label{S6eq1}
|S_{a,r}(z) - c_1z - c_3z^3| \leq \frac{|z|^4}{1 - r}\sum_{l \geq 2} \epsilon^{2l - 3}\delta^{-2l - 1},
\end{equation}
and the latter sum is finite by comparison with the geometric series. Suppose that $z = N^{-1/3} w$ where $N = \frac{1}{1-r}$. Clearly, the RHS of (\ref{S6eq1}) is $O(N^{-1/3})$ and so
$$\lim_{N \rightarrow \infty} |S_{a,r}(N^{-1/3}w) - c_1N^{-1/3}w - c_3N^{-1}w^3| = 0.$$
Using that $\lim_{N\rightarrow \infty}c_3N^{-1} =  \frac{1}{3} \frac{a(1)}{(1 + a(1))^2}$ (this is (\ref{ascoeff})) and the above we conclude that
$$S_{a,r}(N^{-1/3}w) - c_1N^{-1/3}w = O(1) \mbox{ if }w = O(1) \mbox{ and } \lim_{N \rightarrow \infty}S_{a,r}(N^{-1/3}w) - c_1N^{-1/3}w = \frac{1}{3} \frac{a(1)}{(1 + a(1))^2}w^3.$$
This proves (\ref{cubicZ2}), (\ref{cubicW2}) and once we set $(-\log t) = \kappa N^{-1/3}$ also (\ref{cubicW2C}).\\

Suppose $A$ sufficiently small so that the statement of Proposition \ref{descent} holds and so that $\phi = \arctan(A)$ is less than $10^{\circ}$. By choosing a smaller $\epsilon$ than the one we had before we may assume that $\sum_{l \geq 2} \epsilon^{2l - 3}\delta^{2l + 1} \leq \frac{a(1)\sin(3\phi)}{12(1 + a(1))^2} = c'$. In view of (\ref{S6eq1}) and (\ref{ascoeff}) we know that for all large $N$ and $|z| < \epsilon$
$$Re\left(S_{a,r}(z) - c_1z\right) \geq c_3Re(z^3) - c'N|z|^4 \geq N|z|^3 \frac{a(1)\sin(3\phi)}{6(1 + a(1))^2} - c'N|z|^3 \geq c'N|z|^3 \mbox{ if } z \in \gamma_W.$$
This proves (\ref{cubicW1}) when $|z| < \epsilon$. Put $K = \frac{\epsilon}{2\pi}$ and observe that if $z \in \gamma_W$ then $Kz \in \gamma_W$ and $K|z| < \epsilon$. The latter suggests that if $z \in \gamma_W$ we have
$$Re\left(S_{a,r}(z) -c_1z\right) \geq Re(S_{a,r}(Kz) - M(r)Kz) \geq c'NK^3|z|^3,$$
where in the first inequality we used the first statement of Lemma \ref{descent}, and in the second one we used that  $K|z| < \epsilon$ and our earlier estimate. This proves (\ref{cubicW1}) and using that $S_{a,r}(-z) =  - S_{a,r}(z)$, while $\gamma_W = - \gamma_Z$ it also proves (\ref{cubicZ1}). \\

Let $z = 1/4 + \iota s$ and set $(-\log t) = \kappa N^{-1/3}$ for some positive $\kappa$. Suppose $|(-\log t)z| < \epsilon$ with $\epsilon$ as in the beginning of the section. We have the following equality
$$Re(c_{2l + 1}(-\log t)^{2l + 1}z^{2l + 1}) = c_{2l + 1}(-\log t)^{2l + 1} \sum_{k  = 0}^l \binom{2l + 1}{2k}s^{2k}(-1)^k \frac{1}{4^{2l -2k + 1}}.$$
In particular, we see that 
\begin{equation}\label{S6eq2}
\begin{split}
&\left| Re(c_{2l + 1}(-\log t)^{2l + 1}z^{2l + 1}) \right| \leq c_{2l + 1}(-\log t)^{2l + 1} \left( ( |s| + 1/4)^{2l + 1} - |s|^{2l + 1}\right)  \leq \\
 & c_{2l + 1}(-\log t)^{2l + 1}\frac{1}{4} \sum_{k = 0}^{2l}|s|^k(1/4)^{2l - k} \leq (2l + 1) c_{2l + 1}(-\log t)^{2l + 1}|z|^{2l}.
\end{split}
\end{equation}
Using (\ref{S6eq2}) and (\ref{coeffest}) we have for  $|(-\log t)z| < \epsilon$ that
\begin{equation}\label{S6eq3}
\left| Re\left( \sum_{l \geq 2}  c_{2l + 1}(-\log t)^{2l + 1}z^{2l + 1} \right) \right| \leq  \kappa^3  |z|^2\sum_{l \geq 2} (2l + 1)\delta^{-2l - 1} \epsilon^{2l -2}.
\end{equation}

On the other hand, we have that
\begin{equation}\label{S6eq4}
Re(c_3(-\log t)^3z^3) = -(3c_3/4)(-\log t)^3|z|^2 + (-\log t)^3/64 + (3c_3/64)(-\log t)^3 .
\end{equation}
Combining equations (\ref{S6eq3}) and (\ref{S6eq4}) we see that if $|(-\log t)z| < \epsilon$ then 
$$Re(S_{a,r}((-\log t)z) - c_1 (-\log t) z) \leq -(3c_3/4)(-\log t)^3|z|^2 + (-\log t)^3/64 + (3c_3/64)(-\log t)^3  +$$
$$ +   \kappa^3 |z|^2\sum_{l \geq 2} (2l + 1)\delta^{-2l - 1} \epsilon^{2l -2}.$$
Notice that $(3c_3/4)(-\log t)^3 \rightarrow \kappa^3 \frac{a(1)}{4(1 + a(1))^2} = : \rho$ as $N \rightarrow \infty$ from (\ref{ascoeff}). Moreover if we pick $\epsilon$ small enough we can make $ \kappa^3\sum_{l \geq 2} (2l + 1)\delta^{-2l - 1} \epsilon^{2l -2} \leq (\rho /4)$. It follows that for all large $N$ we have
$$Re(S_{a,r}((-\log t)z) - c_1 (-\log t) z) \leq -(\rho/2)|z|^2 +(\rho/8).$$
This proves (\ref{cubicZ1C}) whenever  $|(-\log t)z| < \epsilon$.

Suppose now that $z = 1/4 + \iota s$ and $s \in [-\pi(-\log t)^{-1}, \pi(-\log t)^{-1}]$. Put $K = \frac{\epsilon}{ 2\pi}$ and notice that for all large $N$ we have $\tilde z : = 1/4 + \iota Ks$ satisfies $|\tilde z (-\log t)| < \epsilon$. It follows from the first result of Lemma \ref{monotoneC} and our estimate above that
$$Re(S_{a,r}((-\log t)z) - c_1 (-\log t) z) \leq Re(S_{a,r}((-\log t)\tilde z) - c_1 (-\log t) \tilde z) \leq -(\rho/2)|\tilde z|^2 +(\rho/8).$$
Observing that $|\tilde z|^2 \geq K^{-2}|z|^2$ we conclude (\ref{cubicZ1C}) for all $z \in \gamma^t_+$. The result of (\ref{cubicW1C}) now follows from  (\ref{cubicZ1C})  once we use that $S_{a,r}(-z) = -S_{a,r}(z)$ and that $\gamma_-^t = - \gamma_+^t$.

%
\subsection{Proof of Lemmas \ref{techies} and \ref{techiesC}} \hspace{2mm} \\

Let $z= x+ \iota p$ and $w = y + \iota q$ so that $x > 0$ and $y\leq 0$ . Then we have
$$\left| \frac{1}{e^z - e^w}\right| = \left| \frac{1}{e^x - e^y e^{\iota (q - p)}}\right| \leq \left| \frac{1}{e^x - e^y }\right| \leq \frac{1}{e^x - 1} \leq x^{-1},$$
where in the last inequality we used $e^c \geq c + 1$ for $c \geq 0$. This proves the first parts of (\ref{yellow1}) and  (\ref{yellow1C}).\\

Let $\sigma = (-\log t)^{-1}$. Then we have
$$ \left| \frac{1}{\sin(-\pi\sigma (x - y + \iota (p - q))}\right| = \left| \frac{2}{e^{-\iota \pi\sigma(x-y)}e^{\pi\sigma(p-q)} -  e^{\iota \pi\sigma(x-y)}e^{\pi\sigma(q-p)}}\right| $$
If $q \geq p$ we see 
$$\left| e^{-\iota \pi\sigma(x-y)}e^{\pi\sigma(p-q)} -  e^{\iota \pi\sigma(x-y)}e^{\pi\sigma(q-p)}\right| = \left| e^{\pi\sigma(p-q)} -  e^{2\iota \pi\sigma(x-y)}e^{\pi\sigma(q-p)}\right| \geq e^{\pi\sigma(q-p)}|\sin(2 \pi\sigma(x-y))|.$$
Conversely, if $q < p$ we see
$$\left| e^{-\iota \pi\sigma(x-y)}e^{\pi\sigma(p-q)} -  e^{\iota \pi\sigma(x-y)}e^{\pi\sigma(q-p)}\right| = \left| e^{-2\iota \pi\sigma(x-y)}e^{\pi\sigma(p-q)} -  e^{\pi\sigma(q-p)}\right| \geq e^{\pi\sigma(p-q)}|\sin(2 \pi\sigma(x-y))|.$$
We thus conclude that
\begin{equation}\label{yellow2}
\left| \frac{1}{\sin(-\pi\sigma (x - y + \iota (p - q))}\right| \leq e^{-\pi\sigma|p-q|}\frac{2}{|\sin(2 \pi\sigma(x-y))|}.
\end{equation}
In the assumption of Lemma \ref{techies} we have $x-y \in [u, 2U]$ and $2U \leq \sigma^{-1}/5$. Thus $2\pi\sigma(x-y) \in [2\pi \sigma u , 2\pi/5]$. This implies that
\begin{equation}\label{yellow3}
\left| \frac{1}{|\sin(2 \pi\sigma(x-y))|}\right| \leq e^{-\pi\sigma|p-q|}\frac{1}{\sigma u},
\end{equation}
where we used that $\sin x$ is increasing on $[0,\pi/2]$ and satisfies $\pi \sin x \geq x$ there.
In addition, we have from the above
$$
\sum_{k \in \mathbb Z} \left|\frac{1}{\sin(-\pi\sigma (x - y + \iota (p + 2\pi k - q))}\right| \leq \sum_{k \in \mathbb Z}e^{-\pi\sigma|p+ 2\pi k - q|}\sigma^{-1}u^{-1} \leq 2\sigma^{-1}u^{-1} \sum_{k \geq 0}e^{-2k\pi^2 \sigma}.$$
This proves the second part of (\ref{yellow1}).\\

Finally, suppose that $x = 1/4$ and $y = -1/4$. Notice that if dist$(s, \mathbb{Z}) > c$ for some constant $c > 0$ then $\left|\frac{1}{\sin(\pi s)}\right| \leq c'e^{-\pi|Im(s)|}$ for some $c'$, depending on $c$. Using this we get 
$$\sum_{k \in \mathbb Z} \left| \frac{1}{\sin(\pi (w - \frac{2\pi k\iota}{-\log t} - z))}\right|=\sum_{k \in \mathbb Z} \left| \frac{1}{\sin(\pi/2 - \frac{2\pi^2 k\iota }{-\log t} +  \pi \iota(q - p)))}\right| \leq $$
$$\leq  c'  \sum_{k \in \mathbb Z} \exp\left(- \left|  - \frac{2\pi^2 k }{-\log t} +  \pi (q - p)\right| \right) \leq 2c' \sum_{k \geq 0}\exp\left(- \frac{2\pi^2 k }{-\log t}\right). $$
The latter is uniformly bounded for $t \in (1/2,1)$, by $\frac{2c'}{1 - v}$ with $v = \exp\left( -\frac{2\pi^2 }{-\log (1/2)}\right)$. This concludes the proof of the second part of (\ref{yellow1C}). 

\section{Sampling of plane partitions}\label{Section9}

In this section, we describe a sampler of random plane partitions, based on {\em Glauber dynamics} and obtain some empirical evidence supporting the results of this paper. Subsequently, we formulate several conjectures about the convergence of the measure $\mathbb{P}^{r,t}_{HL}$ and provide some evidence about their validity.

%
\subsection{Glauber dynamics}\label{GD}\hspace{2mm}\\

We start with a brief recollection of the (single-site) {\em Glauber dynamics} for probability measures on labelled graphs. In what follows, we will use Section 3.3 in \cite{Peres} as a main reference and recommend the latter for more details.

Let $V$ and $S$ be finite sets and suppose that $\Omega$ is a subset of $S^{V}$. The elements of $S^V$, called {\em configurations}, are the functions from $V$ to $S$. One visualizes a configuration as a labeling of the vertex set $V$ of some graph by elements in $S$. Let $\pi$ be a probability distribution, whose support is $\Omega$. The (single-site) Glauber dynamics for $\pi$ is a reversible Markov chain with state space $\Omega$, stationary distribution $\pi$ and transition probabilities as described below.

For $x \in \Omega$ and $v \in V$, let
$$\Omega(x,v ) := \{ y\in \Omega : y(w) = x(w) \mbox{ for all } w\neq v\} \mbox{ and } \pi^{x,v}(y) :=  \begin{cases} \frac{\pi(y)}{\pi (\Omega(x,v))} &\mbox{ if } y \in \Omega(x,v), \\ 0 &\mbox{ if } y \not\in \Omega(x,v). \end{cases}$$
With the above notation, the Glauber chain moves from state $x$ as follows: a vertex $v$ is chosen uniformly at random from $V$, then one chooses a new configuration according to $\pi^{x,v}$. 

One can show that $\pi$ is a stationary measure for the Glauber dynamics and that the chain is {\em ergodic}. This implies that if the chain is run for $T$ steps, started from any initial state, then the distribution of the state at step $T$ will converge to the stationary distribution $\pi$ as $T \rightarrow \infty$. The latter observation explains how one can use the Glauber dynamics to numerically (approximately) sample arbitrary distributions $\pi$ on $\Omega.$ Namely, one constructs the Glauber dynamics and runs it for a very long time $T$, so that the distribution is close to the stationary distribution of the chain. This sampling method is called a {\em Gibbs sampler} and it belongs to a more general class of methods called {\em Markov chain Monte Carlo}. The time one has to wait for the chain to converge, is typically referred to as a {\em mixing time}; and finding estimates for mixing times is in general very hard.\\

In our case, we consider the measure $\mathbb{P}_{r,t}$ (here $r \in (0,1)$ and $t \in [-1,1]$ ) on plane partitions, which are contained in a big box $N \times N \times N$, satisfying 
\begin{equation}\label{sampledist}
\mathbb{P}_{r,t} (\pi) \propto r^{|\pi|} B_{\pi} (t),
\end{equation}
where $|\pi|$ is the volume of the partition and $B_\pi(t)$ is as in Section \ref{HL}. Specifically, $\mathbb{P}_{r,t}$ is the same as the distribution $\mathbb{P}_{HL}^{r,t,N}$ of Section \ref{HL}, conditioned on plane partitions not exceeding height $N$. We now describe a Gibbs sampler for the above measure.

Set $V = \left\{ (x,y,z) : x,y,z \in \{ 1,...,N\} \right\}$ and $S = \{0,1 \}$. A configuration $\omega \in S^V$ is interpreted as a placements of unit cubes inside the box $N \times N \times N$, so that $\omega((x,y,z)) = 1$ if an only if there is a cube at position $(x,y,z)$. We next let $\Omega$ be the subset of cube placements, corresponding to plane partitions. This describes the state space of our Glauber dynamics. Since $|S| = 2$, we see that if $\pi_a \in \Omega$ we have $|\Omega(\pi_a,v)| = 1$ or $2$; hence, $\mathbb{P}_{r,t}^{\pi_a,v}$ is either a point mass at $\pi_a$ or a Bernoulli measure, whose support lies on $\pi_a$ and the partition $\pi_b$, which is obtained from $\pi_a$ by changing the value of $\pi_a$ at $v$ from $1$ to $0$ or vice versa.

At this time we introduce some terminology. Given a plane partition $\pi$, we call a cube {\em addable} if the the cube does not belong to $\pi$ and by placing the cube in the box we obtain a valid plane partition. Similarly, we call a cube {\em removable} if it belongs to $\pi$ and removing the cube from the box results in a valid plane partition. Denote by $Add_\pi$ and $Rem_\pi$ the (disjoint) sets of addable and removable cubes respectively. Some of these concepts are illustrated in Figure \ref{S9_11}. We observe that  $|\Omega(\pi,v)| = 2$ precisely when there is an element of $Add_\pi$ or $Rem_\pi$ at position $v$. \\

\begin{figure}[h]
\centering
\scalebox{0.6}{\includegraphics{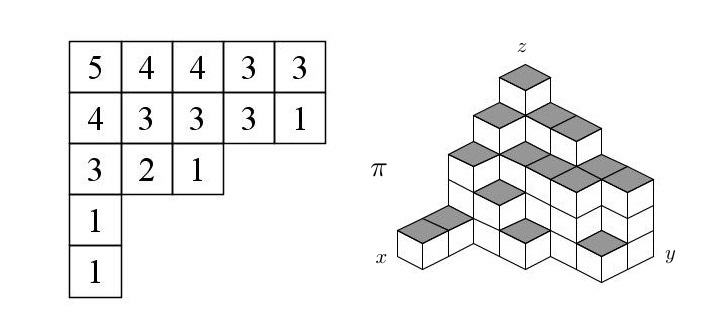}}
\caption{If $N = 5$, then the addable cubes in this example are at positions: $(4,1,2), (3,1,4), (2,1,5), (3,2,3), (2,2,4), (1,2,5), (3,3,2), (1,4,4), (2,5,2)$. The removable cubes are at positions: $(5,1,1), (3,1,3), (2,1,4), (1,1,5), (3,2,2), (3,3,1), \\
(1,3,4), (2,4,3), (2,5,1), (1,5,3)$.}
\label{S9_11}
\end{figure}

We now turn to finding $\mathbb{P}_{r,t}^{\pi,v}$ when $|\Omega(\pi,v)| = 2$. Let $\hat \pi$ be the plane partition obtained from $\pi$ by adding a cube at position $v$ if one is not already present there, otherwise $\hat \pi = \pi$. In addition, let $\check \pi$ be the plane partition obtained from $\pi$ by removing the cube at position $v$ if there is one, otherwise $\check \pi = \pi$. Observe that if $|\Omega(\pi,v)| = 2$, we have either $\hat \pi = \pi$, or $\check \pi = \pi$.

From our earlier discussion, $\mathbb{P}_{r,t}^{\pi,v}$ is a Bernoulli measure supported on $\hat \pi$ and  $\check \pi$. Using results from Section \ref{HL} we have that if $\hat\lambda^k$ and $\check\lambda^k$ denote the diagonal slices of $\hat\pi$ and $\check \pi$ respectively, we have
\begin{equation}
\begin{split}
\mathbb{P}_{r,t}(\hat\pi) \propto  r^{|\pi|} \hspace{-3mm}\prod_{n = -N + 1}^0\hspace{-3mm}\psi_{\hat\lambda^n/\hat\lambda^{n-1}}(0,t) \times \prod_{n = 1}^N \phi_{\hat\lambda^{n-1}/ \hat\lambda^{n}}(0,t), \\
 \mathbb{P}_{r,t}(\check\pi) \propto  r^{|\pi|} \hspace{-3mm}\prod_{n = -N + 1}^0\hspace{-3mm}\psi_{\check\lambda^n/\check\lambda^{n-1}}(0,t) \times \prod_{n = 1}^N \phi_{\check\lambda^{n-1}/ \check\lambda^{n}}(0,t).
\end{split}
\end{equation}
We recall that $\hat\lambda^{-N} = \hat\lambda^N = \varnothing = \check\lambda^{-N} = \check\lambda^N$ and
$$\phi_{\lambda / \mu}(0,t) = \prod_{i \in I}(1 - t^{m_i(\lambda)}) \hspace{3mm} \mbox{ and }  \hspace{3mm} \psi_{\lambda/\mu}(0,t) = \prod_{j \in J}(1 - t^{m_j(\mu)}).$$
In the above formula we assume $\lambda \succ \mu$ otherwise both expressions equal $0$. The sets $I,J$ are:
$$I(\lambda,\mu) = \{ i \in \mathbb{N}: \lambda'_{i+1} = \mu'_{i+1} \mbox{ and }\lambda'_{i} > \mu'_{i}\} \mbox{ and }J(\lambda,\mu) = \{ j \in \mathbb{N}: \lambda'_{j+1} > \mu'_{j+1} \mbox{ and }\lambda'_{j} = \mu'_{j}\}.$$

Set $k = x - y$ and observe that $\check\lambda_i = \hat\lambda_i = \lambda_i$ whenever $i \neq k$. By combining common factors this gives
\begin{enumerate}[label = \arabic{enumi}., leftmargin=1.0cm]
\item $k = 0$: $\mathbb{P}_{r,t}^{\pi,v}(\hat \pi) \propto r\psi_{\hat\lambda^0/\lambda^{-1}}(0,t) \phi_{\hat\lambda^0/\lambda^{1}}(0,t) $ and  $\mathbb{P}_{r,t}^{\pi,v}(\check\pi) \propto \psi_{\check\lambda^0/\lambda^{-1}}(0,t) \phi_{\check\lambda^0/\lambda^{1}}(0,t).$
\item $k > 0$: $\mathbb{P}_{r,t}^{\pi,v}(\hat\pi) \propto r\phi_{\lambda^{k-1}/\hat\lambda^{k}}(0,t) \phi_{\hat\lambda^{k}/\lambda^{k+1}}(0,t) $ and  $\mathbb{P}_{r,t}^{\pi,v}(\check\pi) \propto \phi_{\lambda^{k-1}/\check\lambda^{k}}(0,t) \phi_{\check\lambda^{k}/\lambda^{k+1}}(0,t)$.
\item $k < 0$: $\mathbb{P}_{r,t}^{\pi,v}(\hat\pi) \propto r\psi_{\hat\lambda^k/\lambda^{k-1}}(0,t) \psi_{\lambda^{k+1}/\hat\lambda^{k}}(0,t) $ and  $\mathbb{P}_{r,t}^{\pi,v}(\check\pi) \propto \psi_{\check\lambda^k/\lambda^{k-1}}(0,t) \psi_{\lambda^{k+1}/\check\lambda^{k}}(0,t) $.
\end{enumerate}
In the above $\check\lambda^k$ is obtained from $\hat\lambda^k$, by removing $1$ box from row $\min (x,y)$. The above weights, while explicit, are difficult to calculate efficiently on a computer. Thus we will search for simpler formulas, utilizing that $\check\lambda^k$ is structurally similar to $\hat\lambda^k$. \\

For a partition $\lambda$ we introduce the following notation. Let $S(\lambda)$ be the multiset of positive row-lengths of $\lambda$, counted with multiplicities. One observes that if $\lambda \succ \mu$ one has $I(\lambda, \mu) = S(\lambda) \backslash S(\mu)$ and $J(\lambda, \mu) = S(\mu)\backslash S(\lambda)$ as multisets, in particular $S(\lambda) \backslash S(\mu)$ and $S(\mu)\backslash S(\lambda)$ are honest sets. Let us prove this briefly. 

Since $\lambda \succ \mu$ we have $\lambda_k' = \mu_k'$ or $\mu_{k}' +1$. Consequently, we have $i \in I(\lambda,\mu) \iff \lambda'_i > \mu'_i \mbox{ and } \lambda'_{i + 1} = \mu'_{i + 1} \iff \lambda'_i = \mu'_i + 1 \mbox{ and } \lambda'_{i + 1} = \mu'_{i + 1} \iff  \lambda'_i- \lambda'_{i+1}= \mu'_i - \mu'_{i+1} + 1 \mbox{ and } \lambda'_{i + 1} = \mu'_{i + 1} \iff \lambda'_i- \lambda'_{i+1}= \mu'_i - \mu'_{i+1} + 1  \iff m_i(\lambda) = m_i(\mu) + 1 \iff i \in S(\lambda)/S(\mu)$ and has multiplicity $1$. 

Similarly, $j \in J(\lambda, \mu) \iff \lambda_{j+1}' > \mu_{j+1}' \mbox{ and } \lambda'_{j} = \mu'_{j} \iff \lambda_{j+1}' = \mu_{j+1}' + 1 \mbox{ and } \lambda'_{j} = \mu'_{j} \iff \lambda_{j+1}'  - \lambda_j'= \mu_{j+1}' - \mu_j' + 1 \mbox{ and } \lambda'_{j} = \mu'_{j} \iff \lambda_{j+1}'  - \lambda_j'= \mu_{j+1}' - \mu_j' + 1 \iff m_j(\mu) = m_j(\lambda) + 1 \iff j \in S(\mu) \backslash S(\lambda)$ and has multiplicity $1$. 

The above arguments show that 
$$\phi_{\lambda / \mu}(0,t) = \prod_{i \in S(\lambda) \backslash S(\mu)}(1 - t^{m_i(\lambda)}) \mbox{ and } \psi_{\lambda / \mu}(0,t) = \prod_{i \in S(\mu) \backslash S(\lambda)}(1 - t^{m_i(\mu)}).$$

Suppose that $\lambda, \mu, \nu$ are plane partitions, such that $\lambda \succ \nu$, $\mu \succ \nu$ and $\mu$ is obtained from $\lambda$ by removing a single box from row $k$. In addtition, set $c = \mu_k$. Then we have $S(\lambda) = S(\mu) - \{ c\} + \{c+1\}$ as multisets. Put  $M = \left[ S(\lambda) \backslash S(\nu) \right] \cap  \left[ S(\mu) \backslash S(\nu) \right]$ and observe that $m_i(\lambda) = m_i(\mu)$, whenever $i \in M$. Indeed, we have from our earlier work that $i \in M \iff $  $i \in S(\lambda) \backslash S(\nu)$ and $i \in S(\mu) \backslash S(\nu)  \iff $ $ m_i(\lambda) = 1 + m_i(\nu)$ and $ m_i(\mu) = 1 + m_i(\nu)$ $\implies  m_i(\lambda) = m_i(\mu)$. Then we have
\begin{equation}\label{S9blue1}
\begin{split}
\phi_{\lambda / \nu}(0,t) = \left( 1 - {\bf 1}_{c \in S(\lambda) \backslash S(\nu)}t^{m_c(\lambda)}\right) \left(1 - {\bf 1}_{c + 1 \in S(\lambda) \backslash S(\nu)}t^{m_{c+1}(\lambda)}\right)  \prod_{i \in M }(1 - t^{m_i(\lambda)}),  \\
\phi_{\mu / \nu}(0,t) = \left( 1 - {\bf 1}_{c \in S(\mu) \backslash S(\nu)}t^{m_c(\mu)}\right) \left(1 - {\bf 1}_{c + 1 \in S(\mu) \backslash S(\nu)}t^{m_{c+1}(\mu)}\right)  \prod_{i \in M }(1 - t^{m_i(\mu)}).
\end{split}
\end{equation}
A similar argument shows that if $L =  \left[ S(\nu) \backslash S(\lambda) \right] \cap  \left[ S(\nu) \backslash S(\mu) \right]$, then we have
\begin{equation}\label{S9blue2}
\begin{split}
\psi_{\lambda / \nu}(0,t) = \left( 1 - {\bf 1}_{c \in S(\nu) \backslash S(\lambda)}t^{m_c(\nu)}\right) \left(1 - {\bf 1}_{c + 1 \in S(\nu) \backslash S(\lambda)}t^{m_{c+1}(\nu)}\right)  \prod_{i \in L }(1 - t^{m_i(\nu)}),  \\
\psi_{\mu / \nu}(0,t) = \left( 1 - {\bf 1}_{c \in S(\nu) \backslash S(\mu)}t^{m_c(\nu)}\right) \left(1 - {\bf 1}_{c + 1 \in S(\nu) \backslash S(\mu)}t^{m_{c+1}(\nu)}\right)  \prod_{i \in L }(1 - t^{m_i(\nu)}).
\end{split}
\end{equation}

Set 
\begin{equation}\label{S9G}
G(\lambda, \nu, c) := \begin{cases} 1 - {\bf 1}_{\{m_c(\nu) > m_c(\lambda)\}}t^{m_c(\nu)}&\mbox{ if } c > 0,\\ 
                                          1 &\mbox{ otherwise.} \end{cases}
\end{equation}
 Then the above work implies that when $v = (x,y,z)$ and $k = x - y$ we get
\begin{equation}\label{S9probs}
\begin{split}
\mathbb{P}_{r,t}^{\pi,v}(\hat\pi) \propto r G(\hat\lambda^k, \lambda^{k-1},z-1)G(\hat\lambda^k, \lambda^{k-1},z)G( \lambda^{k+1}, \hat\lambda^k,z-1)G( \lambda^{k+1},\hat \lambda^k,z)  \\
 \mathbb{P}_{r,t}^{\pi,v}(\check\pi) \propto  G(\check\lambda^k, \lambda^{k-1},z-1)G(\check\lambda^k, \lambda^{k-1},z)G( \lambda^{k+1}, \check\lambda^k,z-1)G( \lambda^{k+1}, \check\lambda^k,z).
\end{split}
\end{equation}
In obtaining the above formulas we used (\ref{S9blue1}) and (\ref{S9blue2}) for the three different cases $k <0$, $k > 0$ and $k = 0$. Some special care is needed when $k = N$ and in this case the terms in (\ref{S9probs}) involving $\lambda^{k+1}$ are replaced with $1$'s.\\

Summarizing our results, we see that the transition from $\pi$ is as follows: pick a position $v = (x,y,z)$ in the box $N \times N \times N$ uniformly at random; if the position $v$ does not correspond to an element in the sets $Add_\pi$ or $Rem_\pi$ then leave $\pi$ unchanged with probability $1$; if the position $v \in Add_\pi \sqcup Rem_\pi$, then $\pi$ is goes to $\hat \pi$ with probability $p$ and to $\check \pi$ with probability $1-p$, where
\begin{equation}\label{TP}
p:= \frac{ r }{r  +  \frac{ G(\check\lambda^k, \lambda^{k-1},z-1)G( \check\lambda^k, \lambda^{k-1},z)G( \lambda^{k+1}, \check\lambda^k,z-1)G( \lambda^{k+1},  \check\lambda^k,z) }{  G(\hat\lambda^k, \lambda^{k-1},z-1)G(\hat\lambda^k, \lambda^{k-1},z)G( \lambda^{k+1}, \hat\lambda^k,z-1)G( \lambda^{k+1}, \hat\lambda^k,z) }}.
\end{equation}
As before if $k = N$ we replace the terms in the above formula involving $\lambda^{k+1}$ with $1$'s.

%
\subsection{Gibbs sampler algorithm and simulations}\label{Algorithm}\hspace{2mm}\\

In Section \ref{GD} we described a Gibbs sampler for the measure $\mathbb{P}_{r,t}$ and gave exact formulas for the transition probabilities in (\ref{TP}). Our goal now is to give an outline for an algorithm implementing the sampler and present some simulations of random plane partitions. The main difficulty in constructing Gibbs samplers for distributions involving symmetric functions is finding computationally efficient ways to calculate the transition probabilities, which we did in (\ref{TP}). Beyond this formula there are no particularly novel ideas in the algorithm below; however, as we could not find an adequate reference in the literature, we believe that an outline is in order. It is quite possible that different methods can be used to {\em exactly} sample the distribution $\mathbb{P}_{HL}^{r,t}$ or some variant of it, using ideas like those in \cite{BorGorSam}, \cite{VulSam} or \cite{Bodini}. Unfortunately, we were unable to implement exact sampling algorithms efficiently, which is why we resort to the Gibbs sampler and leave the development of better samplers for future work.  \\

One of the difficulties in making simulations is that the number of iterations necessary to obtain convergence is very large. In the cases described below we will need about $2\times 10^{15}$ iterations to see a limit shape emerge. Part of the reason for needing so many iterations is that most of the time the uniformly sampled position $v$ in the $N\times N \times N$ box will not belong to the sets $Add_\pi$ and $Rem_\pi$ and thus the chain will stay in one place for extended periods of time. Let us call steps of the chain, where $v$ was not chosen inside $Add_\pi$ or $Rem_\pi$ {\em empty}; if $v \in Add_\pi \cup Rem_\pi$ we call the step {\em successful}. Empty steps, although individually computationally cheap, add up and significantly increase the runtime of a simulation. It is thus very important to come up with ways to circumvent spending so much time in empty steps.

We will now describe a neat idea that allows us to group together empty steps and thus greatly reduce the runtime of simulations. Let $add_\pi = |Add_\pi|$ and $rem_\pi = |Rem_\pi|$ and observe that the probability of making an empty step, starting from the plane partition $\pi$, is 
$$\mathbb{P}_\pi(v \not \in Add_\pi \cup Rem_\pi) = 1- \frac{add_\pi + rem_\pi}{N^3} =:x_\pi.$$
Consequently, the number of empty steps $E_\pi$, before a successful one, is distributed according to the geometric distribution
\begin{equation}\label{sampleDist}
\mathbb{P}_\pi( E_\pi = k) = x_\pi ^k(1- x_\pi) \mbox{ for } k \geq 0.
\end{equation}

Using the latter observation, instead of sampling $v$ uniformly from the $N\times N \times N$ box, updating our chain and increasing the number of iterations by $1$, we may sample a geometric random variable $X$ with the above distribution, sample $v$ uniformly from $Add_\pi \cup Rem_\pi$ update our chain and increase the number of iterations by $1 + X$. What we have done is calculate beforehand how many empty moves we need to make before we make a successful one and then do all of them together, which by definition means to just do the successful move.

Typically, the cost of drawing an integer-valued random variable $K$ according to some prescribed distribution is of the order of the value $k$ that is finally assigned to $K$ (see the discussion at the end of Section 3 in \cite{Bodini}). An exception is the geometric law, which is simpler. Indeed, to draw $X$ according to (\ref{sampleDist}) it is enough to set $X = \lfloor \log U / \log (x_\pi)\rfloor$, where $U$ is uniform $(0,1)$. Hence, the cost of drawing a geometric law is $O(1)$.

If $N$ is very large, one observes that $x_\pi$ is very close to $1$. Indeed, $add_\pi$ and $rem_\pi$ are both bounded from above by $N^2$, since there can be at most one addable and removable cube in every column $(x,y, \cdot)$. Consequently, one expects to make on average at most $1$ successful step every $N$ steps of the iteration. The upshot of our idea now is that we have replaced sampling a large number of uniform random variables, with sampling a single geometric random variable at cost $O(1)$. Moreover, we have reduced the number of jump commands in our loop, improving runtime further.\\

With the above discussion we are now prepared to describe our algorithm for the Gibbs sampler. We begin with a brief description of random number generators. 
{\tt Bernoulli}$(p)$ samples a Bernoulli random variable $X$ with parameter $p$, i.e. $\mathbb{P}(X = 1) = p$ and $\mathbb{P}(X = 0) =1- p$. {\tt Geom}$(p)$ samples a geometric random variable $X$ with parameter $p$, i.e. $\mathbb{P}(X = k) = p^k(1-p)$ for $k \geq 0$. {\tt Uniform}$(n)$ samples a uniform random variable $X$ on $\{1,...,n\}$, i.e. $\mathbb{P}(X = k) = 1/n$ for $k = 1,...,n$. The random number generator algorthms are described below.

\vspace{2mm}
\begin{center}
\begin{tabular}{l}
\hline
{\tt Bernoulli}$(p)$\\
\hline
$U$ := uniform(0,1);\\
{\bf if } $U < $p {\bf return} $1$;\\
{\bf else return} $0$; \\
\hline
\end{tabular}
\quad\hspace{10mm}
\begin{tabular}{l}
\hline
 {\tt Geom}$(p)$ \\
\hline
$U$ := uniform(0,1); \\
  {\bf return} $\lfloor \log (U) / \log p \rfloor$;\\
\hspace{4mm}\\
\hline
\end{tabular}
\quad\hspace{10mm}
\begin{tabular}{l}
\hline
{\tt Uniform}$(n)$\\
\hline
 $U$ := uniform(0,1);\\
 {\bf return} $1 + \lfloor nU\rfloor$;\\
\hspace{4mm}\\
\hline
\end{tabular}
\end{center}
\vspace{2mm}

Next we consider the following functions, which perform the basic operations on plane partitions necessary for running the Glauber dynamics. In the functions below we recall that for a plane partition $\pi$, $add_\pi$ and $rem_\pi$ are the number of cubes that can be added to and removed from $\pi$ respectively, so that the result is a plane partition contained in $N \times N \times N$.
\vspace{2mm}

\begin{center}
 \begin{tabular}{l}
\hline
{\tt AddCube}$(\pi, k)$ \\
\hline
 {\bf Input:} $\pi$; index $k\in \{1,...,add_\pi\}$.\\
\hspace{5mm} Add the $k$-th addable cube to $\pi$.\\
\hline
\end{tabular}
\quad\hspace{10mm}
\begin{tabular}{l}
\hline
 {\tt RemCube}$(\pi, k)$\\
\hline
 {\bf Input:} $\pi$; index $k\in \{1,...,rem_\pi\}$.\\
\hspace{5mm} Remove the $k$-th removable cube from $\pi$.\\
\hline
\end{tabular}
\end{center}

\vspace{5mm}

\begin{center}
 \begin{tabular}{l}
\hline
{\tt GetAdd}$(\pi, k)$ \\
\hline
 {\bf Input:} $\pi$; index $k\in \{1,...,add_\pi\}$.\\
 {\bf Output:} The position $(x,y,z)$ of the \\
\hspace{15mm} $k$-th addable cube.\\
\hline
\end{tabular}
\quad\hspace{10mm}
\begin{tabular}{l}
\hline
 {\tt GetRem}$(\pi, k)$\\
\hline
 {\bf Input:} $\pi$; index $k\in \{1,...,rem_\pi\}$.\\
{\bf Output:} The position $(x,y,z)$ of the \\
\hspace{15mm} $k$-th removable cube.\\
\hline
\end{tabular}
\end{center}

\vspace{5mm}

\begin{tabular}{l}
\hline
{\tt GetMult}$(\pi, k, c)$ \\
\hline
 {\bf Input:} $\pi$, $k$ - slice index, $c \geq 0$.\\
{\bf Output:} $m_c(\lambda^k) $ -  multiplicity of $c$ in the $k$-th slice of $\pi$.\\
\hspace{16mm} If $c = 0$ the output is $-2$.\\
\hline
\end{tabular}
\quad\begin{tabular}{l}
\hline
{\tt WeightG}$(m,n,t)$ \\
\hline
{\bf if } (($n < 0$) {\bf or} ($m < 0$)) {\bf return } $1$;\\
 {\bf if } $m >n$ {\bf return} $(1-t^m)$;\\
{\bf return} $1$;\\
\hline
\end{tabular}

\vspace{2mm}

With the above functions we now write an algorithm, which runs the Glauber dynamics for some predescribed number of iterations.

\vspace{2mm}

\begin{tabular}{l}
\hline
{\bf Algorithm} {\tt GibbsSampler}$(\pi, N, T,r,t)$ \\
\hline
 {\bf Input:} $\pi$ - initial plane partition, $N$ - size of box, $T$ - total number of iterations,\\
\hspace{13mm} $r\in (0,1)$ and $t \in [-1,1]$ - parameters of the distribution.\\
$iter := 0$; \\
{\bf while} ($iter < T$) {\bf do}\\
\hspace{5mm} $X:=$ {\tt Geom}$\left(1 -  \frac{add_\pi + rem_\pi}{N^3}\right)$;\\
\hspace{5mm} $iter = iter + X$;\\
\hspace{5mm} {\bf if } $(iter \geq T)$ {\bf break};\\
\hspace{5mm} $u :=$ {\tt Uniform}($add_\pi + rem_\pi$);\\
\hspace{5mm} {\bf if } ($u < add_\pi$) \\
\hspace{12mm}  $(x,y,z): =$ {\tt GetAdd}($\pi,u$); \\
\hspace{12mm} $ k: = x - y$; \\
\hspace{12mm} $w_1: = r *$ {\tt WeightG}({\tt GetMult}($\pi, k-1,z$), {\tt GetMult}($\pi, k,z) +1,t) $; \\
\hspace{12mm} $w_1 = w_1 *$ {\tt WeightG}({\tt GetMult}($\pi, k-1,z-1$), {\tt GetMult}($\pi, k,z-1) -1,t) $; \\
\hspace{12mm} $w_2: =$ {\tt WeightG}({\tt GetMult}($\pi, k-1,z$), {\tt GetMult}($\pi, k,z),t)$; \\
\hspace{12mm} $w_2 = w_2 *$ {\tt WeightG}({\tt GetMult}($\pi, k-1,z-1$), {\tt GetMult}($\pi, k,z-1),t)$; \\
\hspace{12mm} {\bf if } ($k < N$)\\
\hspace{19mm} $w_1 = w_1 *$ {\tt WeightG}({\tt GetMult}($\pi, k,z) +1$, {\tt GetMult}($\pi, k+1,z),t)$; \\
\hspace{19mm} $w_1 = w_1 *$ {\tt WeightG}({\tt GetMult}($\pi, k,z-1$)$-1$, {\tt GetMult}($\pi, k+1,z-1),t)$; \\
\hspace{19mm} $w_2 = w_2 *$ {\tt WeightG}({\tt GetMult}($\pi, k,z$), {\tt GetMult}($\pi, k+1,z),t)$; \\
\hspace{19mm} $w_2 = w_2 *$ {\tt WeightG}({\tt GetMult}($\pi, k,z-1$), {\tt GetMult}($\pi, k+1,z-1),t)$; \\
\hspace{12mm} {\bf end}\\
\hspace{12mm} $p:= w_1/(w_1 + w_2)$;\\
\hspace{12mm} $B:=$ {\tt Bernoulli}($p$);\\
\hspace{12mm} {\bf if} ($B == 1$) {\tt AddCube}($\pi, u$);\\
\hspace{5mm} {\bf else}\\
\hspace{12mm}  $(x,y,z): =$ {\tt GetRem}($\pi,u - add_\pi$); \\
\hspace{12mm} $ k: = x - y$; \\
\hspace{12mm} $w_1: = $ {\tt WeightG}({\tt GetMult}($\pi, k-1,z$), {\tt GetMult}($\pi, k,z) -1,t) $; \\
\hspace{12mm} $w_1 = w_1 *$ {\tt WeightG}({\tt GetMult}($\pi, k-1,z-1$), {\tt GetMult}($\pi, k,z-1) +1,t) $; \\
\hspace{12mm} $w_2: = r*$ {\tt WeightG}({\tt GetMult}($\pi, k-1,z$), {\tt GetMult}($\pi, k,z),t)$; \\
\hspace{12mm} $w_2 = w_2 *$ {\tt WeightG}({\tt GetMult}($\pi, k-1,z-1$), {\tt GetMult}($\pi, k,z-1),t)$; \\
\hspace{12mm} {\bf if } ($k < N$)\\
\hspace{19mm} $w_1 = w_1 *$ {\tt WeightG}({\tt GetMult}($\pi, k,z) -1$, {\tt GetMult}($\pi, k+1,z),t)$;  \\
\hspace{19mm}  $w_1 = w_1 *$ {\tt WeightG}({\tt GetMult}($\pi, k,z-1)+1$, {\tt GetMult}($\pi, k+1,z-1),t)$; \\
\hspace{19mm}  $w_2 = w_2 *$ {\tt WeightG}({\tt GetMult}($\pi, k,z$), {\tt GetMult}($\pi, k+1,z),t)$; \\
\hspace{19mm}  $w_2 =  w_2 *$ {\tt WeightG}({\tt GetMult}($\pi, k,z-1$), {\tt GetMult}($\pi, k+1,z-1),t)$\\
\hspace{12mm} {\bf end}\\
\hspace{12mm} $p:= w_1/(w_1 + w_2);$\\
\hspace{12mm} $B:=$ {\tt Bernoulli}($p$);\\
\hspace{12mm} {\bf if} ($B == 1$) {\tt RemCube}($\pi, u - add_\pi$);\\
\hspace{5mm}{\bf end}\\
\hspace{5mm} $iter = iter + 1;$\\
{\bf end}\\
{\bf Output:} $\pi$.\\
\hline
\end{tabular}

\vspace{2mm}

\begin{remark} In the above algorithm, an expression of the form \\
{\tt WeightG}({\tt GetMult}($\pi, \cdot,\cdot$), {\tt GetMult}($\pi, \cdot ,\cdot),t)$ simulates the function $G$, given in (\ref{S9G}). The case $z = 1$ is special, since $G$ is defined differently depending on $c > 0$ and $c = 0$. In order to make the algorithm more concise, and exclude additional checks of whether $z = 1$, we have rigged the functions {\tt GetMult} and {\tt WeightG} so that the end results agree with (\ref{S9probs}).
\end{remark}

Using {\tt GibbsSampler} we can run different simulations, to verify empirically some of the results from this paper. In particular, we will be interested in showing that the bottom slice of a plane partition, distributed according to $\mathbb{P}_{r,t}$ with $N$ taken very large, does indeed converge to the desired limit shape. At this time we remark that we have not done any analysis to estimate the mixing time of the chain we have constructed, hence our choice of number of iterations below will be somewhat arbitrary. The major point to be made here is that we are only interested in qualitative information about the distribution, such as a limit-shape phenomenon, and the purpose of the iterations is to pictorially support statements for which we have analytic proofs.\\

In the simulations below, the sampler is started from $\pi = \varnothing$, the size of the box $N = 2000$, the number of iterations is $T = 2\times 10^{15}$ and $r = 0.99$. The only parameter we will vary is $t$. Results are summarized in Figures \ref{S9_2} - \ref{S9_5}, where the red curve indicates the limit shape from our results. 

\begin{figure}[h]
\centering
\begin{minipage}{.5\textwidth}
  \centering
  \includegraphics[width=0.9\linewidth]{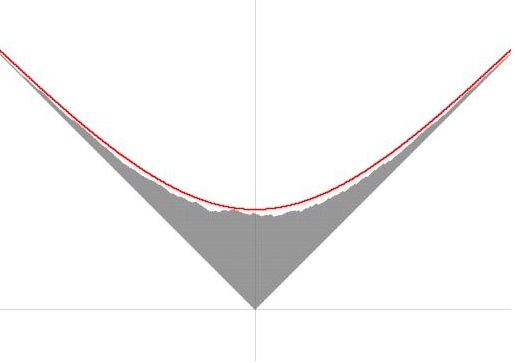}
\captionsetup{width=.9\linewidth}
  \caption{$ t= 0$.}
  \label{S9_2}
\end{minipage}%
\begin{minipage}{.5\textwidth}
  \centering
  \includegraphics[width=0.9\linewidth]{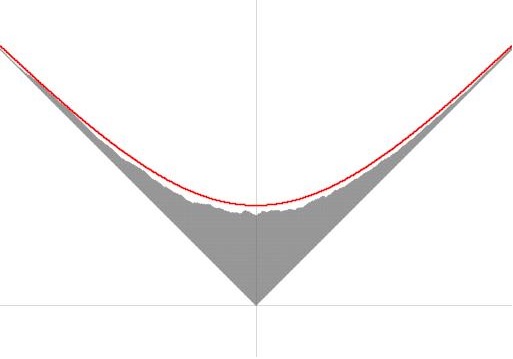}
\captionsetup{width=.9\linewidth}
  \caption{$t = 0.2$. }
  \label{S9_3}
\end{minipage}
\end{figure}

\begin{figure}[h]
\centering
\begin{minipage}{.5\textwidth}
  \centering
  \includegraphics[width=0.9\linewidth]{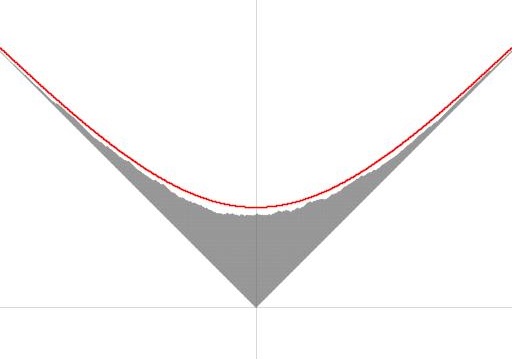}
\captionsetup{width=.9\linewidth}
  \caption{$t = 0.4$.}
  \label{S9_4}
\end{minipage}%
\begin{minipage}{.5\textwidth}
  \centering
  \includegraphics[width=0.9\linewidth]{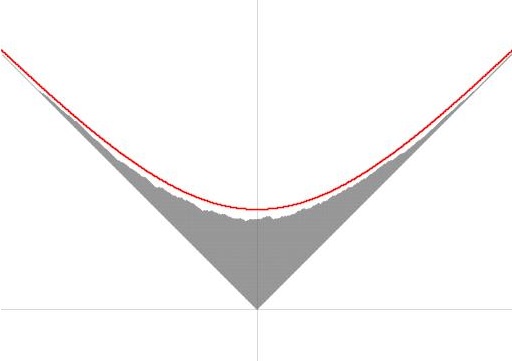}
\captionsetup{width=.9\linewidth}
  \caption{$t = 0.6$. }
  \label{S9_5}
\end{minipage}
\end{figure}
\FloatBarrier

As can be seen from the above figures, the behavior of the bottom slice asymptotically does not depend on the parameter $t$, and the behavior nicely agrees with the predictions from Theorem \ref{TW}. 

%
\subsection{Conjectural covergence to the Airy and KPZ line enesembles} \hspace{2mm}\\

In this section we state a couple of conjectures about the convergence of $\mathbb{P}_{HL}^{r,t}$ that go beyond the results of this paper. At this time we do not have any clear strategy on how they can be proved, however, we will provide some evidence for their validity. We start by rather informally recalling the definitions and properties of the Airy and KPZ line ensembles. For more details about these objects the reader is encouraged to look at \cite{CorHamA} and \cite{CorHamK}, where they were introduced and analyzed. \\

Let $B_1^N,...,B_N^N$ be $N$ independent standard Brownian bridges on $[-N,N]$, $B_i^N(-N) = B_i^N(N) = 0$, conditioned on not intersecting in $(-N,N)$ and set $\Sigma_N = \{1,...,N\}$. The latter object can be viewed as a {\em line ensemble}, i.e. a random variable with values in the space $X$ of continuous functions $f: \Sigma_N \times [-N,N] \rightarrow \mathbb{R}$ endowed with the topology of uniform convergence on compact subsets of $\Sigma_N \times [-N,N]$. In \cite{CorHamA} these line ensembles are called {\em Dyson line ensembles} and it is shown that under suitable shifts and scaling they converge (in the sense of line ensembles - see the discussion at the beginning of Section 2.1 in \cite{CorHamA}) to a continuous non-intersecting $\mathbb{N} \times \mathbb{R}$-indexed line ensemble. The limit is called the {\em Airy line ensemble} and is denoted by$\mathcal{A} : \mathbb{N} \times \mathbb{R} \rightarrow \mathbb{R}.$
The two properties of $\mathcal{A}$ that we will focus on are that $\mathcal{A}_1(t)$ is distributed according to the {Airy process} and that the $\mathbb{N}$-indexed line ensemble $\mathcal{L}: \mathbb{N}\times \mathbb{R} \rightarrow \mathbb{R}$, given by $\mathcal{L}_i(x):=2^{-1/2} (\mathcal{A}_i(x) - x^2)$ for each $i \in \mathbb{N}$ satisfies a certain {\em Brownian Gibbs property} that we describe below.

The Airy process first appeared in the paper of Pr{\" a}hofer and Spohn \cite{Spohn}, as the scaling limit of the fluctuations of the PNG droplet and it is believed to be the universal scaling limit of a large class of stochastic growth models. It's single time distribution is given by the GUE Tracy-Widom distribution.

We now describe an instance of the Brownian Gibbs property, satisfied by $\mathcal{L}_i$. Let $k \geq 2$, and consider the curves $\mathcal{L}_{k-1}, \mathcal{L}_{k}$ and $\mathcal{L}_{k+1}$. Let $a, b \in \mathbb{R}$ and $a < b$ be given and put $x = \mathcal{L}_k(a),$ $y = \mathcal{L}_k(b)$. Then if we erase $\mathcal{L}_k([a,b])$ and sample an independent Brownian bridge on $[a,b]$ between the points $x$ and $y$, conditional on not intersecting $\mathcal{L}_{k-1}$ and $\mathcal{L}_{k+1}$, then the new line ensemble has the same distribution as the old one.\\

We shift our attention to the KPZ line ensemble. Let $N \in \mathbb{N}$ and $s > 0$ be given. For each sequence $0 < s_1 < \cdots < s_{N-1} < s$ we can associate an up/right path $\phi$ in $[0,s] \times \{1,...,N\}$ that is the range of the unique non-decreasing surjective map $[0,s] \rightarrow \{1,...,N\}$ whose set of jump times is $\{s_i\}_{i=1}^{N-1}$. Let $B_1, ..., B_N$ be independent standard Brownian motions and define
$$E(\phi) = B_1(s_1) + (B_2(s_2) - B_2(s_1)) + \cdots + (B_N(s) - B_{N}(s_{N-1})).$$ 
The {\em O'Connel-Yor polymer partition function line ensemble} is a $\{1,...,N\} \times \mathbb{R}_+$-indexed line ensemble $\{Z_n^N(s): n\in \{1,...,N\}, s > 0\}$, defined by
$$Z_n^N(s):= \int_{D_n(s)} \exp \left( \sum_{i = 1}^n E(\phi_i) \right)d\phi_1 \cdots d\phi_n, $$
where the integral is with respect to Lebesgue measure on the Euclidean set $D_n(s)$ of all $n$-tuples of non-intersecting (disjoint) up/right paths $\phi_1,...,\phi_n$ with initial points $(0,1), ..., (0,n)$ and endpoints $(s,N-n+1), ..., (s,N)$. Setting $Z_0^N(s) \equiv 1$ we define the {\em O'Connel-Yor polymer free energy line ensemble} as the $\{1,...,N\} \times \mathbb{R}_+$-indexed line ensemble $\{X_n^N(s): n \in \{1,...,N\}, s> 0\}$ defined by
$$X_n^N(s) = \log \left( \frac{Z_n^N(s)}{Z_{n-1}^N(s)}\right).$$
In \cite{CorHamK} it was shown that under suitable shifts and scaling the line ensembles $X_n^N(\sqrt{tN} + \cdot)$ are sequentially compact and hence have at least one weak limit, called the {\em KPZ$_t$ line ensemble} and denoted by $\mathcal{H}^t: \mathbb{N} \times \mathbb{R} \rightarrow \mathbb{R}$. The uniqueness of this limit is an open problem, however any weak limit has to satisfy the following two properties. The lowest index curve $\mathcal{H}^t_1: \mathbb{R} \rightarrow \mathbb{R}$ is equal in distribution to $\mathcal{F}(t, \cdot)$ - the time $t$ Hopf-Cole solution to the narrow wedge initial data KPZ equation (see Definition \ref{freeE}). In addition, the ensemble $\mathcal{H}^t$ satisfies a certain ${\bf H}_1$- {\em Brownian Gibbs property}, an instance of which we now describe. 

 Let $k \geq 2$, and consider the curves $\mathcal{H}^t_{k-1}, \mathcal{H}^t_{k}$ and $\mathcal{H}^t_{k+1}$. Let $a, b \in \mathbb{R}$ and $a < b$ be given and put $x = \mathcal{H}^t_k(a),$ $y = \mathcal{H}^t_k(b)$. We erase $\mathcal{H}^t_k([a,b])$ and sample an independent Brownian bridge on $[a,b]$ between the points $x$ and $y$. The new path is accepted with probability
$$\exp \left[ - \int_a^b {\bf H}_1\left( \mathcal{H}^t_{k+1}(u) - \mathcal{H}^t_k(u)\right)du - \int_a^b {\bf H}_1\left( \mathcal{H}^t_{k}(u) - \mathcal{H}^t_{k-1}(u)\right)du\right], \hspace{5mm} {\bf H}_t(x) = e^{t^{1/3}x},$$
and if the path is not accepted we sample a new Brownian bridge and repeat. This procedure yields a new line ensemble and it  has the same distribution as the old one.

The {\em Hamiltonian} ${\bf H}_t$ acts as a potential in which the Brownian paths evolve, assigning more weight to certain path configurations. Formally, setting $t = \infty$ we have $H_{+\infty}(x) = \infty$ if $x > 0$ and $0$ if $x < 0$. This Hamiltonian corresponds to conditioning consecutively labeled curves to not touch and hence reduces the ${\bf H}$- Brownian Gibbs property to the Brownian Gibbs property we had earlier.\\

 For $\tau > 0$ let $f(\tau) = 2\log (1 + e^{-\tau/2})$, $f'(\tau) = -\frac{e^{-\tau/2}}{1 + e^{-\tau/2}}$ and $f''(\tau) = \frac{1}{2} \frac{e^{-\tau/2}}{(1 + e^{-\tau/2})^2}$. Also set $N(r) = \frac{1}{1-r}$. With this notation we have the following conjectures.

\begin{conjecture}\label{Conj1}
Consider the measure $\mathbb{P}_{HL}^{r,t}$ on plane partitions, given in (\ref{PDEF}), with $t \in (0,1)$ fixed. For $\tau \in \mathbb{R}$ define the random $\mathbb{N} \times \mathbb{R}$-indexed  line ensemble $\Lambda^\tau$ as
\begin{equation}\label{conj1E}
\Lambda^\tau_k(s) = \frac{\lambda_k'( \lfloor\tau N + sN^{2/3} \rfloor) - Nf(|\tau|) - sN^{2/3}f'(|\tau|) - (1/2)s^2N^{1/3}f''(|\tau|)}{\sqrt[3]{2f''(|\tau|) N}}.
\end{equation}
 Then as $r \rightarrow 1^-$ we have $\Lambda^\tau \implies \mathcal{A}^\tau$ (weak convergence in the sense of line ensembles), where $\mathcal{A}^\tau$ is defined as $\mathcal{A}^\tau_k(s) = \mathcal{A}_k(s\sqrt[3]{2f''(|\tau|)}/2)$ and $(\mathcal{A}_k)_{k \in \mathbb{N}}$ is the Airy line ensemble.
\end{conjecture}

\begin{conjecture}\label{Conj2}
Consider the measure $\mathbb{P}_{HL}^{r,t}$ on plane partitions, given in (\ref{PDEF}). Suppose $T > 0$ is fixed and $\frac{- \log t}{(1-r)^{1/3}} =\frac{(T/2)^{1/3}}{ \sqrt[3]{2f''(|\tau|)}}$. For $\tau \in \mathbb{R}$ define the random $\mathbb{N} \times \mathbb{R}$-indexed  line ensemble $\Xi^\tau$ as
\begin{equation}\label{conj2E}
\begin{split}
\Xi^\tau_k(s) =  \frac{ \lambda_k'( \lfloor\tau N + sN^{2/3} \rfloor)  - Nf(|\tau|) - sN^{2/3}f'(|\tau|) - (1/2)s^2N^{1/3}f''(|\tau|)}{(T/2)^{-1/3}\sqrt[3]{2f''(|\tau|) N} } + \\
 + \log ((T/2)^{-1/3}\sqrt[3]{2f''(|\tau|) N} ) + (k-1)\log\left( \frac{ NT^{-1}(2f''(|\tau|))^{-3/2}}{2\sqrt{2}}\right) - \frac{s^2T^{1/3} (2f''(|\tau|))^{2/3}}{8} - T/24.
\end{split}
\end{equation}
Then as $r \rightarrow 1^-$ we have $\Xi^{\tau} \implies \mathcal{H}^{\tau, T}$ (weak convergence in the sense of line ensembles), where $\mathcal{H}^{\tau, T}$ is defined as $\mathcal{H}^{\tau, T}_k(s) = \mathcal{H}^T_k(sT^{2/3}\sqrt[3]{2f''(|\tau|)}/2)$ and $(\mathcal{H}_k^T)_{k \in \mathbb{N}}$ is the KPZ line ensemble.
\end{conjecture}

\begin{remark} 
We provide some motivation behind our choice of scaling in Conjecture \ref{Conj1}. Since the lines in the Airy line ensemble a.s. do not intersect as do the lines $\lambda_k'( \lfloor\tau N + sN^{2/3}\rfloor)$ we expect that all lines undergo the same scaling and translation. This allows us to only concern ourselves with $\lambda_1'( \lfloor\tau N + sN^{2/3}\rfloor)$, whose limit should be some rescaled version of the Airy process (the distribution of $\mathcal{A}_1$). Arguments in the proof of Theorem \ref{TW} can be used to show that in distribution the expression on the RHS in (\ref{conj1E}) converges to the GUE Tracy-Widom distribution for each $s$. The latter still leaves the question of possible argument scaling since $\mathcal{A}_1(\kappa s)$ has the same one-point marginal distribution for all values of $\kappa$. In \cite{FSpohn} an expression similar to $\Lambda_1^{\tau}(s)$ (related to setting $t = 0$ in $\mathbb{P}^{r,t}_{HL}$), was shown to converge to the Airy process, with a rescaled argument. Consequently, we have chosen to rescale the argument so that it matches this result. 
\end{remark}

\begin{remark}
The choice for scaling in Conjecture \ref{Conj2} is somewhat more involved. When $k = 1$ in equation (\ref{conj2E}) we run into the same argument scaling issue as in Conjecture \ref{Conj1}; however, we no longer have results in the literature that we can use as a guide. Nevertheless, in \cite{CorQ} it was conjectured that $(T/2)^{-1/3}\left(\mathcal{F}(T, T^{2/3}X) + \frac{T^{4/3}X^2}{2T} + \frac{T}{24}\right)$ converges to the Airy process as $T \rightarrow \infty$. Consequently, we have picked a scaling of the argument in Conjecture \ref{Conj2} in such a way that under the scaling by $(T/2)^{-1/3}$ we would obtain the (argument rescaled) Airy process in Conjecture \ref{Conj1}. Since the lines in the KPZ line ensemble are allowed to cross, we no longer expect that all lines $ \lambda_k'( \lfloor\tau N + sN^{2/3} \rfloor) $ undergo the same translation and scaling and in equation (\ref{conj2E}) we see that each line is deterministically shifted by a $N^{1/3}\log(N)$ factor compared to the previously indexed line. The precise choice of this shift is explained below and it is related to the ${\bf H}_1$-  Brownian Gibbs property, enjoyed by the KPZ line ensemble.
\end{remark}

We will now present some evidence that supports the validity of the above conjectures, starting from the results of this paper. Theorems \ref{TW} and \ref{TCDRP} only deal with $\lambda_1'$ and can be understood as one-point convergence results about the bottom slice of the partition $\pi$ as follows. The proof of Theorem \ref{TW} shows that 
$$\lim_{r \rightarrow 1^-}\mathbb{P}_{HL}^{r,t}\left( \frac{\lambda_1'( \lfloor\tau N + sN^{2/3} \rfloor) - M(r)}{\sqrt[3]{2f''(|\tau|) N}} \leq x \right) = F_{GUE}(x) = \mathbb{P}(\mathcal{A}^{\tau}_1(s) \leq x).$$
In the last equality we used that the one-point distribution of the Airy process is given by the Tracy-Widom GUE distribution \cite{Spohn}. In the above formula we have
$M(r) = 2\sum_{k = 1}^\infty a(r)^k  \frac{(-1)^{k+1}}{1-r^k}$, where $a(r) = r^{\lfloor N(r)\tau +  sN(r)^{2/3}\rfloor}$. Using ideas that are similar to those in Lemma \ref{LdelM} one obtains $M(r) = Nf(|\tau|) + sN^{2/3}f'(|\tau|) +(1/2)s^2N^{1/3}f''(|\tau|) + O(1)$. Consequently, Theorem \ref{TW} implies that the one-point distribution of $\Lambda^{\tau}_1$ converges to that of $\mathcal{A}^\tau_1$. 

Similarly, the proof of Theorem \ref{TCDRP} shows that  
$$\lim_{r \rightarrow 1^-}\mathbb{P}_{HL}^{r,t}\left( \frac{\lambda_1'( \lfloor\tau N + sN^{2/3} \rfloor) - M(r)}{(T/2)^{-1/3}\sqrt[3]{2f''(|\tau|) N}}+ \log ((T/2)^{-1/3}\sqrt[3]{2f''(|\tau|) N} ) \leq x \right) = F_{CDRP}(x) = $$
$$ = \mathbb{P}\left(\mathcal{H}^T_1(sT^{2/3}\sqrt[3]{2f''(|\tau|)}/2) +\frac{s^2T^{1/3} (2f''(|\tau|))^{2/3}}{8} + T/24\leq x\right).$$
In the last equality we used that $\mathcal{F}(T,X) + \frac{X^2}{2T}$ is a stationary process in $X$ and hence $\mathcal{F}(T,0) + T/24$ has the same distribution as $\mathcal{H}^T_1(sT^{2/3}\sqrt[3]{2f''(|\tau|)}/2) +\frac{s^2T^{1/3} (2f''(|\tau|))^{2/3}}{8} + T/24$. In the above formula we have
$M(r) = 2\sum_{k = 1}^\infty a(r)^k  \frac{(-1)^{k+1}}{1-r^k}$, where $a(r) = r^{\lfloor N(r)\tau +  sN(r)^{2/3}\rfloor}$. Using $M(r) = Nf(|\tau|) + sN^{2/3}f'(|\tau|) +(1/2)s^2N^{1/3}f''(|\tau|) + O(1)$ we see that Theorem \ref{TCDRP} implies that the one-point distribution of $\Xi^{\tau}_1$ converges to that of $\mathcal{H}^{\tau,T}_1$. \\

The next observation that we make is that in the statement of Conjecture \ref{Conj1}, the separation between consecutive horizontal slices of $\pi$, distributed according to $\mathbb{P}_{HL}^{r,t}$ is suggested to be of order $N^{1/3}$, which is the order of the fluctuations. On the other hand, in Conjecture \ref{Conj2} there is a deterministic shift of order $N^{1/3}\log N$, while fluctuations remain of order $N^{1/3}$. The latter phenomenon can be observed in simulations, as is shown in Figures \ref{S9_7} and \ref{S9_6}. Namely, the conjectures suggest that as $t$ goes to $1$, one should observe a larger spacing between the bottom slices of $\pi$, which is clearly visible.\\
\begin{figure}[h]
\centering
\scalebox{0.57}{\includegraphics{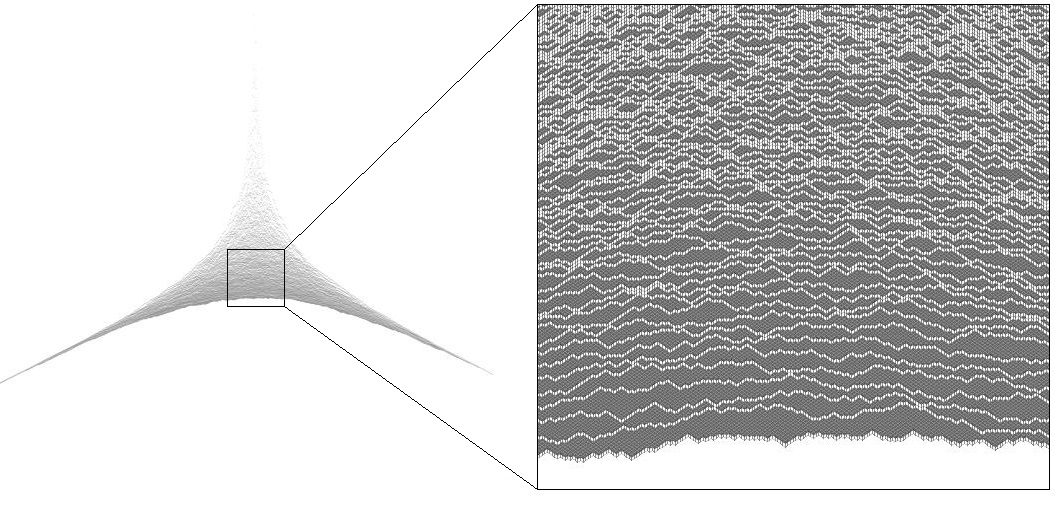}}
\caption{Simulation with $t = 0.4$.}
\label{S9_7}
\end{figure}

\begin{figure}[h]
\centering
\scalebox{0.57}{\includegraphics{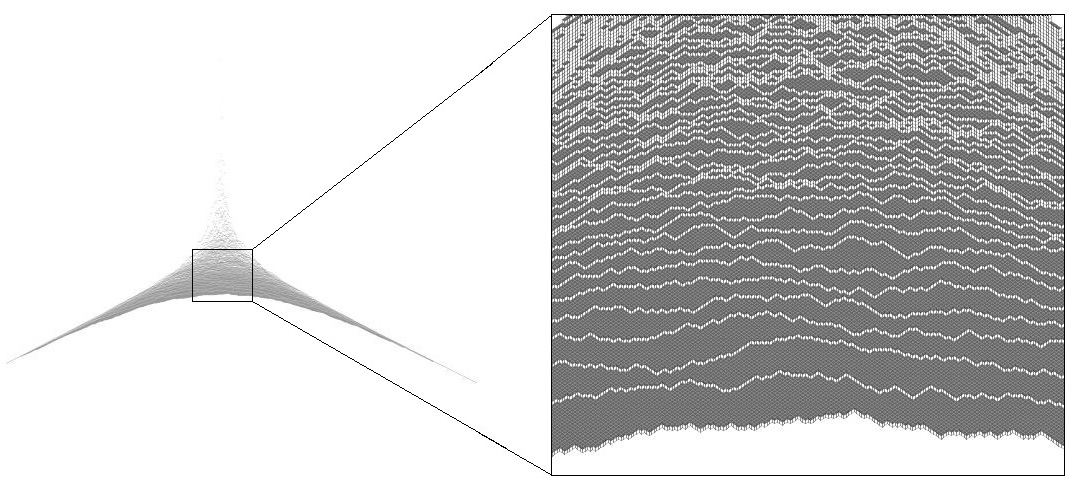}}
\caption{Simulation with $t = 0.8$.}
\label{S9_6}
\end{figure}

Finally, we match the Brownian Gibbs and ${\bf H}_1$-Brownian Gibbs properties. Suppose that we fix the slices $\lambda'_{k-1}(m)$ and $\lambda'_{k+1}(m)$, $m\in \mathbb{Z}$ and consider the conditional distribution of $\lambda_k'([A,B])$. The weight $w(\lambda_k'([A,B])$ that each path obtains consists of two terms: an {\em entropy term}, which comes from the $r^{|\pi|}$ dependence of $\mathbb{P}_{HL}^{r,t}$, and a {\em potential term}, which comes from the dependence on $A_\pi(t)$. Specifically, if the number of cells between $\lambda_k'([A,B])$ and $\lambda'_{k+1}([A,B])$ is $P$ then the entropy term is given by $r^{P}$. The potential term is a bit more involved but depends only on the local structure of the paths. It is constructed as follows: start from $A$ and move to the right towards $B$, every time the distance between $\lambda_k'(m)$ and $\lambda'_{k \pm 1}(m)$ decreases by $1$ when we increase $m$ by $1$ we obtain a factor of $(1 - t^{| \lambda_k'(m) - \lambda'_{k \pm 1}(m)|})$; the potential term is now the product of these factors. The weight $w(\lambda_k'([A,B])$ is given by the product of the entropy and potential terms and the conditional probability is the ratio of the weight and the sum of all path weights. See Figure \ref{S9_8} for a pictorial depiction of the latter construction.

\begin{figure}[h]
\centering
\begin{minipage}{.5\textwidth}
  \centering
  \includegraphics[width=0.9\linewidth]{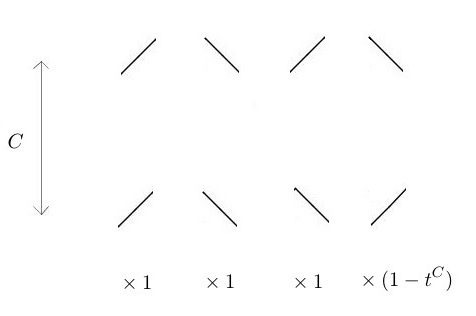}
\captionsetup{width=.9\linewidth}
\end{minipage}%
\begin{minipage}{.6\textwidth}
\hspace{-7mm}
  \includegraphics[width=0.9\linewidth]{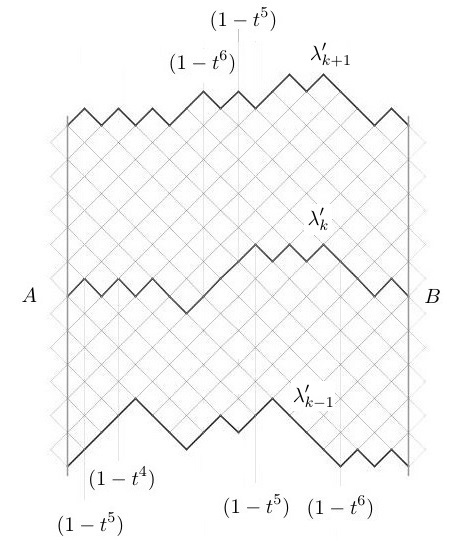}
\captionsetup{width=.9\linewidth}
\end{minipage}
  \caption{The left part of the figure shows that we get a non-trivial factor only when the distance between two slices decreases. For the path on the right we have $P = 16\times 5 + 2\times 6 + 1\times 4 = 96$, hence the entropy term is $r^{96}$. The potential term is given by $(1-t^4)\times(1-t^5)^3\times(1-t^6)^2$. The weight is the product of the entropy and potential terms and equals $w(\lambda_k'([A,B])=r^{96}(1-t^4)(1-t^5)^3(1-t^6)^2$.}
 \label{S9_8}
\end{figure}

In the limit as $r \rightarrow 1^-$, the entropy term goes to $1$ and if we ignore the potential, we see that the measure converges to the uniform measure on all paths from $A$ to $B$, which do not intersect the lines $\lambda_{k-1}'$ and $\lambda_{k+1}'$. This motivates the Brownian limit of the paths. When $t \in (0,1)$ is fixed, we have the conjectural separation of consecutive lines in $\Lambda^{\tau}$ being of order $N^{1/3}$. This implies that if $B-A$ is of order $N^{2/3}$, which is the conjectural scaling we have suggested, then the potential term is bounded from below by an expression of the form $(1 - t^{cN^{1/3}})^{CN^{2/3}}$. The latter converges to $1$ exponentially fast, and so we see that the contribution of the potential disappears in the limit. Consequently, the limit distribution of $\mathcal{A}^{\tau}_k$, at least heuristically, converges to a Brownian path, which is conditioned on not intersecting  $\mathcal{A}^{\tau}_{k\pm1 }$. This is precisely the Brownian Gibbs property.\\

When both $r$ and $t$ converge to $1^-$ as in Conjecture \ref{Conj2}, the potential term can no longer be ignored. One can understand the contribution of the potential term as an acceptance probability similarly to the KPZ line ensemble. Specifically, suppose we fix the slices $\lambda'_{k-1}(m)$ and $\lambda'_{k+1}(m)$, $m\in \mathbb{Z}$ and consider the conditional distribution of $\lambda_k'([A,B])$. One way to obtain it is to draw a random path between the points $A$ and $B$ that does not intersect the slices $\lambda'_{k-1}(m)$ and $\lambda'_{k+1}(m)$ using the entropy term alone. Then with probability equal to the potential term we accept the path and otherwise we draw again and repeat. When $r$ and $t$ go to $1^-$ we have that the paths we sample converge to a uniform sampling of all paths, suggesting the Brownian nature of the limits; and what we would like to show is that the acceptance probability in the discrete case converges to the acceptance probability in the limit. Notice that the separation between slices being of order $N^{1/3}\log (N)$, while fluctuations remaining of order $N^{1/3}$ suggests that non-intersection of the lines automatically holds with large probability and hence can be ignored.

\begin{figure}[h]
\centering
\scalebox{0.6}{\includegraphics{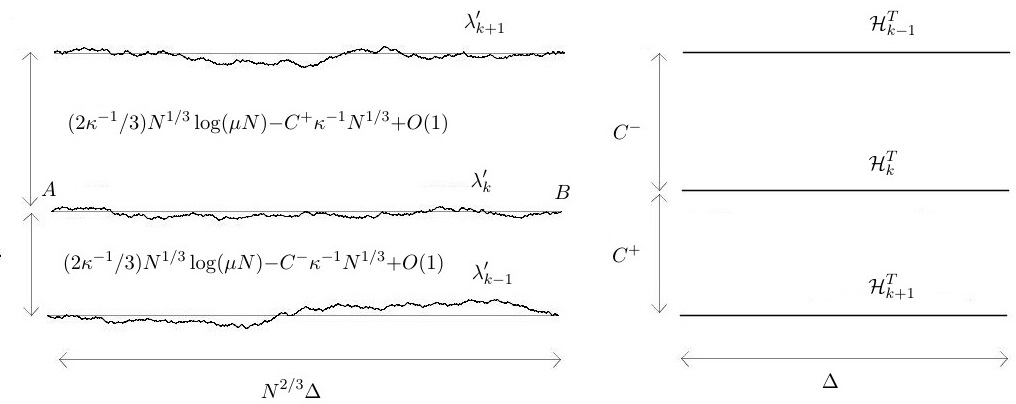}}
\caption{$\Xi^{\tau}_k$ and $\Xi^{\tau}_{k\pm 1}$ converge to constant functions. Quantities increase downwards.}
\label{S9_10}
\end{figure}

We will now proceed to match the acceptance probabilities, by considering a simple to analyze case, when the paths converge to constant lines. The situation is depicted in Figure \ref{S9_10}. To simplify notation, let $\Delta = N^{-2/3}(B - A)$, $\chi^{-1} = \sqrt[3]{2f''(|\tau|)}$, $\kappa = (T/2)^{1/3}\chi^{-1}= (-\log t)N^{1/3}$ and $\mu =\frac{ T^{-1}\chi^{3/2}}{2\sqrt{2}}$. Due to the Brownian nature of the limit of the paths, one expects roughly $\frac{\Delta N^{2/3}}{4} $ of the steps to lead to decreasing the distance between $\lambda_{k}'$ and $\lambda_{k\pm 1}'$. Suppose that $|\lambda'_k(m) - \lambda_{k\pm 1}'(m)| = (2\kappa^{-1}/3)N^{1/3} \log(\mu N) - C^{\pm}\kappa^{-1} N^{1/3} + O(1)$, for $m \in [A,B]$; then the acceptance probability is roughly equal to 
$$p_N(t) = (1 - t^{ (2\kappa^{-1}/3)N^{1/3} \log \left(\mu N\right)   - C^{+}\kappa^{-1} N^{1/3}})^{\Delta N^{2/3}/4}  (1 - t^{ (2\kappa^{-1}/3)N^{1/3} \log\left(\mu N\right)    - C^{-}\kappa^{-1} N^{1/3}})^{\Delta N^{2/3}/4}.$$
Taking logarithms we see that $\log (p_N(t)) = -\frac{\Delta N^{2/3}}{4} (  e^{ -(2/3)  \log \left(\mu N\right)   + C^{+} } + e^{ -(2/3)\log\left(\mu N\right)  + C^{-}}) + O(N^{-2/3}).$ We thus see that $\lim_{N \rightarrow \infty} \log (p_N(t))  = - (\Delta/4) e^{ -(2/3)  \log \left(\mu \right) }( e^{ C^+} + e^{C^-})$.

On the other hand, the acceptance probability for $\mathcal{H}^{\tau, T}$ is given by $\exp(- (\Delta T^{2/3} \chi^{-1}/2)(e^{C^+}  + e^{C^-})$. Equality of the latter and $\lim_{N \rightarrow \infty} p_N(t)$ is equivalent to 
$$-(\Delta/4) e^{ -(2/3)  \log \left(\mu\right) }( e^{ C^+} + e^{C^-}) = - (\Delta T^{2/3} \chi^{-1}/2)(e^{C^+}  + e^{C^-})\iff e^{ -(2/3)  \log\left( \mu\right) }= 2T^{2/3} \chi^{-1}.$$
Substituting $\mu =\frac{ T^{-1}\chi^{3/2}}{2\sqrt{2}}$  one readily verifies that the latter equality holds. This shows that the discrete acceptance probability, at least heuristically, converges to the limiting one, verifying the ${\bf H}_1$-Brownian Gibbs property.\\

\bibliographystyle{amsplain}
\bibliography{PD}

\end{document}